\pdfoutput=1
\documentclass[10pt,reqno]{amsart}
\usepackage[margin=1in]{geometry}
\usepackage{bbm}
\usepackage{amsfonts}
\usepackage{latexsym, amssymb, amsmath, amscd, amsthm, amsxtra}
\usepackage{mathtools}
\usepackage{enumerate}
\usepackage[all]{xy}
\usepackage{mathrsfs}
\usepackage{fancyhdr}
\usepackage{listings}
\usepackage{hyperref}
\usepackage{cleveref}
\usepackage{soul}
\usepackage{xcolor}
\usepackage{dsfont}
\usepackage{stmaryrd}
\usepackage{bm}
\usepackage{comment}

\allowdisplaybreaks
\hypersetup{colorlinks=true}
\usepackage{graphicx, color}
\usepackage{cite} 

\numberwithin{equation}{section}
\numberwithin{figure}{section}

\DeclareMathOperator{\bbe}{\mathbb{E}}

\theoremstyle{plain} 
\newtheorem{theorem}{Theorem}[section]
\newtheorem*{theorem*}{Theorem}
\newtheorem{lemma}[theorem]{Lemma}
\newtheorem*{lemma*}{Lemma}

\newtheorem*{corollary*}{Corollary}

\newtheorem*{proposition*}{Proposition}
\newtheorem{definition}[theorem]{Definition}
\newtheorem*{definition*}{Definition}
\newtheorem{assumption}[theorem]{Assumption}
\newtheorem{claim}[theorem]{Claim}

\newtheorem*{conjecture*}{Conjecture}

\theoremstyle{definition} 

\newtheorem*{example*}{Example}
\newtheorem{remark}[theorem]{Remark}
\newtheorem*{remark*}{Remark}

\newcommand{\cal}{\mathcal}

\newcommand{\wh}{\widehat}
\newcommand{\wt}{\widetilde}
\renewcommand{\dag}{*}
\renewcommand{\dagger}{*}
\newcommand{\qqD}{\llbracket D\rrbracket}
\newcommand{\size}{\mathsf{size}}

\definecolor{darkred}{rgb}{0.9,0,0.3}
\definecolor{darkblue}{rgb}{0,0.3,0.9}

\renewcommand{\P}{\mathbb{P}}
\newcommand{\E}{\mathbb{E}}
\newcommand{\R}{\mathbb{R}}
\newcommand{\C}{\mathbb{C}}
\newcommand{\N}{\mathbb{N}}
\newcommand{\Z}{\mathbb{Z}}

\newcommand{\cG}{{\mathcal G}}
\newcommand{\cI}{\mathcal{I}}
\newcommand{\fa}{\mathfrak a}
\newcommand{\fb}{\mathfrak b}
\newcommand{\fc}{\mathfrak c}
\newcommand{\ft}{\mathfrak t}

\newcommand{\eqq}{\varepsilon_L}

\newcommand{\ii}{\mathrm{i}}
\newcommand{\dd}{\mathrm{d}}
\newcommand*{\deq}{\mathrel{\vcenter{\baselineskip0.65ex \lineskiplimit0pt \hbox{.}\hbox{.}}}=}

\renewcommand{\leq}{\leqslant}
\renewcommand{\geq}{\geqslant}
\renewcommand{\epsilon}{\varepsilon}

\newcommand{\qq}[1]{\llbracket{#1}\rrbracket}

\newcommand{\avg}[1]{\langle #1 \rangle}

\newcommand{\avgB}[1]{\Bigl\langle #1 \Bigr\rangle}

\DeclareMathOperator{\diag}{diag}
\DeclareMathOperator{\tr}{Tr}

\DeclareMathOperator{\re}{Re}
\DeclareMathOperator{\im}{Im}

\DeclareMathOperator{\dist}{dist}

\DeclareMathOperator{\spec}{spec}

\newcommand{\bg}{{\bf{g}}}
\newcommand{\bx}{{\bf{x}}}

\newcommand{\bu}{{\bf{u}}}
\newcommand{\bv}{{\bf{v}}}
\newcommand{\bw}{{\bf{w}}}

\newcommand{\al}{\alpha}

\newcommand{\be}{\begin{equation}}
\newcommand{\ee}{\end{equation}}

\newcommand{\e}{{\varepsilon}}

\newcommand{\VV}{H_{\Lambda}}
\newcommand{\HS}{{\mathrm{HS}}}
\newcommand{\oo}{\mathrm{o}}
\newcommand{\OO}{\mathrm{O}}

\newcommand{\LK}{\Delta}
\newcommand{\LKE}{\widehat\Delta}
\newcommand{\opr}[1]{\mathrm{O}_\prec({#1})}

\allowdisplaybreaks

\numberwithin{equation}{section}
\usepackage{environ}

\begin{document}
\title{A random matrix model towards the quantum chaos transition conjecture}


\author{Bertrand Stone}
\address{Department of Mathematics, University of California, Los Angeles, Los Angeles, CA, USA.}
\email{bertrand.stone@math.ucla.edu}

\author{Fan Yang}
\address{Yau Mathematical Sciences Center, Tsinghua University, and Beijing Institute of Mathematical Sciences and Applications, Beijing, China.}
\email{fyangmath@mail.tsinghua.edu.cn}

\author{Jun Yin}
\address{Department of Mathematics, University of California, Los Angeles, Los Angeles, CA, USA.}
\email{jyin@math.ucla.edu}

\maketitle

\begin{abstract}
Consider $D$ random systems that are modeled by independent $N\times N$ complex Hermitian Wigner matrices. Suppose they are lying on a circle and the neighboring systems interact with each other through a deterministic matrix $A$. We prove that in the asymptotic limit $N\to \infty$, the whole system exhibits a quantum chaos transition when the interaction strength $\|A\|_{\HS}$ varies. Specifically, when $\|A\|_{\HS}\ge N^{\e}$, we prove that the bulk eigenvalue statistics match those of a $DN\times DN$ GUE asymptotically and each bulk eigenvector is approximately equally distributed among the $D$ subsystems with probability $1-\oo(1)$. These phenomena indicate quantum chaos of the whole system. In contrast, when $\|A\|_{\HS}\le N^{-\e}$, we show that the system is integrable: the bulk eigenvalue statistics behave like $D$ independent copies of GUE statistics asymptotically and each bulk eigenvector is localized on only one subsystem. In particular, if we take $D\to \infty$ after the $N\to \infty$ limit, the bulk statistics converge to a Poisson point process under the $DN$ scaling.
\end{abstract}

{
\hypersetup{linkcolor=black}
\tableofcontents
}

\section{Introduction}

Consider a quantum Hamiltonian $H_{\lambda}=H +\lambda H_{I}$, where $H$ represents the Hamiltonian of some separable (or independent) subsystems in a specific coordinate system, and $H_{I}$ is a non-separable perturbation describing the interactions between these subsystems. In the study of quantum chaos, researchers aim to answer the following questions. (i) Does the behavior of the system become ``chaotic" when the interaction (measured by the parameter $\lambda>0$) becomes sufficiently strong? If the chaotic phase indeed exists, then (ii) what kind of features of energy levels, eigenstates, or some other physical observables can be used to distinguish between the integrable (or non-chaotic) and chaotic phases, and (iii) when the quantum chaos transition occurs as $\lambda$ varies. 

In the context of a generic integrable model, the well-known Berry-Tabor conjecture \cite{Berry-Tabor,Berry_1984,Jens-BT} states that the local energy level statistics should follow a Poisson process, i.e., neighboring energy levels should be uncorrelated. Conversely, for a generic chaotic model, the renowned Bohigas-Giannoni-Schmit (BGS) conjecture \cite{BGS-conjecture,BGS-conjecture2} suggests that its local energy level statistics should resemble those of Gaussian orthogonal ensemble (GOE). Specifically, it predicts the presence of level repulsion between neighboring energy levels, and the nearest neighbor spacing is expected to align with the ``Wigner surmise" predicted by \cite{Wigner} for Wigner matrices. The eigenfunctions of chaotic models are predicted to satisfy a fundamental property called \emph{quantum ergodicity} \cite{QE,QE2,QE3}, i.e., almost all eigenfunctions become equidistributed as the energy tends to $\infty$. Later, a stronger notion called \emph{quantum unique ergodicity} (QUE) states that the above equipartition principle holds for all eigenfunctions in the high-energy limit \cite{RudSar1994}. 
We refer the readers to \cite{BG_book1,BG_book2,BourgadeKeating2013} for a more comprehensive review of the above conjectures and the relation between quantum chaos and random matrix theory. There have been numerous studies on the Berry-Tabor conjecture \cite{Cheng1994,CL_PRA,Marklof_Annals,Marklof_Duke,VanderKam} and quantum unique ergodicity \cite{Lindenstrauss2006,Holowinsky2010}, to name just a few. However, to the best of our knowledge, the BGS conjecture remains unproven in the existing literature.

In the past few decades, there has been great progress in understanding the chaotic behaviors of some classical ensembles of random matrices, including the delocalization of eigenvectors, QUE, the GOE/GUE statistics of eigenvalues (or the Wigner surmise), and so on.
They are now considered a good paradigm for exploring many ideas related to quantum chaos and the BGS conjecture.  
The extensive body of literature explores these behaviors across various random matrix ensembles, such as Wigner matrices (see, for example, \cite{vector_flow,cipolloni-erdos2021,CES_QUE3,CIPOLLONI2022109394,ESY_local,AjaErdKru2015,CEJK2023,bourgade2016fixed,ErdPecRamSchYau2010,ErdSchYau2011,erdos2011universality,erdHos2012bulk,erdHos2012rigidity,TaoVu2011,ErdYau2012,LanYau2015,LANDON20191137,ESY1}),
random graphs (see, for example, \cite{erdHos2012spectral,erdHos2013spectral,BHY2019,BouHuaYau2017,BKH2017,AnaLeM2013,ADK2021,BHKY_2017_AOP,HLY_2015}), heavy-tailed random matrices (see, for example, \cite{ALY_Levy,ALM_Levy,Mobility_Levy}), and non-mean-field random matrices (see, for example, \cite{bourgade2017universality,bourgade2020random,yang2021delocalization,yang2022delocalization,xu2022bulk,YangYin2021,erdos2013delocalization}), to name just a few. For a historical review of seminal works that have led to the proof of the Wigner surmise and the delocalization of eigenvectors for random matrices, readers can refer to the book \cite{Erds2017ADA}. However, most of these models (except for non-mean-field random matrices) are chaotic in nature, so they generally do not exhibit integrable features and hence are not suitable for the investigation of the transition between integrable and chaotic phases. On the other hand, random band matrices are proposed for this purpose. The random band matrix ensemble with band width $W$ and system size $N$ is conjectured to exhibit a quantum chaos transition as $W$ varies (see, for example, \cite{PhysRevLett.67.2405,PhysRevLett.64.1851,PhysRevLett.64.5,PhysRevLett.66.986,PB_review,Spencer2}): there exists a critical band width $W_c=\sqrt{L}$ such that 
\begin{itemize}
    \item if $W \ll W_c$, $H$ has localized bulk eigenvectors, and the bulk eigenvalues satisfy Poisson statistics; 
    \item if $W \gg W_c$, $H$ has delocalized bulk eigenvectors satisfying QUE, and the bulk eigenvalues satisfy GOE/GUE statistics depending on the symmetry of the system.  
\end{itemize}
This conjecture remains unresolved; however, significant progress has been made in the past decade towards establishing its validity (see, for example, \cite{BaoErd2015,Goldstein_1/2,CS1_4,CPSS1_4,Hislop:2022td,Shcherbina:2021wz,erdos2013delocalization,bourgade2017universality,bourgade2020random,Sch2009,PelSchShaSod,Sch2014,Sch1,Sch3,Sch2,SchMT,EK_band1,ErdKno2011,YangYin2021}).

\subsection{Overview of the main results} 
In this paper, we consider a simpler random matrix model (than random band matrices) that also naturally exhibits the quantum chaos transition. More precisely, fix an integer $D\ge 2$, we take $D$ independent random subsystems, assuming that they can be modeled by $N\times N$ Wigner matrices, whose entries have mean zero, variance $N^{-1}$, and fulfill certain moment conditions. Without introducing interactions, the whole system is modeled by a block diagonal matrix $H$ with diagonal blocks being independent Wigner matrices $H_a$, $a=1,\ldots, D$. Next, we assume that the subsystems are aligned along a cycle and the neighboring subsystems interact with each other through an arbitrary deterministic $N\times N$ matrix $A$ (non-nearest-neighbor interactions can also be allowed; see \Cref{remark_more_general} below). 
We denote the interaction Hamiltonian by $\Lambda$, which is a tridiagonal block matrix with non-diagonal blocks being $A$ or $A^*$, and denote the whole system by $\VV$. In other words, $\VV$ is defined as 
\be\label{eq:def_model}
\VV=H+\Lambda,
\ee
where $H$ and $\Lambda$ are $D\times D$ block matrices defined as
\be\label{matrix_H}
H = \begin{pmatrix}
H_1 & 0 & 0 & \cdots & 0 & 0 \\
0 & H_2 & 0 & \cdots & 0 & 0 \\
0 & 0 & H_3 & \cdots & 0 & 0 \\
\vdots & \vdots & \vdots & \ddots & \vdots & \vdots \\
0 & 0 & 0 & \cdots & H_{D-1} & 0 \\
0 & 0 & 0 & \cdots & 0 & H_D
\end{pmatrix},\quad \Lambda = \begin{pmatrix}
0 &A & 0 & \cdots & 0 &A^\dagger \\
A^\dagger & 0 &A & \cdots & 0 & 0 \\
0 & A^\dagger & 0 & \cdots & 0 & 0 \\
\vdots & \vdots & \vdots & \ddots & \vdots & \vdots \\
0 & 0 & 0 & \cdots & 0 &A \\
A & 0 & 0 & \cdots & A^\dagger & 0
\end{pmatrix}.
\ee
Furthermore, we assume that $\VV$ is a perturbation of $H$, i.e., $\|A\|\ll \E\|H\|\sim 1$. 

In this paper, we prove that the bulk eigenvalues and eigenvectors of $\VV$ exhibit a transition from the integrable phase to the chaotic phase, and this transition occurs at $\|A\|_{\HS}\sim 1$. 
For simplicity, we introduce the index sets ${\cal I}_a:=\llbracket (a-1)N+1,aN\rrbracket$, $a \in \{ 1,\ldots,D\}$, for the subsystems and let $\cal I:=\llbracket DN\rrbracket$ be the index set for the whole system. Hereafter, for any $n,m\in \R$, we denote $\llbracket n, m\rrbracket: = [n,m]\cap \Z$ and $\llbracket n\rrbracket:=\llbracket1, n\rrbracket$. 
Roughly speaking, we prove the following results:
\begin{itemize}
    \item If $\|A\|_{\HS}\ll 1$, then every bulk eigenvector $\mathbf v$ is concentrated on only one subsystem, i.e., there exists an $\cal I_a$ such that $\sum_{k\in\mathcal{I}_a}|\bv(k)|^2 =1-\oo(1)$ with probability $1-\oo(1)$ (\Cref{nonmix}). The bulk eigenvalues of $\VV$ are negligible perturbations of those of $H$ (\Cref{NonMixEV}); in particular, if $D$ is large, the eigenvalue statistics around a bulk energy is almost a Poisson point process under the $DN$ scaling.

    \item If $\|A\|_{\HS}\gg 1$, then every bulk eigenvector $\mathbf v$ extends to the whole system (\Cref{mix}): on each $\cal I_a$,
    \begin{equation}\label{eq:QUE}
    \sum_{k\in\mathcal{I}_a}|\bv(k)|^2=D^{-1} +  \oo(1)  \quad \text{with probability } 1-\oo(1).
    \end{equation}
    The bulk eigenvalue statistics of $\VV$ match those of a $DN\times DN$ GOE/GUE asymptotically (\Cref{MixEV}). 
    For ease of presentation, we will, with a slight departure from strict notation, refer to \eqref{eq:QUE} as the ``QUE estimate" within the context of this paper.  It is important to note, however, that QUE typically denotes a stronger property; see the discussion below \Cref{mix}.  
\end{itemize}
We would like to emphasize that these results apply to any deterministic interaction $A$---no specific structure about $A$ will be used in the proof except for the operator norm bound $\|A\|\ll 1$ and the conditions on $\|A\|_{\HS}$.

As $N\to \infty$, when $D$ is fixed, the typical spectral distances between $D$ independent Wigner matrices are of order $N^{-1}$. Intuitively, quantum chaos arises when the interaction $\Lambda$ mixes the spectra of subsystems. In the proof, we will show that $\Lambda$ induces a shift in the eigenvalues of magnitude $N^{-1}\|A\|_{\HS}$, thereby rigorously confirming the aforementioned intuition: when $\|A\|_{\HS}\ll 1$, the spectra of subsystems remain separate, indicating integrability of the entire system; when $\|A\|_{\HS}\gg 1$, the subsystems become mixed, resulting in chaotic behavior.

\begin{figure}[htb]
\begin{center}
\includegraphics[width=12cm]{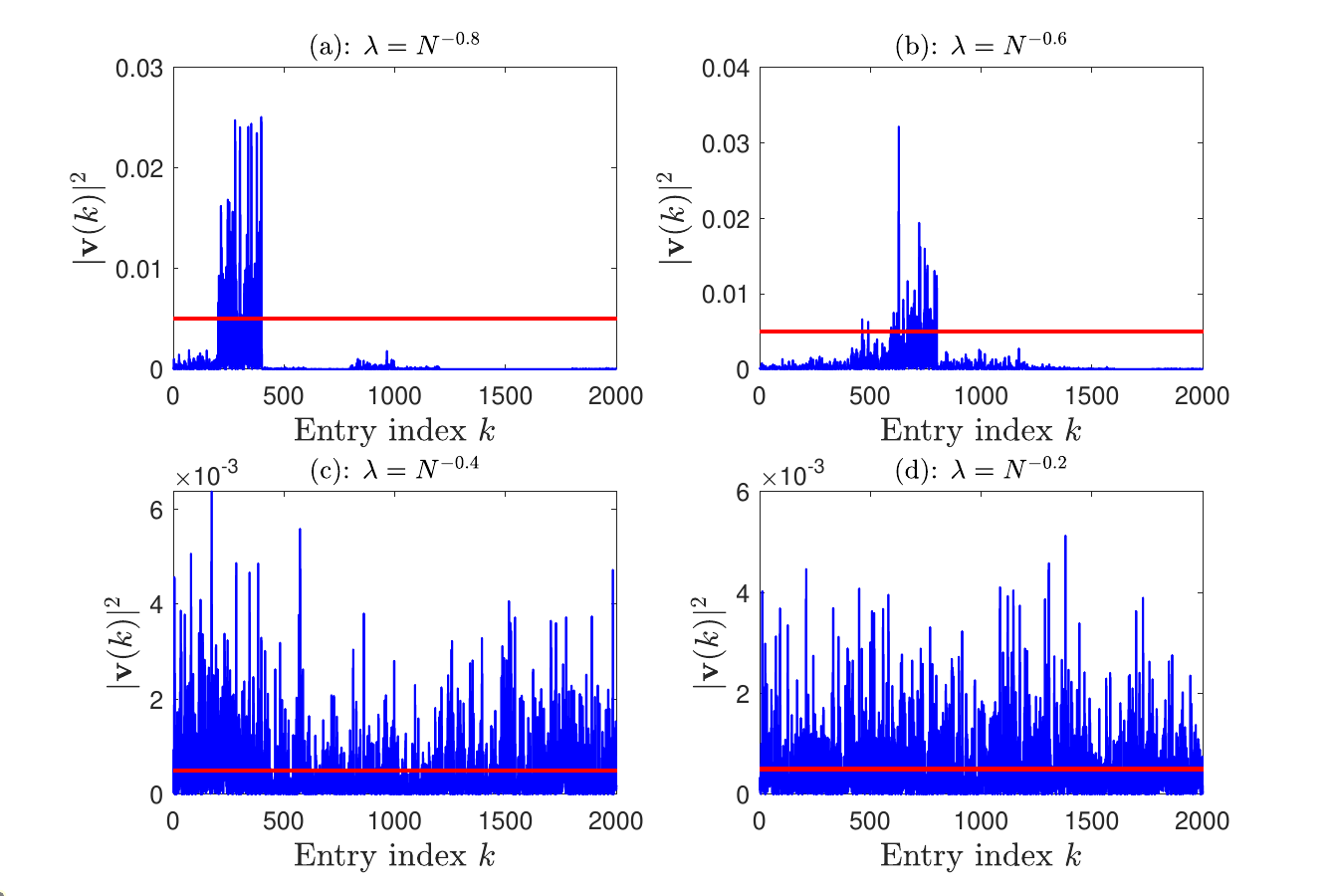}
\caption{The entries $|\bv(k)|^2$ for $k\in \cI$, where $\bv$ is the $(DN/2)$-th eigenvector of $\VV$. 
We choose $D=10$, $N=200$, and $A = \lambda I$. 
In (a) and (b), the red lines mark the value $N^{-1}$; in (c) and (d), the red lines mark the value $(DN)^{-1}$.}\label{simfigvec}
\end{center}
\end{figure}

We now perform some simulations to illustrate the theoretical results. We choose the $D$ independent $N\times N$ Wigner matrices to be GOE, and take $A = \lambda I$ (so $\|A\|_{\HS}=\lambda \sqrt{N}$ and the transition occurs at $\lambda=N^{-1/2}$). In \Cref{simfigvec}, we show the squared magnitudes of the entries of a bulk eigenvector. One can observe that in the subcritical cases, the eigenvector entries are localized in only one block, while in the supercritical cases, the eigenvector is delocalized. In \Cref{simfigvalue}, we depict the simulated distributions of the energy gaps between neighboring bulk eigenvalues under the scaling $DN$. In the subcritical cases, the simulated gap distribution exhibits a good resemblance to the exponential distribution, providing evidence that the behavior of bulk eigenvalues locally resembles that of a Poisson process. This finding substantiates the Berry-Tabor conjecture within our framework. Conversely, in the supercritical cases, the simulated gap distribution closely aligns with the Wigner surmise, indicating that the bulk eigenvalues follow the GOE statistics. 
This result confirms the BGS conjecture within our context.

\begin{figure}[htb]
\begin{center}
\includegraphics[width=12cm]{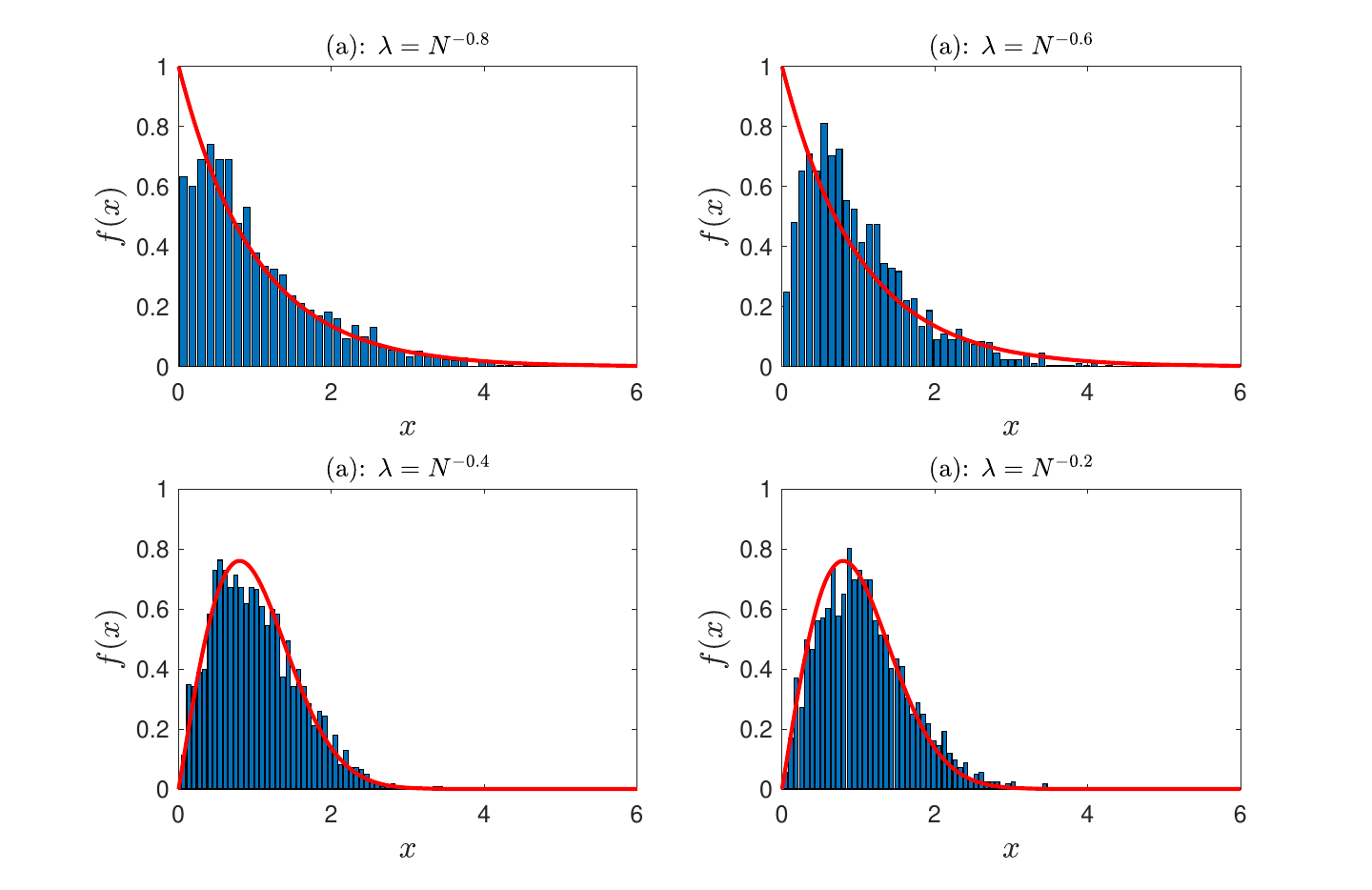}
\caption{
The bulk eigenvalue gap distributions under the same setting as \Cref{simfigvec}. The normalized histograms plot the rescaled eigenvalue gaps $DN\rho_{sc}(\lambda_k)(\lambda_{k+1}-\lambda_k)$, where $\lambda_k$ are the eigenvalues of $\VV$ with $k\in \qq{DN/20,19DN/20}$ and $\rho_{sc}$ is the semicircle density given by \eqref{eq:semidensity} below. In (a) and (b), the red curves plot the probability density function for the exponential distribution: $f(x)=e^{-x}$. In (c) and (d), the red curves plot the Wigner surmise: $f(x)=\frac{\pi x}{2}e^{-\pi x^2/4}$. Note that repulsions between eigenvalues already appear to be present in the $\lambda = N^{-0.6}$ case. We attribute this phenomenon to the finite-$N$ effect.}
\label{simfigvalue}
\end{center}
\end{figure}

\begin{remark}\label{remark_more_general}
There are some direct extensions of our results to various other models. Our method relies primarily on the block diagonal structure of matrix $H$ and the block translation invariance of the interactions. Hence, the proofs in our paper can be readily extended to more general cases encompassing non-nearest-neighbor interactions (including interactions within each block), as long as the interaction between the $i$-th and $j$-th blocks depends only on $j-i$ (by utilizing the periodic distance on the cycle). 
We can further relax the condition and only require that the interactions have the same distribution. For instance, the interactions can consist of independent lower triangular matrices with i.i.d.~entries, in which case we get random band matrices. Alternatively, the interactions can be independent Ginibre matrices with i.i.d.~entries, leading to the Wegner orbital model \cite{PelSchShaSod}. In addition, the random matrix characterizing each subsystem can be more intricate models, such as sparse random matrices, generalized Wigner matrices, or deformed Wigner matrices. 
Finally, our results can be easily extended to higher-dimensional models, where the subsystems reside on a $d$-dimensional lattice with $d\ge 2$. In this setting, all our arguments remain applicable as long as the interactions between different blocks maintain translation symmetry on the lattice.

\end{remark}

\begin{remark}
At the critical point where $\|A\|_{\HS}\sim 1$, simulations show that each block contains a substantial portion of a bulk eigenvector, although certain blocks contain significantly larger amounts (measured in terms of the $\ell^2$-norm) than others. We conjecture that the $\ell^2$-mass distribution across the $D$ blocks would follow a polynomial decay. Additionally, we anticipate that the bulk eigenvalue statistics would exhibit an interpolation between the GOE/GUE statistics and $D$ independent copies of GOE/GUE statistics. The investigation of this critical scenario is deferred for future research.
\end{remark}

\subsection{Physical background} 

One of the main physical motivations behind this model is to gain insights into the Anderson metal-insulator transition of the well-known Anderson model \cite{Anderson}. In the Anderson model, the random potential (denoted by $V$) consists of i.i.d.~diagonal entries. However, in our model (and its higher-dimensional counterparts, as mentioned in \Cref{remark_more_general}), we have replaced $V$ with a block diagonal random matrix $H$ and generalized the Laplacian matrix to a more general matrix $\Lambda$ that introduces interactions between adjacent blocks. From a physical perspective, our ``block potential" is related to the Anderson potential through a scaling transformation. Regarding $\Lambda$, choosing $A=I$ yields the block Anderson model \cite{PelSchShaSod}, while incorporating random independent interactions leads to random band matrices or the Wegner orbital model, as mentioned in \Cref{remark_more_general} above. These models represent important extensions of the Anderson model and all exhibit the desired metal-insulator transition as the strength of interactions varies. 
In particular, the Wegner orbital models draw inspiration from the work of Wegner \cite{Wegner1}, which was further developed in \cite{Wegner2,Wegner3}. They are proposed to model the motion of a quantum particle with many internal degrees of freedom in a disordered medium.  
We expect that investigating the block model $\VV$ could provide an alternative approach to the renowned Anderson metal-insulator transition conjecture. We will delve into this aspect in an ongoing paper \cite{Block_Anderson}, where we study the localization-delocalization transition of the block Anderson and Wegner orbital models as the interaction strength changes. It is worth highlighting that the framework presented in that paper poses greater challenges, as it involves examining ``more local" interactions where $D\to \infty$ as $N\to \infty$. 

Next, we explore another physical connection, aiming to gain some \emph{heuristic insights} into the mechanism of quantum chaos transition through the lens of random matrix theory. To accomplish this, we begin by considering a simplified toy model. 
Consider the two-dimensional (2D) quantum billiard inside the unit disk $\mathcal D$:
$$ -\Delta \psi(r,\theta) = E\psi(r,\theta),\quad \psi|_{\partial \mathcal D}=0.$$
This is an integrable model with eigenfunctions $ \psi_{n,m}(r,\theta)=J_{|m|}(k_{m,n}r)e^{\ii m\theta}$ and eigenvalues $E=k_{m,n}^2$, where $m\in \Z,$ $ n\in \N, $ $J_{|m|}$ is the $|m|$-th Bessel function, and $k_{m,n}$ is the $n$-th zero of $J_{|m|}$. We then add a potential $V(r,\theta)$ to the Hamiltonian: $H=-\Delta + V(r,\theta)$. Note when $V$ tends to $\infty\cdot \mathbf 1_{\cal B}$ for a region $\cal B \subset \cal D$, we get a quantum billiard inside the region between $\partial \cal D$ and $\partial \cal B$. 
Our goal is to understand when quantum chaos arises as the potential $V$ varies. In particular, we are interested in the asymptotic behavior of the eigenvalues and corresponding eigenstates around an energy $E$ in the high-energy limit $E\to \infty$.

In the original quantum billiard in $\cal D$, there is a conserved physical quantity, namely the \emph{angular momentum $-\ii \partial_{\theta}$}. As a result, the operator $-\Delta$ can be decomposed according to its invariant subspaces characterized by the quantum numbers for the angular momentum:
$$ V_m :=\left\{ \psi(r,\theta):\psi(r,\theta)=f(r)e^{\ii m \theta} \text{ for some radial function } f(r)\right\},$$
where $-\Delta$ restricted to $V_m$ is given by  
\be\label{eq:Deltam}
-\Delta \psi = - \Delta^{(m)}\psi,\quad \text{with}\quad \psi(r,\theta)=f(r)e^{\ii m \theta}, \quad \Delta^{(m)}:= \frac{\partial^2}{\partial r^2} +\frac{1}{r}\frac{\partial }{\partial r} - \frac{m^2}{r^2}.
\ee
This decomposition naturally leads to a block structure of $H$ with blocks labeled by the quantum numbers for the angular momentum. More precisely, let \smash{$\{f_n(r)\}$} be an orthonormal basis of the space of radial functions with $f_n(1)=0$. Then, the $(m,m')$-block of $H$ can be represented by an $\infty$-dimensional matrix with entries
\be\label{eq:blocksH} 
H^{(m',m)}_{n',n} 
= \frac{1}{2\pi} \iint \bar f_{n'} (r)\left[-\Delta^{(m)}+V(r,\theta)\right]f_n (r) e^{\ii (m-m')\theta}\dd r \dd \theta.
\ee
By performing a finite-dimensional reduction, we can express $H$ as a block matrix. 

\begin{figure}[htb]
\begin{center}
\includegraphics[width=7.5cm]{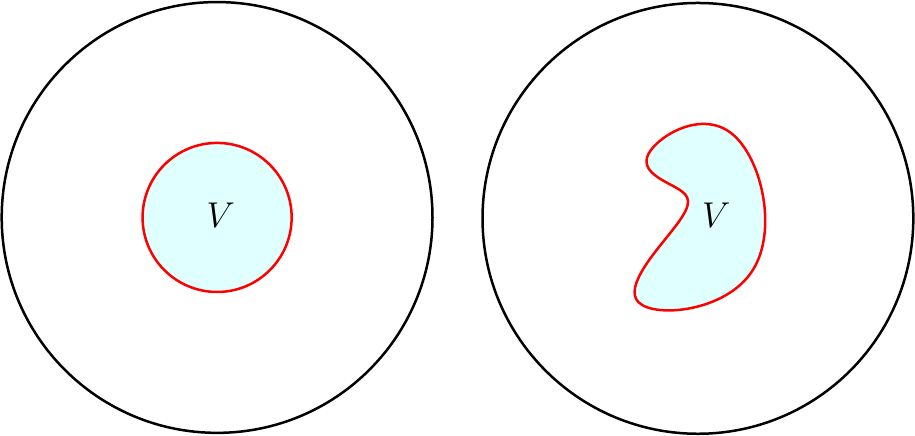}
\end{center}
\caption{A 2D quantum billiard inside a circle with a potential $V$ supported in a compact region. The left picture depicts an integrable system featuring a radially symmetric potential, which does not mix eigenstates with different angular momentum. In the right picture, the rotational symmetry of $V$ is broken, resulting in mixing of eigenstates with different angular momentum, which can induce chaotic behavior. Note when the potential $V\to \infty$, essentially forming an inner boundary, the system becomes a quantum billiard between two boundaries.}\label{figure:QC}
\end{figure}

The block matrix obtained above is a deterministic matrix, whose behavior seems to be inherently complex to analyze. Intuitively, one would expect that, for a general $V$, $H$ exhibits characteristics akin to a ``quasi-random" matrix, with entries possessing intricate magnitudes and phases. 
To facilitate a more accessible study of the quantum chaos transition, we propose employing a block random matrix \smash{$\VV$} as a computationally feasible ``toy model". This model allows for easier investigation while still capturing essential qualitative features that establish a connection between the aforementioned quantum billiard model and the model \eqref{eq:def_model}.

\begin{itemize}

\item[(i)] Due to the fast-oscillating nature of the functions $f_n$ when $n$ is large, the entries within the diagonal blocks \emph{$H^{(m,m)}$} exhibit complicated phases. Furthermore, neighboring diagonal blocks demonstrate slightly different spectra (due to the term $-m^2/r^2$ in \eqref{eq:Deltam}). Following Wigner's insight \cite{Wigner}, we propose employing independent Wigner matrices to model these blocks---random matrices serve as effective phenomenological models for complex ``quasi-random" matrices, and independent Wigner matrices have distinct yet closely aligned spectra.

\item[(ii)] When the potential $V(r,\theta)$ is zero or exhibits rotational symmetry, as depicted in the left graph of \Cref{figure:QC}, the eigenstates with different angular momentum remain unmixed, resulting in an integrable model. This scenario corresponds to the $A=0$ case in our model \eqref{eq:def_model}. 

\item[(iii)] When we modify the potential $V$ as shown in the right graph of \Cref{figure:QC}, off-diagonal block entries \smash{$H^{(m',m)}_{n',n} = (2\pi)^{-1} \iint \bar f_{n'} (r) V(r,\theta)f_n (r) e^{\ii (m-m')\theta}\dd r \dd \theta$} with $m'\neq m$ will emerge and they reflect the strength of $V$ and the departure of $V$ from rotational symmetry. From \eqref{eq:blocksH}, we first notice that $H$ exhibits a block translation symmetry, meaning that \smash{$H^{(m',m)}$} depends on $m'$ and $m$ only through their difference $m'-m$. This aligns with the block translation invariance of $\Lambda$ in \eqref{matrix_H}. 
Second, the magnitudes of the off-diagonal block entries \smash{$H^{(m',m)}_{n',n}$} generally decrease as the difference $|m'-m|$ increases, given fixed $n$ and $n'$. Consequently, we have chosen $\Lambda$ in \eqref{matrix_H} to be nearest-neighbor interactions to model such ``short-range phenomenon". (However, this simplification is not crucial since, as mentioned in \Cref{remark_more_general}, we do allow for the inclusion of non-nearest-neighbor interactions as well.) 
\end{itemize}
Physically, item (ii) above means that the introduction of a circular hole at the center still gives a block diagonal matrix, so our results in the integrable phase imply \emph{heuristically} the integrability of the quantum billiard model. On the other hand, in this paper, we provide evidence of a quantum chaos transition for $\VV$ as the interaction strength increases. In particular, as $V\to \infty$, our results suggest, in a heuristic sense, that adding a ``generic" non-rotationally symmetric hole will lead to chaos. 

We expect that our paper may serve as an initial step towards heuristically understanding the quantum chaos transition from the perspective of random matrix theory. Significant efforts are still needed to attain a more comprehensive understanding and navigate the vast realm awaiting exploration. One important extension is explained in the following remark.

\begin{remark}\label{rem_largeD}
In the model \eqref{eq:def_model}, the quantities $D$ and $N$ correspond to the number of quantum numbers $m$ for the angular momentum and the number of basis functions $f_n$ retained in the finite-dimensional reduction, respectively. On the other hand, given a target energy $E$, the matrix dimension $DN$ represents the number of energy levels falling within an energy window around $E$, say $[0,2E]$. 
When $|m|$ or $n$ is large, utilizing the asymptotics of the zeros of Bessel functions, we see that the eigenvalues of $-\Delta$ scale as \smash{$k_{m,n}^2 \sim m^2+n^2$}. Consequently, it becomes apparent that the most relevant regime for investigating the quantum chaos transition should be \smash{$D\sim N \sim \sqrt{E}$}. Note that this regime corresponds to a random band matrix model with a \emph{critical band width}, which is a highly challenging problem as mentioned earlier. As a simplification, this paper focuses on a scenario where $D$ remains fixed as $N\to \infty$, i.e., we only examine the mixing of a finite number of blocks with similar angular momentum. It is possible to relax the assumption to a certain extent with $1\ll D\ll N$, but we defer such exploration to \cite{Block_Anderson} and more future studies.
\end{remark}

The rest of this paper is organized as follows. In Section \ref{sec:res}, we define our model and state the main results, which consist of two parts: in the chaotic regime with $\|A\|\ge N^\e$, we state the QUE of bulk eigenvectors (\Cref{mix}) and the GUE statistics for the bulk eigenvalues (\Cref{MixEV}); in the integrable regime, we state the localization of bulk eigenvectors (\Cref{nonmix}) and the Poisson statistics for the bulk eigenvalues (\Cref{NonMixEV}). These theorems (Theorems \ref{mix}--\ref{NonMixEV}) are proved in \Cref{sec:proof_mix}, \Cref{sec:mix_evalue}, \Cref{sec:localization}, and \Cref{sec_nonmix_evalue}, respectively. 
For the proof of \Cref{mix}, we reduce its proof to that of a Gaussian divisible case with a small Gaussian component in \Cref{sec:flow_prelim}, and then prove the Gaussian divisible case using a characteristic flow argument in \Cref{sec:flow}. Finally, the proofs of several estimates used in the main proofs are included in Appendix \ref{appd_determ}--\ref{sec:Gprop}, including some basic deterministic estimates, a Green's function comparison estimate, and the local laws for the Green's function of $\VV$.

\medskip

\noindent{\bf Notations.} To facilitate the presentation, we introduce some necessary notations that will be used throughout this paper. In this paper, we are interested in the asymptotic regime with $N\to \infty$. When we refer to a constant, it will not depend on $N$. Unless otherwise noted, we will use $C$ to denote generic large positive constants, whose values may change from line to line. Similarly, we will use $\epsilon$, $\delta$, $\tau$, $c$ etc.~to denote generic small positive constants. 
For any two (possibly complex) sequences $a_N$ and $b_N$ depending on $N$, $a_N = \OO(b_N)$ or $a_N \lesssim b_N$ means that $|a_N| \le C|b_N|$ for a constant $C>0$, whereas $a_N=\oo(b_N)$ or $|a_N|\ll |b_N|$ means that $\lim_{N\to \infty}|a_N| /|b_N| \to 0$. We say that $a_N\sim b_N$ if $a_N = \OO(b_N)$ and $b_N = \OO(a_N)$. For any $a,b\in\R$, we denote 
$a\vee b:=\max\{a, b\}$ and $a\wedge b:=\min\{a, b\}$. 
For an event $\Xi$, we let $\mathbf 1_\Xi$ or $\mathbf 1(\Xi)$ denote its indicator function. 
Given a vector $\mathbf v$, $\|\mathbf v\|\equiv \|\mathbf v\|_2$ denotes the Euclidean norm and $\|\mathbf v\|_p$ denotes the $\ell_p$-norm. 
Throughout this paper, we use ``$*$" to denote the Hermitian conjugate of a matrix.
Given a matrix $B = (B_{ij})$, we use $\|B\|$, $\|B\|_{\HS}$, and $\|B\|_{\max}:=\max_{i,j}|B_{ij}|$ to denote the operator, Hilbert-Schmidt, and maximum norms, respectively. We also adopt the notion of generalized entries: $B_\mathbf{uv}\equiv \mathbf u^* B \mathbf v$ for vectors $\bu,\bv$.

\subsection*{Acknowledgement}

Fan Yang is supported in part by the National Key R\&D Program of China (No. 2023YFA1010400). 
Jun Yin is supported in part by the Simons Fellows in Mathematics Award 85515. The authors would like to thank L{\'a}szl{\'o} Erd{\H{o}}s for helpful comments on an earlier draft of this paper. We are very grateful to two anonymous referees for the helpful comments, which have resulted in a significant improvement of the paper.

\section{Main results}\label{sec:res}

\subsection{The model and main results}

In this paper, we consider a block random matrix model. Fix any integer $D\geq 2$, let $H_1,H_2,\ldots, H_D$ be $D$ independent copies of $N\times N$ Wigner matrices, i.e., the entries of $H_a$ are independent (up to symmetry $H=H^*$) random variables satisfying that 
\be\label{eq:meanvar}
\E(H_{a})_{ij}=0,\quad \E |(H_{a})_{ij}|^2 = {N}^{-1},\quad a \in \qqD,\ \ i,j \in \qq{N}.
\ee
For the definiteness of notations, in this paper, we consider the complex Hermitian case, while the real case can be proved in the same way with some minor changes in notations. In the complex case, we assume additionally that 
\be\label{eq:meanvar2}
\E [(H_{a})^2_{ij}] = 0,\quad a \in \qqD,\ \ i\ne j \in \qq{N}.  
\ee
We assume that the diagonal entries are i.i.d.~real random variables and the entries above the diagonal are i.i.d.~complex random variables. Let $A$ be an arbitrary $N\times N$ (real or complex) deterministic matrix. Then, we consider the block random matrix model $\VV$ defined in \eqref{eq:def_model} with $H$ and $\Lambda$ given in \eqref{matrix_H}. 
Note that the distribution of the model $\VV$ exhibits block translation symmetry. As we consider $\VV$ as a perturbation of $H$, we always assume in this paper that $\|A\|\ll \E\|H\| \sim 1$. Under this assumption, the empirical measure of $\VV$ follows approximately the semicircle law with support $[-2,2]$. 
(We believe that the results in the chaotic regime, Theorems \ref{mix} and \ref{MixEV}, remain valid even for $\|A\|\gtrsim 1$ provided that $A$ has a reasonably large rank. However, they may fail if, for example, $A$ contains only one nonzero entry.) We now summarize the assumptions for the model \eqref{eq:def_model}.

\begin{assumption}\label{main_assm}
Fix any integer $D\ge 2$, we consider the model \eqref{eq:def_model}, where $A$ is an arbitrary $N\times N$ deterministic matrix with $\|A\|\le N^{-\delta_A}$ for a constant $\delta_A>0$, and $H_1,H_2,\ldots, H_D$ are $D$ i.i.d.~$N\times N$ complex Hermitian Wigner matrices satisfying \eqref{eq:meanvar}, \eqref{eq:meanvar2}, and the following high moment condition: for any $p\in \N$, there exists a constant $C_p>0$ such that 
\be\label{eq:highmoment}
\bbe |H_{11}|^p  + \bbe |H_{12}|^p \leq C_pN^{-p/2}. 
\ee
\end{assumption}

Let the eigenvalues of $\VV$ be $\lambda_1\le \lambda_2\le \cdots \le \lambda_{DN}$ and denote the corresponding (unit) eigenvectors by ${\bf v}_1,\bv_2,\ldots, {\bf v}_{DN}$. Now, we are ready to state the first two main results concerning the bulk eigenvectors and eigenvalues in the quantum chaotic phase when $\|A\|_{\HS}\gg 1$. Our first theorem states that with probability $1-{\rm o}(1)$, each bulk eigenvector is asymptotically uniformly distributed among all ${\cal I}_a$, $a\in \llbracket D\rrbracket$.

\begin{theorem}[Chaotic regime: eigenvectors]\label{mix} 
Under \Cref{main_assm}, let $\kappa \in (0,1/2)$ be an arbitrary constant and suppose there exists a constant $\varepsilon_A>0$ such that
\begin{equation}
\label{eq:condA1}
\|A\|_{\HS}\ge N^{\varepsilon_A} .
\end{equation}
Then, for any $k\in  \llbracket  \kappa DN,(1-\kappa)DN  \rrbracket$, there exists a constant $c>0$ such that 
\begin{equation}
\label{eq:main_evector1}
\mathbb P\left( \max_{a\in\llbracket D\rrbracket}\left|{\bf v}_k^* E_{a}  {\bf v}_k-D^{-1}\right|\ge N^{-c}\right)\le N^{-c},
\end{equation}
where $E_a\in \C^{DN\times DN}$ denotes the block identity matrix restricted to $\cal I_a$, i.e., $(E_{a})_{ij}={\bf 1}(i=j\in {\cal I}_a)$.
\end{theorem}

In the proof, we will prove a slightly stronger result than \eqref{eq:main_evector1}: 
\begin{equation}\label{eq:extend:main_evector1}
\mathbb P\left(\max _{i,j\in \qq{k-N^c,k+N^c}}\max_{a\in \qqD}
\left| {\bf v}_i^* (E_a-D^{-1}){\bf v}_j \right|\ge N^{-c}\right) \le  N^{-c}.
\end{equation}
As mentioned below \eqref{eq:QUE}, we will refer to \eqref{eq:main_evector1} and \eqref{eq:extend:main_evector1} as the ``QUE estimates" of bulk eigenvectors. However, QUE usually represents a stronger property: with probability $1-\oo(1)$, the entries of $\bv_k$ are approximately equally distributed on \emph{any subset} $I\subset \cI$ with $|I|\gg 1$. Such probabilistic QUE was first proved for Wigner matrices \cite{vector_flow}; later, a stronger notion called the \emph{eigenstate thermalization hypothesis} (ETH) was also established for Wigner matrices \cite{cipolloni-erdos2021,CES_QUE3,CIPOLLONI2022109394,CEH2023}. The QUE (or ETH) of eigenvectors has been extended to many other types of mean-field random matrices and random graphs (such as \cite{ALM_Levy,BHY2019,LP2021,BouHuaYau2017,Marcinek_thesis,CEJK2023,adhikari2023eigenstate}, to name just a few) and non-mean-field random matrices \cite{bourgade2017universality,bourgade2020random,xu2022bulk}. In this paper, we refrain from proving the stronger QUE for brevity, and the current form is already sufficient for our proof of the following result, \Cref{MixEV}. We leave the exploration of optimal QUE or ETH to future works.

Our second theorem shows that the local bulk eigenvalue statistics of $\VV$ match those of GUE asymptotically under the condition \eqref{eq:condA1}, that is, $\VV$ satisfies the bulk universality. Let $p_{\VV}(\lambda_1,\ldots, \lambda_{DN})$ denote the joint symmetrized probability density of the eigenvalues of $\VV$. For any $1\le n \le DN$, define the $n$-point correlation function by
$$
p_{\VV}^{(n)}\left(\lambda_1, \cdots, \lambda_n\right)
:=\int_{\R^{DN-n}} p_{\VV}\left(\lambda_1, \cdots, \lambda_{DN}\right) \mathrm{d} \lambda_{n+1} \cdots \mathrm{d} \lambda_{DN}.
$$
Denote the corresponding $n$-point correlation function for $DN\times DN$ GUE by $p_{G U E}^{(n)}$.

\begin{theorem}[Chaotic regime: eigenvalues]\label{MixEV}
In the setting of \Cref{mix}, let $O \in C_c^{\infty}\left(\mathbb{R}^n\right)$ be an arbitrary smooth, compactly supported function. Then, for any $|E|\le 2-\kappa$ and fixed $n\in \N$, there exists a constant $c>0$ so that 
$$
 \left|	\int_{\mathbb{R}^n} \mathrm{~d} \boldsymbol{\alpha}\; O(\boldsymbol{\alpha}) \left[p_{\VV}^{(n)}-p_{GUE}^{(n)}\right]\left(E+\frac{\alpha_1}{DN  }, \ldots, E+\frac{\alpha_n}{DN  }\right)   \right|\leq N^{-c}, 
\quad \boldsymbol{\alpha}=\left(\alpha_1, \ldots, \alpha_n\right) .
$$
\end{theorem}

The next two main results concern the bulk eigenvectors and eigenvalues of $\VV$ in the integrable (or non-chaotic) phase when $\|A\|_{\HS}\ll 1$. First, \Cref{nonmix} shows that in contrast to \Cref{mix}, each bulk eigenvector is mostly localized in only one $\cI_a$ in the integrable phase.

\begin{theorem}[Integrable regime: eigenvectors] \label{nonmix}
Under \Cref{main_assm}, let $\kappa\in (0,1/2)$ be an arbitrary constant and suppose there exists a constant $\varepsilon_A$ such that
\begin{equation}
\label{eq:condA2}
\|A\|_{\HS}\le N^{-\varepsilon_A} . 
\end{equation}
Then, for any $k\in \llbracket \kappa DN,(1-\kappa)D N\rrbracket$, there exists a constant $c>0$ such that
\begin{equation}
\label{eq:main_evector2}
\mathbb P\left(\max_{a=1}^D \|E_a{\bf v}_k\|^2 \le 1 - N^{-c}\right)\le N^{-c}. 
\end{equation}
\end{theorem}

Next, we show that under the condition \eqref{eq:condA2}, the bulk eigenvalues of $\VV$ are negligible perturbations of those of $H$ under the scaling $N$. As a result, the bulk eigenvalue statistics of $\VV$ and $H$ match asymptotically. Denote the eigenvalues of $H$ as $\lambda_1(H)\le \ldots \le \lambda_{DN}(H)$, and for any $1\le n\le N$, let $p_H^{(n)} $ represent the $n$-point correlation function of them.

\begin{theorem}[Integrable regime: eigenvalues]\label{NonMixEV}
Under the setting of \Cref{nonmix}, for any $k\in \llbracket \kappa DN,(1-\kappa)D N\rrbracket$, there exists a constant $c>0$ such that 
\be\label{eq:main_perurb}
\P\left( \left|\lambda_k-\lambda_k(H)\right| \ge N^{-1-c}\right)\le N^{-c}.
\ee
As a consequence, it implies that for any $|E|\le 2-\kappa$, fixed $n\in \N$, and smooth, compactly supported test function $O \in C_c^{\infty}\left(\mathbb{R}^n\right)$,  there exists a constant $c>0$ so that
$$ \left|	\int_{\mathbb{R}^n} \mathrm{~d} \boldsymbol{\alpha}\; O(\boldsymbol{\alpha}) \left[p_{\VV}^{(n)}-p_{H}^{(n)}\right]\left(E+\frac{\alpha_1}{DN  }, \ldots, E+\frac{\alpha_n}{DN  }\right)   \right|\leq N^{-c}, 
\quad \boldsymbol{\alpha}=\left(\alpha_1, \ldots, \alpha_n\right) . $$
\end{theorem}

The spectrum of $H$ consists of the eigenvalues of $D$ independent copies of Wigner matrices $H_1,\ldots, H_D$. Recall that it has been shown in a sequence of works (see e.g., \cite{bourgade2016fixed,ErdPecRamSchYau2010,ErdSchYau2011,erdos2011universality,erdHos2012bulk,erdHos2012rigidity,TaoVu2011}) that the eigenvalue statistics of each $H_a$, $a\in \qqD$, around a bulk energy $E\in [-2+\kappa,2-\kappa]$ and under a scaling $N$ converge to a determinantal point process with sine-kernel. Hence, the eigenvalue statistics of $\VV$ around $E$ and under a scaling $DN$ will converge to $D$ independent copies of sine processes as $N\to \infty$. In particular, if we consider the limit $N\to \infty$ followed by $D\to \infty$, the bulk eigenvalue statistics ${DN(\lambda_k-E)}$ is expected to converge to a Poisson point process whose density is determined by the limiting density $\rho_N$ defined in \eqref{eq:density_rhoN} below. 

\begin{remark} 
We expect that for $k\in \qq{1,\kappa DN}$, the transition threshold from the integrable phase to the chaotic phase should be 
\be\label{eq:thres}\|A\|_{\HS}\sim N^{1/3}/k^{1/3}.\ee 
In fact, the proofs for the bulk case in this paper suggest that the interaction $\Lambda$ induces an eigenvalue perturbation of order $N^{-1}\|A\|_{\HS}$. We expect chaos to enter when this perturbation can mix neighboring energy levels. On the other hand, the typical energy gap between the $k$-th and $(k+1)$-th eigenvalues of $H$ near the spectral edge is known to be of order $N^{-2/3}k^{-1/3}$. Thus, the transition should occur when 
$$ N^{-1}\|A\|_{\HS} \sim N^{-2/3}k^{-1/3},$$
which leads to the threshold \eqref{eq:thres}. 
We leave the detailed study of the edge case for future research. 
\end{remark}

To summarize, Theorems \ref{mix} and \ref{MixEV} together indicate the chaotic behavior of the whole system under \eqref{eq:condA1}, while Theorems \ref{nonmix} and \ref{NonMixEV} indicate the non-chaotic/integrable behavior of the whole system under \eqref{eq:condA2}.
Hence, we have established the quantum chaos transition of the random matrix model \eqref{eq:def_model} as $\|A\|_{\HS}$ crosses the threshold 1.  

\subsection{Local law}

One basic tool for our proof is the local law for the Green's function (or resolvent) of $\VV$,
\be\label{def_resolv}
G(z)\equiv G(z,H,\Lambda):=(\VV-z)^{-1},\quad z\in \mathbb C_+:=\{z\in \C:\im z>0\},
\ee
as we will state in \Cref{lem_loc} below. Note the model \eqref{eq:def_model} can be regarded as a deformed generalized Wigner matrix. Many methods have been developed in the literature (see e.g., \cite{LeeSchSteYau2015,knowles2017anisotropic,He2018,AEK_PTRF,EKS_Forum}) to show that in the $N\to \infty$ limit, $G(z)$ converges to a deterministic matrix $M(z)$ in the sense of local laws (see \Cref{lem_loc}). Moreover, $M(z)\equiv M(z,\Lambda)$ satisfies the \emph{matrix Dyson equation}: 
\be\label{def_M}
\left( \cal S(M) +z-\Lambda\right)M+I=0, 
\ee
where $\cal S(\cdot)$ is a linear operator acting on $M$ such that $\cal S(M)$ is a diagonal matrix with entries
$$\cal S(M)_{ij}=\mathbf 1(i=j)\sum_x s_{ix}M_{xx} =\mathbf 1(i=j)D\avg{ME_a},\quad i,j\in \cal I_a. $$
Hereafter, we denote the variances of the entries of $H$ by 
\be\label{eq:sij}
s_{ij}=\E|H_{ij}|^2= N^{-1}\mathbf 1(i,j\in \cal I_a \  \text{for some } a \in \llbracket D\rrbracket),
\ee
and let $S=(s_{ij}:i,j\in \cal I)$ be the variance matrix. In addition, we use $\langle B\rangle := (DN)^{-1}\tr B$ to denote the normalized trace of a $DN\times DN$ matrix $B$. 
Due to the block translation symmetry of $S$ and $\Lambda$, we see that $M$ is also block translation invariant, which implies that $\cal S(M)$ should be a scalar matrix 
$\cal S(M)=mI$, where $m(z)$ is defined as $m(z):=\langle M(z)\rangle$. 

\begin{remark}
When $D=2$, the block translation symmetry may not hold. In this case, we denote
$$M=\begin{pmatrix}
M_{(11)} & M_{(12)} \\
M_{(21)} & M_{(22)}
\end{pmatrix}.
$$
Then, we can derive directly from equation \eqref{def_M} that 
\be\label{eq:MD=2}
\begin{aligned}
    M_{(11)}=\frac{m+z}{AA^*-(m+z)^2},\quad &   M_{(22)}=\frac{m+z}{A^*A-(m+z)^2},\\  M_{(12)}=\frac{1}{AA^*-(m+z)^2}A,\quad & M_{(21)}=\frac{1}{A^*A-(m+z)^2}A^*,
\end{aligned} 
\ee
where $m(z)$ satisfies the self-consistent equation 
$m(z)=N^{-1}\tr M_{(11)}(z)=N^{-1}\tr M_{(22)}(z).$
\end{remark}

 \begin{definition}[Matrix limit of $G$]\label{defn_Mm}
 We define $m(z)\equiv m_N(z)$ as the unique solution to 
 \be\label{self_m}
m(z)= \left\langle \left(\Lambda -z- m(z)\right)^{-1} \right\rangle 
\ee
 such that $\im m(z)>0$ whenever $z\in \C_+$.  
 Then, we define the matrix $M(z)\equiv M_N(z,\Lambda)$ as 
\be\label{def_G0}
M(z):= \left(\Lambda -z- m(z)\right)^{-1}.
 \ee
Since $\Lambda$ is Hermitian, we have that $m(\bar z)= \overline {m(z)}$ and $M(\bar z)=M(z)^*$.
 \end{definition}

Under this definition, $m(z)$ is actually the Stieltjes transform of a probability measure $\mu_{N}$, called the \emph{free convolution} of the empirical measure of $\Lambda $ and the semicircle law with density 
\be\label{eq:semidensity}\rho_{sc}(x)=\frac{1}{2\pi} \sqrt{4-x^2}\mathbf 1_{x\in [-2,2]}.\ee
Moreover, the probability density $\rho_{N}$ of $\mu_{N}$ is determined from $m(z)$ by \be\label{eq:density_rhoN}
\rho_{N}(x)=\pi^{-1}\lim_{\eta\downarrow 0}\im m(x+\ii \eta).
\ee
Under the assumption $\|A\|=\oo(1)$, the support of $\rho_{N}$ is a single interval $[-a_N,b_N]$ with $|a_N-2|+|b_N-2|=\oo(1)$, and $m(z)$ is close to the Stieltjes transform of $\rho_{sc}$ given by
$m_{sc}(z)=(-z+\sqrt{z^2-4})/{2}$. 
In this paper, we focus on the eigenvalues and eigenvectors of $\VV$ in the bulk $[-2+\kappa,2-\kappa]$ for some constant $\kappa>0$, away from the edges $-a_N$ and $b_N$. We define the quantiles $\gamma_k$, $k\in \cal I$, of $\rho_N$ as 
\be\label{eq:gammak}
\gamma_k:=\sup_{x\in \R}\left\{\int_{-\infty}^{x}\rho_N(x)\dd x < \frac{k}{DN}\right\}.
\ee

To state the local law and streamline the presentation, in this paper, we adopt the following convenient notion of stochastic domination introduced in \cite{Average_fluc}.

\begin{definition}[Stochastic domination and high probability event]\label{stoch_domination}
	{\rm{(i)}} Let
	\[\xi=\left(\xi^{(N)}(u):N\in\mathbb N, u\in U^{(N)}\right),\hskip 10pt \zeta=\left(\zeta^{(N)}(u):N\in\mathbb N, u\in U^{(N)}\right),\]
	be two families of non-negative random variables, where $U^{(N)}$ is a possibly $N$-dependent parameter set. We say $\xi$ is stochastically dominated by $\zeta$, uniformly in $u$, if for any fixed (small) $\tau>0$ and (large) $D>0$, 
	\[\mathbb P\bigg(\bigcup_{u\in U^{(N)}}\left\{\xi^{(N)}(u)>N^\tau\zeta^{(N)}(u)\right\}\bigg)\le N^{-D}\]
	for large enough $N\ge N_0(\tau, D)$, and we will use the notation $\xi\prec\zeta$. 
	If for some complex family $\xi$ we have $|\xi|\prec\zeta$, then we will also write $\xi \prec \zeta$ or $\xi=\OO_\prec(\zeta)$. 
	
	\vspace{5pt}
	\noindent {\rm{(ii)}} As a convention, for two deterministic non-negative quantities $\xi$ and $\zeta$, we will write $\xi\prec\zeta$ if and only if $\xi\le N^\tau \zeta$ for any constant $\tau>0$. 

 \vspace{5pt}
	\noindent {\rm{(iii)}} Let $A$ be a family of random matrices and $\zeta$ be a family of non-negative random variables. Then, we use $A=\OO_\prec(\zeta)$ to mean that $\|A\|\prec \xi$, where $\|\cdot\|$ denotes the operator norm. 
	
	\vspace{5pt}
	\noindent {\rm{(iv)}} We say an event $\Xi$ holds with high probability (w.h.p.) if for any constant $D>0$, $\mathbb P(\Xi)\ge 1- N^{-D}$ for large enough $N$. More generally, we say an event $\Omega$ holds $w.h.p.$ in $\Xi$ if for any constant $D>0$,
	$\P( \Xi\setminus \Omega)\le N^{-D}$ for large enough $N$.
\end{definition}

\begin{lemma}[Local laws and rigidity of eigenvalues]\label{lem_loc}
Under \Cref{main_assm}, for any small constant $\tau>0$, the following local laws hold uniformly in $z=E+\ii\eta $ with $|z|\le \tau^{-1}$ and $ \eta \ge N^{-1+\tau}$. 
\begin{itemize}
    \item  \textbf{Anisotropic local law:} For any deterministic unit vectors $\mathbf{u}, \mathbf{v}\in \C^{DN}$, we have
\be\label{eq:aniso_local}
 \left(G(z)-M(z)\right)_{\bu\bv}\prec \sqrt{\frac{\im m(z)}{N\eta}}+\frac{1}{N\eta}.
\ee

\item  \textbf{Averaged local law:} For any deterministic matrix $B \in \C^{DN\times DN}$ with $\|B\|\le 1$, we have 
\be\label{eq:aver_local}
 \left\langle \left(G-M\right) B\right\rangle \prec \frac{1}{ N\eta } .
\ee

\end{itemize}
As a consequence of \eqref{eq:aver_local} when $B=I$, we have the rigidity of eigenvalues:
\be\label{eq:rigidity}
\lambda_k-\gamma_k \prec N^{-2/3}\min (k, DN+1-k)^{-1/3},\quad k\in \cal I.
\ee
In addition, all the above estimates remain valid even if we do not assume identical distributions for the diagonal and off-diagonal entries of $H$.
\end{lemma}

With the anisotropic local law \eqref{eq:aniso_local}, we can derive the following estimate on the products of resolvents. The proofs of \Cref{lem_loc} and \Cref{appen1} are both presented in \Cref{sec:Gprop}.
\begin{lemma}
\label{appen1}
Fix any integer $p\ge 1$. Suppose $(\Lambda_i)_{1\leq i \leq p}$ is an arbitrary sequence of $D\times D$ block matrices of the same form as $\Lambda$ and consisting of $N\times N $ deterministic blocks $A_i$ and $A_i^*$ with $\|A_i\|=\oo(1)$. Let $(B_i)_{1\leq i \leq p}$ be an arbitrary sequence of  deterministic matrices satisfying $\|B_i\|\le 1$. Suppose the anisotropic local law \eqref{eq:aniso_local} holds for all $G_i$, where $G_i:= G(z_i,H,\Lambda_i)$ for a sequence $(z_i)_{1\leq i \leq p}$ with $z_i\in \{z, \bar z\}$.  Then, for any deterministic unit vectors $\bu, \bv\in\mathbb{C}^{DN}$, we have   
\be\label{entprodG}
\bu^* \left( \prod_{i=1}^p G_i B_i\right)\bv \prec \frac{1}{\eta^{p-1}}.
\ee
\end{lemma}

\subsection{Proof ideas} \label{sec:idea}
In this subsection, we discuss some key ideas for the proof of the main results. 

\medskip

\noindent{\bf Integrable regime.} 
We define a sequence of interpolating matrices as 
\be\label{eq:VVtheta}
\VV(\theta):=H+\theta \Lambda, \quad \theta\in [0,1], \quad \text{with}\quad \VV(0)=H, \quad \VV(1)=\VV.
\ee
By the standard perturbation theory for eigenvalues, we have 
$\lambda_k'(\theta)=\bv_k(\theta)^* \Lambda \bv_k(\theta)$, where $k$ is a bulk index, $\lambda_k(\theta)$ is the $k$-th eigenvalue of $\VV(\theta)$, and $\bv_k(\theta)$ denotes the corresponding eigenvector. Thus, the shift of the eigenvalue $\lambda_k$ is of order $\int_{0}^1 \E|\bv_k(\theta)^* \Lambda \bv_k(\theta)|\dd \theta \le \int_{0}^1 (\E|\bv_k(\theta)^* \Lambda \bv_k(\theta)|^2)^{1/2}\dd \theta$. With the spectral decomposition of $G_\theta(z):=(\VV(\theta)-z)^{-1}$, we can show that 
\be\label{eq:GLambda0}
\E|\bv_k(\theta)^* \Lambda \bv_k(\theta)|^2 \prec \eta^2 \mathbb E { \tr\left[ \im G_\theta(z_k(\theta)) \Lambda \im G_\theta(z_k(\theta)) \Lambda \right]},\quad z_k(\theta):=\gamma_k(\theta)+\ii \eta,
\ee
where $\im G_\theta := (G_\theta-G_\theta^*)/(2\ii)$, $\eta=N^{-1+c}$ for an arbitrary small constant $c>0$, and $\gamma_k(\theta)$ is the classical location of $\lambda_k(\theta)$ defined as in \eqref{eq:gammak}. 
Then, we will prove a key technical estimate that 
\be\label{eq:GLambda}
\mathbb E \tr \left[ \im G_\theta(z_k(\theta)) \Lambda \im G_\theta(z_k(\theta)) \Lambda \right] \prec \|A\|_{\HS}^2. 
\ee
Together with \eqref{eq:GLambda0}, it implies that as long as $\|A\|_{\HS}\le N^{-\e_A}$ and $c<\e_A/2$, $\Lambda$ leads to a perturbation of $\lambda_k$ of order $\OO (N^{-1-\e_A/2})$ with probability $1-\oo(1)$, which is negligible under the scaling $DN$. 

Now, we have shown that $\lambda_k(1)$ is a perturbation of the eigenvalue $\lambda_k(0)$ of $H$. Suppose $\lambda_k(0)$ is the eigenvalue of a diagonal block $H_a$, $a\in \qqD$. By the bulk universality of Wigner matrices \cite{bourgade2016fixed}, we know that conditioning on $\lambda_k(0)$, the eigenvalue spectra of other blocks $H_b$, $b\ne a$, are away from $\lambda_k$ by a distance of order $N^{-1}$ with probability $1-\oo(1)$. Without loss of generality, suppose $a\ne 1$. We write $\bv_k$ as $\bv_k=(\bu_k^\top, \bw_k^\top)^\top$ for $\bu_k\in \C^N$ and $\bw_k\in \C^{(D-1)N}$, and let $\wt A:=\begin{pmatrix}
    A ,
    \mathbf 0_{N\times (D-3)N} ,
    A^*
\end{pmatrix}$ denote the upper right $N\times (D-1)N$ block of $\VV$. 
Then, from the eigenvalue equation $\VV \bv_k =\lambda_k \bv_k$, we get that 
$$H_1 \bu_k + \wt A \bw_k = \lambda_k \bu_k \ \Rightarrow \ \bu_k = - \left(H_1-\lambda_k \right)^{-1}\wt A \bw_k.$$
Our intuition is that $H_1$ and $\bw_k$ are almost independent in the integrable regime, so when $\dist(\lambda_k,\spec(H_1)) \gtrsim N^{-1}$, $\|\bu_k\|=\|\left(H_1-\lambda_k \right)^{-1}\wt A \bw_k\|$ should be small. In fact, we will prove that
$$\E\|\bu_k\|^2=\E\|\left(H_1-\lambda_k \right)^{-1}\wt A \bw_k\|^2 \prec N^{\tau+2c} \E \tr \left[ \im G_0(z_k) \Lambda \im G(z_k) \Lambda \right]$$
for any small constant $\tau>0$, where $G_0$ denotes the resolvent of $H$ and $z_k\equiv z_k(\theta=1)$ is defined in \eqref{eq:GLambda0}. For the right-hand side (RHS), we will prove a similar bound as in \eqref{eq:GLambda}:
\be\label{eq:GLambda2}
\E \tr \left[ \im G_0(z_k) \Lambda \im G(z_k) \Lambda \right] \prec \|A\|_{\HS}^2. 
\ee
Thus, as long as $\|A\|_{\HS}\le N^{-\e_A}$ and $\tau+2c<\e_A$, we obtain that $\|\bu_k\|\le N^{-\e_A/2}$ with probability $1-\oo(1)$. 
Applying this argument to all blocks $H_b$ with $b\ne a$, we can show that the $\ell_2$-norm of $\bv_k$ within each $\cI_b$ is negligible, implying that $\bv_k$ must be concentrated on the subset $\cI_a$. 

The proof of the \emph{multi-resolvent estimates} \eqref{eq:GLambda} and \eqref{eq:GLambda2} constitutes the main technical part of the above argument. We will apply the cumulant expansions formula in \cite{Cumulant2,He:2017wm} iteratively to bound the left-hand sides (LHS) of \eqref{eq:GLambda} and \eqref{eq:GLambda2}. A key point of the proof is that we need to find a systematic way to expand the resolvent entries so that after each expansion, the resulting expressions become strictly smaller until they satisfy the desired bound $\|A\|_{\HS}^2$. For this purpose, we need to keep track of the fine structures of the expressions and find the ``correct" resolvent entry to expand; see \Cref{sec:Gauss_exp} for more details.

\medskip

\noindent{\bf Chaotic regime.} 
By Markov's inequality, the QUE estimate \eqref{eq:main_evector1} follows immediately from the second moment bound $\E[\|E_a \bv_k\|^2-D^{-1}]^2\le N^{-\delta}$ for a constant $\delta>0$. To prove this bound, we again need to establish a multi-resolvent estimate for $G(z)$. More precisely, when $\|A\|_{\HS}\ge N^{\e_A}$, we aim to show that 
\be\label{eq:normtrace}
\E \avg{\im G(z) (E_a -D^{-1}) \im G(z) (E_a -D^{-1})} \le N^{-1-\delta}\eta^{-2}, \quad a\in \qqD, 
\ee
for a constant $\delta>0$ depending on $\e_A$. With the spectral decomposition of $\im G$, we obtain from \eqref{eq:normtrace} that 
\begin{align*}
  \mathbb E {\eta^4} \sum_{i,j\in \cal I} \frac{|\bv_i^*(E_a -D^{-1})\bv_j|^2}{|\lambda_i-z|^2 |\lambda_j-z|^2 } \le  N^{-\delta} ,\quad z=\gamma_k+\ii \eta,\ \ a\in \qqD. 
\end{align*}
From the above equation, we see that a bound of the form \eqref{eq:normtrace} at $\eta=N^{-1+c}$ implies that a local sum of $|\bv_i^*(E_a -D^{-1})\bv_j|^2$ in $|i-k|=\OO(N\eta)$ and $|j-k|=\OO(N\eta)$ is bounded by $N^{-\delta}$. In particular, by taking $i=j=k$, it gives the desired estimate.

Let $z_1,z_2\in \{z,\bar z\}$ and abbreviate $G_i:=G(z_i)$. Writing $\im G(z)=(G(z)-G(\bar z))/(2\ii)$ and $I=\sum_a E_a$, we can expand the LHS of \eqref{eq:normtrace} into a linear combination of two-resolvent traces of the form $\E L_{ab}$, $a,b\in \qqD$, with $L_{ab}:=D\avg{G_1 E_a G_2 E_b}$. Then, a key point of our proof is to establish the following two-resolvent local law for $\E L_{ab}$: there exists a deterministic matrix $K$ and a constant $\delta'>0$ so that  
\be\label{eq:2resol}
|\E L_{ab} - K_{ab}| \le N^{-1-\delta'} \eta^{-2}. 
\ee
This estimate shows that $\E L_{ab}$ is approximated by $K_{ab}$ up to a negligible error.  In addition, we will show that the deterministic part $K$ is ``flat" when $\|A\|_{\HS}\gg 1$, that is, $\max_{a,b} |K_{ab}-K_{11}|=\OO(N/\|A\|_{\HS}^2)$. It implies that the leading part of the LHS of \eqref{eq:normtrace} is of order $\OO(N/\|A\|_{\HS}^2)\ll N^{-1-\e_A}\eta^{-2}$ as long as we take $2c<\epsilon_ A$. 
This observation, combined with \eqref{eq:GLambda}, elucidates the essence of the quantum chaos transition occurring at $\|A\|_{\HS}\sim 1$.

We will see that, without the expectation, $L_{ab}$ satisfies the following (essentially sharp) local law: 
\be\label{eq:2resolweak}
| L_{ab} - K_{ab}| \prec N^{-1} \eta^{-2}. 
\ee
Multi-resolvent local laws as in \eqref{eq:2resolweak} have been an important development in the random matrix theory literature over the past few years. Specifically, (almost) sharp multi-resolvent local laws have been established for Wigner matrices in \cite{CEH2023,cipolloni-erdos2021,CEDS_EJP,CIPOLLONI2022109394,CES_Forum,CES_AAP} and they have played an important role in proving the eigenstate thermalization hypothesis for Wigner matrices \cite{cipolloni-erdos2021,CES_QUE3,CIPOLLONI2022109394,CEH2023}. Later, they are extended to deformed Wigner matrices \cite{CEJK2023} and generalized Wigner matrices with mean-field variance profiles \cite{adhikari2023eigenstate}. In these works, the canonical proof is to derive a self-consistent equation for $L$ and then bound the fluctuation error carefully using the cumulant expansion method. However, their proofs depend crucially on the flat (or some special mean-field) variance profiles of the random matrices, which do not apply to our block variance profile setting. Instead, we will adopt a dynamical approach called the method of \emph{characteristic flow}. 
This idea of estimating resolvents dynamically first appeared in \cite{Lee_edge}, and was later applied to various models \cite{HL2019_PTRF,Landon2024_PTRF,Adhikari2020_PTRF,Adhikari2023_PTRF,LS_2022,Bourgade_JEMS} to show single-resolvent local laws (or closely related quantities). More recently, its usefulness in establishing multi-resolvent local laws has also been recognized in \cite{BF_LQG,CES_PTRF,CEH2023}, as well as the concurrently posted paper \cite{CEX2023}. However, in comparison to these works, our specific setting presents two distinct challenges: the non-mean-field nature of the problem and the arbitrary interaction $A$. Moreover, we need to improve the estimate \eqref{eq:2resolweak} to \eqref{eq:2resol} by a small factor $N^{-\delta'}$.

In the proof of \eqref{eq:2resolweak}, we will consider a matrix-valued process $\VV(t)$, where each diagonal block of $\VV$ satisfies an independent matrix-valued Ornstein-Uhlenbeck (OU) process $H_a(t)$ with initial condition $H_a$, and each off-diagonal block satisfies a matrix-valued linear differential equation. Then, we let the spectral parameter $z_t$ of the resolvent $G_t:=(\VV(t)-z_t)^{-1}$ satisfy a first-order differential equation, which is indeed the characteristic flow of the underlying complex Burgers equation. Along this flow, the leading terms in the time evolution of $G_t$ cancel each other, and $\eta_t=\im z_t$ decreases, which enables us to transfer the estimates \eqref{eq:2resol} and \eqref{eq:2resolweak} from larger scales of $\eta$ to smaller scales at the expense of an added Gaussian component. But, we can easily remove it with a standard Green's function comparison argument.    
The aforementioned flow approach comprises two components: an initial estimate at $t=0$ and large $\eta_0$, followed by a flow estimate for the subsequent time evolution. One interesting observation is that the initial estimate, often believed to be simple at large $\eta$, actually becomes a major obstacle when attempting to improve \eqref{eq:2resolweak} to \eqref{eq:2resol}. 
More precisely, when we expand $L_{ab} - K_{ab}$ using cumulant expansions, besides some centered random fluctuations, we also find leading deterministic terms in $L_{ab}$ of exact order $N^{-1}\eta^{-2}$. 
We will need to obtain the exact expressions of these terms and show an intricate cancellation between them through an explicit calculation. However, we remark that this cancellation is not accidental and indeed corresponds to a deep \emph{self-energy renormalization mechanism} originally identified in the context of random band matrices \cite{yang2021delocalization}; see \Cref{rem:self} below. 
We believe that this kind of self-energy renormalization is a conceptually universal mechanism behind the chaotic behavior of non-mean-field random models.

Finally, we remark that $L_{ab}=N^{-1}\sum_{x\in \cal I_b,y\in \cal I_a}G_{xy}(z_1)G_{yx}(z_2)$ corresponds to the $T$-variable in the study of random band matrices, so the $T$-expansion method developed in \cite{xu2022bulk,yang2021delocalization,yang2022delocalization,YangYin2021,bourgade2019random,erdos2013delocalization} is potentially useful for the proof of the two-resolvent local laws for $L$. Although it is not as concrete as the flow method in the current setting, the $T$-expansion method can be crucial when we consider extensions to the large $D$ case later.

Given the QUE estimates in \Cref{mix}, we can prove the GOE/GUE statistics for the bulk eigenvalues of $\VV$ adopting an idea developed recently in \cite{xu2022bulk}. More precisely, similar to many previous proofs of bulk universality of random matrices, our starting point is the three-step strategy initiated in \cite{ErdPecRamSchYau2010,ErdSchYau2011}. Readers can refer to \cite{Erds2017ADA} for an overview of this strategy. The first step is to establish a local law for the resolvent $G(z)$ up to the optimal scale $\eta\gg N^{-1}$, which has been accomplished in \Cref{lem_loc}.
For the second step, one considers the matrix Brownian motion $\VV(t)=H_t +\Lambda$, where $H_t=(h_{ij}(t))_{i,j\in \cI}$ is a matrix-valued OU process. Note that in contrast to the flow $\VV(t)$ defined above in the characteristic flow approach, $H_t$ here satisfies a $DN\times DN$ matrix-valued OU equation (see \eqref{defVt} for the definition) and $\Lambda$ remains unchanged. 
With the local law as main input, it was proved in \cite{LANDON20191137,LanYau2015} that the local bulk statistics of $\VV(t)$ converge to those of GOE/GUE at time $t\gg N^{-1}$. 
In the third step, one needs to show that the local bulk statistics of the original matrix $\VV$ are well-approximated by those of $\VV(t)$, which is achieved by comparing the moments of $\langle\im G(z) \rangle$ with those of $\langle\im G_t(z) \rangle$, where $G_t(z):=(\VV(t)-z)^{-1}$ denotes the resolvent of $\VV(t)$. 
In the literature (see e.g., \cite{ErdPecRamSchYau2010,ErdSchYau2011,Erds2017ADA,erdHos2012rigidity}), this step is typically accomplished through a standard Green's function comparison approach via moments matching between the entries of $\VV$ and those of $\VV(t)$. However, this kind of moment-matching argument fails in our specific setting. 
To address this challenge, we adopt the idea in \cite{xu2022bulk}, that is, we utilize the QUE of bulk eigenvectors to perform the Green's function comparison. We will see that the leading terms appearing in the comparison of $\langle\im G(z) \rangle$ with $\langle\im G_t(z) \rangle$ will contain factors of the form
${\bf v}_i^* (E_a-D^{-1}){\bf v}_j $.
These factors are small as shown in the QUE estimate \eqref{eq:extend:main_evector1}, which will help with bounding the leading terms in the comparison. This allows us to prove that the local bulk statistics of $\VV$ match those of $\VV(t)$ asymptotically.

\section{Chaotic regime: eigenvectors}\label{sec:proof_mix}

In this section, we prove \Cref{mix}. As discussed in \Cref{sec:idea}, the main ingredient for the proof is a precise two-resolvent local law on quantities of the form $\langle G(z_1)E_{a}G(z_2)E_{b}\rangle$, $a,b\in \llbracket D\rrbracket$. We first introduce several notations. 

\begin{definition}
\label{def_wtM} 
Define the spectral domain $\mathbf D(\kappa,\tau):=\{z=E+\ii \eta\in \C :|z|\le \tau^{-1}, |E|\le 2-\kappa, |\eta| \ge N^{-1+\tau}\}$ for an arbitrarily small constant $\tau>0$.
For $z_1,z_2\in \mathbf D(\kappa,\tau)$, we define the $D\times D$ matrices $\wh M$ and $L$ as
\begin{align}\label{def_ML}
    \wh M_{ab}(z_1,z_2,\Lambda):=D\langle M(z_1) E_{a} M(z_2) E_{b}\rangle ,\quad L_{ab}(z_1,z_2, H, \Lambda):= D\langle G(z_1)E_{a}G(z_2)E_{b}\rangle , 
\end{align}
for  $a,b\in \llbracket D\rrbracket$, and define the $D\times D$ matrix $K$ by
\begin{align}\label{def_K}   K(z_1,z_2, \Lambda ):= \left[1-\wh M(z_1,z_2,\Lambda)\right]^{-1} \wh M(z_1,z_2,\Lambda).
\end{align}
 \end{definition}

For ease of presentation, we introduce the following simplified notations: given a matrix-valued function (e.g., $G$, $M$, $\wh M$, $L$, and $K$) of $z$, we use subscripts to indicate its dependence on the spectral parameters. For example, we will denote $G_i:= G(z_i,H,\Lambda)$, $M_{i}:=M(z_{i},\Lambda),$ \smash{$\wh M_{(1,2 )}:=\wh M(z_{1 },z_{2 } ,\Lambda )$}, $L_{(1,2)}:=L(z_{1 },z_{2 }, H ,\Lambda ),$ and $K_{(1,2 )}:=K(z_{1 },z_{2 } ,\Lambda )$.

\begin{lemma}\label{main_lemma} 
Take any $z=E+\ii\eta \in \mathbf D(\kappa,\tau)$ with $\eta = N^{-1+\eqq}$ for a small constant $\eqq>0$. Then, under the assumptions of \Cref{mix}, there exists a constant $c_L>0$ (depending on $\eqq$, $\e_A$, and $\delta_A$) such that 
\be\label{main_EGG}
\left(\E L_{(1,2)}- K_{(1,2)}\right)_{ab}=\OO(N^{-1-c_L}\eta^{-2})
\ee
for $z_1,z_2\in \{z, \bar z\}$ and $a,b \in \qqD$.
\end{lemma}

Before proving this key lemma, we first use it to prove \Cref{mix}. 

\begin{proof}[Proof of \Cref{mix}]
For $z=E+\ii \eta$, using the spectrum decomposition of $\im G(z)$, we get that for any $DN\times DN$ matrix $B$,
$$
\tr \left[ \im G (z) B\im G(z)  B^* \right]=\eta^2\sum_{i,j\in \cI}\frac{|\bv^*_i B \bv_j|^2}{|\lambda_i-z|^2 |\lambda_j-z|^2 }.
$$
In particular, choosing $B=E_{a}-D^{-1}I$ and $z_k=\gamma_k+\ii N^{-1+\eqq}$ and using the rigidity of eigenvalues in \eqref{eq:rigidity}, we get from this estimate that for any constant $c\in (0,\e_L)$, 
\be\label{eq:bdd_spec}
 \max_{i,j\in \qq{k-N^c,k+N^c}}|\bv^*_i (E_{a}-D^{-1}) \bv_j|^2 \prec N^{-2+2\eqq} \tr \left[ \im G (z_k) (E_{a}-D^{-1}I)\im G(z_k) (E_{a}-D^{-1}I) \right].
\ee
It remains to bound the RHS. By denoting $z_1=z_k$, $z_2=\bar z_k$ and using \eqref{main_EGG}, its expectation is estimated as
\begin{align}
   &-\frac{N^{-2+2\eqq}}{4} \E\tr \Big[ (G_1-G_2)\Big(E_{a}-D^{-1}\sum_b E_b\Big)(G_1-G_2) \Big(E_{a}-D^{-1}\sum_{b'} E_{b'}\Big) \Big] \nonumber\\
   &=N^{-1+2\eqq} \bigg(\E\cal L_{aa} - \frac{2}{D} \sum_{b=1}^D \E\cal L_{ab} + \frac{1}{D^2} \sum_{b,b'=1}^D \E\cal L_{bb'}\bigg)\nonumber\\
   &=N^{-1+2\eqq} \bigg(\cal K_{aa} - \frac{2}{D} \sum_{b=1}^D \cal K_{ab} + \frac{1}{D^2} \sum_{b,b'=1}^D \cal K_{bb'}\bigg) + \OO\left(N^{-c_L}\right),\label{eq:bdd_spec2}
\end{align}
where the $D\times D$ matrices $\cal L$ and $\cal K$ are defined as $\cal L:=(L_{(12)}+L_{(21)}-L_{(11)}-L_{(22)})/4$ and $\cal K:=(K_{(12)}+K_{(21)}-K_{(11)}-K_{(22)})/4$. 
On the other hand, we claim that for $i,j\in \{1,2\}$, 
\be\label{eq:est_M-M}
\max_{a,b,a',b'\in \qqD}\left|\left(K_{(ij)}\right)_{ab}-\left(K_{(ij)}\right)_{a'b'}\right|=\OO \left({N}/{\|A\|_{\HS}^2}\right).
\ee
We postpone the proof of this estimate to \Cref{lem_bas_M} in the appendix. With \eqref{eq:est_M-M}, we obtain that
\be\label{eq:bdd_spec3} N^{-1+2\eqq} \bigg(\cal K_{aa} - \frac{2}{D} \sum_{b=1}^D \cal K_{ab} + \frac{1}{D^2} \sum_{b,b'=1}^D \cal K_{bb'}\bigg) \lesssim N^{-2\e_A + 2\e_L} .
\ee
Combining \eqref{eq:bdd_spec}, \eqref{eq:bdd_spec2}, and \eqref{eq:bdd_spec3}, we obtain that for any small constant $\e>0$, 
\be\label{eq:bdd_spec4} \E  \max_{i,j\in \qq{k-n^c,k+n^c}}|\bv^*_i (E_{a}-D^{-1}) \bv_j|^2 \le N^{-c_L +\e} +N^{-2\e_A + 2\eqq+\e}.
\ee
If we take $\eqq<\e_A/2$ and $\e< (c_L\wedge \e_A)/2$, this gives that 
$$ \E  \max_{i,j\in \qq{k-n^c,k+n^c}}|\bv^*_i (E_{a}-D^{-1}) \bv_j|^2 \le N^{-c_L/2} +N^{-\e_A/2}.$$
Then, applying Markov's inequality and a simple union bound over $a\in \qqD$ concludes \eqref{eq:extend:main_evector1}. Taking $i=j=k$, we obtain \eqref{eq:main_evector1}.
\end{proof}

For the proof of \Cref{main_lemma}, we first establish it for a special case with Gaussian divisible $H$. Its proof will be presented in \Cref{sec:flow_prelim}.

\begin{lemma}\label{main_lemma_G}
Under the setting of \Cref{mix}, suppose $H_a$, $a\in \qqD$, are of the form
\be\label{Gcom_a}
H_a = \sqrt{1-N^{-\e_g}}\cdot H_a^{(0)}+ {N^{-\e_g/2}} H_a^{(g)},
\ee
where $H_a^{(0)}$ are independent Wigner matrices satisfying the assumptions for $H_a$ in \Cref{main_assm} and $H_a^{(g)}$ are i.i.d.~GUE satisfying \eqref{eq:meanvar} and \eqref{eq:meanvar2}. Then, for small enough constant $\e_g>0$ (depending on $\e_L$ and $\e_A$) and $z=E+\ii\eta $ with $|E|\le 2-\kappa$, there exists an absolute constant $C>1$ such that 
\be\label{main_G_es}
\|\E L_{(1,2)}- K_{(1,2)}\|_{\HS}\prec N^{-1-(\e_g\wedge \delta_A)}\eta^{-2},\quad \text{for} \quad  N^{-1+C\e_g } \le \eta \le N^{-C\e_g }\ , \ z_1,z_2\in \{z, \bar z\}.
\ee
\end{lemma}

Now, \Cref{main_lemma} follows from \Cref{main_lemma_G} combined with a Green's function comparison result stated in \Cref{main_lemma_com}. We postpone the proof of \Cref{main_lemma_com} to \Cref{sec:comparison}.
\begin{lemma}\label{main_lemma_com}
Let $H$ and $\wt H$ be two matrices satisfying \Cref{main_assm}. Suppose they satisfy the following moment-matching conditions: for $i,j\in \cal I$ and integers $k,l\ge 0$, 
\be\label{comH0}
\E (H_{ij})^k (H^*_{ij})^l - \E (\wt H_{ij})^k (\wt H^*_{ij})^l=0 \ \ \text{for} \ \ k+l\le 3, 
\ee
and there exists a constant $\delta \in (0,1/2)$ such that 
\be\label{comH}
\left|\E (H_{ij})^k (H^*_{ij})^l - \E (\wt H_{ij})^k (\wt H^*_{ij})^l\right| \lesssim N^{-2-\delta}  \ \ \text{for} \ \ k+l= 4. 
\ee
Then, for any $z\in \mathbf D(\kappa,\tau)$, $z_1,z_2\in \{z, \bar z\}$, and $a,b \in \llbracket D\rrbracket$, 
\be\label{eq:cpmparison}
\mathbb E \langle G_1E_{a}G_2E_{b}\rangle-\mathbb E \langle \wt G_1E_{a}\wt G_2E_{b}\rangle\prec N^{-1-\delta}\eta^{-2},
\ee
where $\wt G_i \equiv G(z_i,\wt H,\Lambda)$, $i\in \{1,2\}$, denote the Green's functions of $\wt H$.
\end{lemma}

\begin{proof}[Proof of Lemma \ref{main_lemma}]  
Given the matrix $H$ considered in \Cref{main_lemma}, we can construct another random matrix $\wt H$ satisfying the setting in \Cref{main_lemma_G} and such that the moment-matching conditions \eqref{comH0} and \eqref{comH} hold with $\delta=\e_g$ (see e.g., Lemma 6.5 in \cite{erdHos2012bulk}).  
By \Cref{main_lemma_G}, as long as we choose $\e_g$ small enough such that $C\e_g \le \eqq \le 1-C\e_g$, there is 
$$D\mathbb E \langle \wt G_1E_{a}\wt G_2E_{b}\rangle- (K_{(1,2)})_{ab}
\prec N^{-1-(\e_g\wedge \delta_A)}\eta^{-2},$$
for $\eta = N^{-1+ \eqq} $. On the other hand, by \Cref{main_lemma_com}, we have that 
$$\mathbb E \langle  G_1E_{a} G_2E_{b}\rangle- \mathbb E \langle \wt G_1E_{a}\wt G_2E_{b}\rangle \prec N^{-1-\e_g}\eta^{-2}.$$
Combining the above two estimates, we conclude Lemma \ref{main_lemma} by choosing $c_L<\e_g\wedge \delta_A$.
\end{proof}

\section{Gaussian divisible case} 
\label{sec:flow_prelim}

In this section, we prove the key lemma---\Cref{main_lemma_G}---in the proof of \Cref{mix}. We first introduce some new notations that are similar to those in \Cref{def_wtM} but with three $z$ arguments. 

\begin{definition}\label{defLK3}
Define the $D\times D\times D$ tensors $L$ and $K$ as
$$
\left[L(z_1, z_2, z_3, H, \Lambda)\right]_{a_1 a_2 a_3}:= D\langle G_1E_{a_1}G_2E_{a_2}G_3E_{a_3}\rangle,
$$
$$\left[K(z_1, z_2, z_3, \Lambda)\right]_{a_1 a_2 a_3} = \sum_{b_1, b_2, b_3}(I - \wh M_{(1,2)})^{-1}_{a_1b_1}(I - \wh M_{(2,3)})^{-1}_{a_2b_2}(I - \wh M_{(3,1)})^{-1}_{a_3b_3}D\langle M_1E_{b_1}M_2E_{b_2}M_3E_{b_3}\rangle ,
$$
for $a_i\in \qqD$, $i \in \{1,2,3\}$. Here, we have abused the notations a little bit and still use $L$ and $K$ to denote these tensors. Moreover, we will also abbreviate them by $L_{(1,2,3)}$ and $K_{(1,2,3)}$.
\end{definition} 

As discussed in \Cref{sec:idea}, the proof of \Cref{main_lemma_G} is based on the characteristic flow method in \cite{Lee_edge}. We start with several estimates at the initial scale $\eta\sim N^{-\e_g}$ as stated in Lemmas \ref{lemma_G_large} and \ref{lemma_G_large_2}. Their proofs will be postponed to Sections \ref{sec_simpleLK} and \ref{sec_hardLK}. 

\begin{lemma}\label{lemma_G_large}
Under the assumptions of \Cref{mix}, take any $z=E+\ii \eta$ with $|E|\le 2-\kappa$ and $\eta\sim N^{-\e_g}$. Then, for any constant $\e_g\in (0,1/4)$ and $z_i\in \{z, \bar z\}$, $i\in\{1,2,3\}$, we have that 
\begin{align}
 \|L_{(1,2 )}-K_{(1,2)}\|_{2}&\prec N^{-1 }\eta^{-3},\label{largG_Ent}\\
 \|L_{(1,2,3 )}-K_{(1,2,3)}\|_2 & \prec N^{-1 }\eta^{-4} ,\label{largG_3G}
\end{align}
where $\|\cdot\|_2$ denotes the $\ell_2$-norm by regarding matrices and tensors as vectors (for matrices, it is the Hilbert-Schmidt norm). 
\end{lemma}

\begin{lemma}\label{lemma_G_large_2}
Under the assumptions of \Cref{mix}, take any $z=E+\ii \eta$ with $|E|\le 2-\kappa$ and $\eta\sim N^{-\e_g}$. Then, for any constant $\e_g\in (0,1/4)$ and $z_i\in \{z, \bar z\}$, $i\in\{1,2\}$, we have that 
\begin{align}
& \max_{a\in \qqD} \left|\E \langle \left(G(z)-M(z)\right) E_a\rangle\right|   \prec N^{-1},\label{largG_1G} \\
& \| \E L_{(1,2 )}-K_{(1,2)}\|_2 \prec N^{-1}\eta^{-2}\left(\eta + N^{-\delta_A}\right).\label{largG_Ex}
 \end{align}
\end{lemma}

Next, we apply the flow method to transfer the bounds at larger scales of $\eta$ to smaller scales of $\eta$.  

\begin{definition}[Characteristic flow]\label{flow} 
Given a starting time $t_0\in \R$ and initial values $(z_{t_0},\Lambda_{t_0})$, we define flows of $z$ and $\Lambda$ as 
\be\label{eq:deterflow} \frac{\dd}{\dd t}z_t=-\frac12 z_t- \langle M_t\rangle,\quad   \frac{\dd}{\dd t}\Lambda_t=-\frac12 \Lambda_t,\quad t\ge t_0, 
\ee
where $M_t:= M(z_t, \Lambda_t)$ is the solution to \eqref{def_M} with $z$ and $\Lambda$ replaced by $z_t$ and $\Lambda_t$. Let $t_c:=\inf\{t\ge t_0: \im z_{t_c}=0\}$ be the first time $\im z_t$ vanishes. We also introduce the function $Z:\C\times \C^{DN\times DN}\to\C^{DN\times DN} $ as $Z(z ,\Lambda ):=z I - \Lambda$ and abbreviate that $Z_t:=Z(z_t , \Lambda_t )$. Note that $Z_t$ satisfies
\be\label{eq:Zt}
 \frac{\dd}{\dd t}Z_t=-\frac12 Z_t- \langle M_t\rangle .
\ee
Given the initial random matrix $H_{t_0}$ satisfying \Cref{main_assm} with diagonal blocks $(H_a)_{t_0}$, $a\in \qqD$, we define the flow $H_t$ as a $DN\times DN$ random matrix with diagonal blocks $(H_a)_{t}$ being matrix-valued OU processes 
\be\label{eq:Ht} \dd (H_a)_t = -\frac12(H_a)_t\dd t + \frac{1}{\sqrt{N}}\dd (B_a)_t,\ee
where $(B_a)_t$, $a\in \qqD$, are independent complex Hermitian matrix Brownian motions (i.e., $\sqrt{2} \re (B_a)_{ij}$ and $ \sqrt{2} \im (B_a)_{i j}$, $i<j$, and $(B_a)_{i i}$ are independent standard Brownian motions and $(B_a)_{j i}=(\overline{B}_a)_{i j}$). In particular, for each $t\ge t_0$, $(H_a)_t$ has the same law as  
\be\label{eq:Gauss_div}
e^{-(t-t_0)/2}\cdot  H_a^{(0)}+ \sqrt{1-e^{-(t-t_0)}}\cdot H_a^{(g)},
\ee
where $H_a^{(g)}$, $a\in \qqD$, are i.i.d.~GUE. Then, we define the Green's function flow 
$G_t=\left(H_t+\Lambda_t-z_t\right)^{-1}.$
Finally, with $(z_{i})_t$, $i\in \{1,2,3\}$, $\Lambda_t$, $H_t$, and $M_t$, we can define 
$$ \wh M_{(1,2),t}=\wh M((z_{1})_t,(z_{2})_t,\Lambda_t),\quad L_{(1,2),t}=L((z_{1})_t,(z_{2})_t,H_t, \Lambda_t),\quad K_{(1,2),t}=K((z_{1})_t,(z_{2})_t,\Lambda_t)$$ 
as in \Cref{def_wtM}, and define
$$L_{(1,2,3),t}=L((z_{1})_t, (z_{2})_t, (z_{3})_t, H_t, \Lambda_t),\quad K_{(1,2,3),t}=K((z_{1})_t, (z_{2})_t, (z_{3})_t, \Lambda_t)$$ 
as in \Cref{defLK3}. 
\end{definition}

We now collect some basic properties of the characteristic flows in \eqref{eq:deterflow}. 
\begin{lemma} 
\label{lem_prop_flow} 
Under \Cref{flow}, the following properties hold.

\begin{itemize}
    \item Denote $m_t:=\langle M_t\rangle$. Suppose $\im z_t=\oo(1)$. Then, for any $t\ge t_0$, we have that
    \be\label{t_eta}
t_c-t=\frac{\im z_t}{\im m_t}(1+\oo(1)).
 \ee

 \item $M_t$ satisfies the following equation for ${t\ge t_0}$:
 \be\label{devM}
  \frac{\dd }{\dd t} M( z_t, \Lambda_t)=\frac{1}{2}M( z_t, \Lambda_t).
\ee

 \item {\bf Conjugate flow}: We have $\overline Z_t=Z(\bar z_t, \Lambda_t)$ and $\overline M_t=M(\bar z_t, \Lambda_t)$. Moreover, they satisfy the following equations under the conjugate flows $(\bar z_t, \Lambda_t)$ for ${t\ge t_0}$:
\be\label{Conflow}
 \frac{\dd}{\dd t}Z(\bar z_t, \Lambda_t )=-\frac12 Z(\bar z_t, \Lambda_t )- \langle M(\bar z_t, \Lambda_t )\rangle ,\quad \frac{\dd}{\dd t} M( \bar z_t, \Lambda_t)=\frac{1}{2}M( \bar z_t, \Lambda_t).
\ee

   \item  For any $   (z_{i})_{t}\in \{z_t, \bar z_t\}$, $i\in\{1,2,3\}$, $\wh M_{(1,2),t}$ and $K_{(1,2),t}$ satisfy the equations
   \be\label{eq:whMK}
    \frac{\dd}{\dd t}\wh M_{(1,2),t} =\wh M_{(1,2),t},\quad \frac{\dd}{\dd t} K_{(1,2),t}= \left(K_{(1,2),t}\right)^2+K_{(1,2),t},
    \ee
     and $K_{(1,2,3),t}$ satisfies that for any $a_1,a_2,a_3\in \qqD$, 
    \begin{align}
        \frac{\dd}{\dd t} (K_{(1,2,3),t})_{a_1 a_2 a_3}= \frac{3}{2}K_{(1,2,3),t} + \sum_{a=1}^D &\left[(K_{(1,2),t})_{a_1 a}(K_{(1,2,3),t})_{a a_2 a_3}  +(K_{(2,3),t})_{a_2 a}(K_{(1,2,3),t})_{a_1 a a_3} \right. \nonumber\\
    &\left.+ (K_{(3,1),t})_{a_3 a}(K_{(1,2,3),t})_{a_1 a_2 a}\right] .   \label{eq:whMK3} 
    \end{align}

\end{itemize} 
  
\end{lemma}

\begin{proof}
The estimate \eqref{t_eta} follows by tracking $\im m_t$ solved from the first equation in \eqref{eq:deterflow} and using the simple facts that $z_s-z_t=\oo(1)$ and $m_s(z_s)-m_t(z_t)=\oo(1)$ for all $s\in [t,t_c]$ as long as $|t_c-t|=\oo(1)$. Now, we prove the identity  \eqref{devM}. Using $M_t=-(Z_t+m_t)^{-1}$ and equation \eqref{eq:Zt}, we obtain that 
\be\label{derM_t}
\dot M_t = M_t(\dot Z_t + \dot m_t)M_t = M_t(- Z_t/2 - m_t + \dot m_t)M_t,
\ee
where we have used Newton's notation for derivatives with respect to time. Taking trace and using $Z_tM_t = -I - m_tM_t$, we obtain 
\[\dot m_t = \frac{1}{2}m_t + \frac{1}{2}m_t\langle M_t^2\rangle - m_t\langle M_t^2\rangle + \dot m_t\langle M_t^2\rangle \ \Rightarrow \  \dot m_t = \frac{1}{2}m_t .\] 
Plugging it back into \eqref{derM_t}, we conclude \eqref{devM}. 
Taking conjugate transposes of \eqref{eq:Zt} and \eqref{devM}, we get \eqref{Conflow}. 
The equations \eqref{eq:whMK} and \eqref{eq:whMK3} follow from direct calculations with \eqref{eq:Zt} and \eqref{devM}. 
\end{proof}

For the proof of \Cref{main_lemma_G}, we need the following several key lemmas, Lemmas \ref{lem_flow_1}--\ref{lem_flow_3}. These lemmas rely on the previously defined flow and their detailed proofs will be provided in \Cref{sec:flow}.

\begin{lemma}\label{lem_flow_1}   
Suppose the assumptions of \Cref{mix} hold for $H_{t_0}$ and $\Lambda_{t_0}$. 
Under \Cref{flow}, take any $z_{t_0}=E+\ii \eta\in \mathbf D(\kappa, \tau)$ such that $t_c-t_0\sim N^{-\e_g}$. Let $(z_1)_t, (z_2)_t\in \{z_t, \bar z_t\}$ for $t\in [t_0,t_c]$. Then, for any fixed constant $C>4$ and $t\in [t_0,t_c- N^{-1+C\e_g}]$, 
\be\label{flow_1_result}
 \|L_{(1,2),t}- K_{(1,2),t}\|_{2}\prec \frac{(t_c-t_0)^2}{(t_c-t)^2}\left\|L_{(1,2),t_0}- K_{(1,2),t_0}\right\|_2+\frac{1}{N(t_c-t)^2}.
 \ee 
Together with \eqref{largG_Ent}, it implies that for any $t\in [t_0,t_c- N^{-1+C\e_g}]$, 
\be\label{flow_1_result_2}
 \left\| L_{(1,2),t}- K_{(1,2),t}\right\|_2\prec  \frac{N^{\e_g}}{N(t_c-t)^2}.
 \ee
\end{lemma}

\begin{lemma}\label{lem_flow_2}   
Under the assumptions of \Cref{lem_flow_1}, let $(z_1)_t, (z_2)_t, (z_3)_t\in \{z_t, \bar z_t\}$ for $t\in [t_0,t_c]$. Then,  we have that for any $t\in [t_0,t_c- N^{-1+C\e_g}]$,
\be\label{flow_3G_res}
\left\| L_{(1,2,3),t}- K_{(1,2,3),t}\right\|_2\prec \frac{(t_c-t_0)^3}{(t_c-t)^3}\left\|L_{(1,2,3),t_0}- K_{(1,2,3),t_0}\right\|_2+\frac{N^{\e_g}}{N(t_c-t)^3}. 
\ee
Together with \eqref{largG_3G}, it implies that for any $t\in [t_0,t_c- N^{-1+C\e_g}]$,
\be\label{flow_3G_res_2}
\left\|L_{(1,2,3),t}- K_{(1,2,3),t}\right\|_2 \prec \frac{N^{\e_g}}{N(t_c-t)^3}. 
\ee
\end{lemma}

\begin{lemma}\label{lem_flow_1.5}
Under the assumptions of \Cref{lem_flow_1}, we have that for any $t\in [t_0,t_c- N^{-1+C\e_g}]$,
\be\label{flow_1G_res}
\max_{ a\in \qqD}\E \langle \left(G_t-M_t\right) E_a\rangle   \prec \frac{t_c-t_0}{t_c-t}
\max_{ a\in \qqD}
\left|\E \langle \left(G_{t_0}-M_{t_0}\right) E_a\rangle \right|
+\frac{N^{\e_g}}{N^2(t_c-t)^2}.
\ee
Together with \eqref{largG_1G}, it implies that for any $t\in [t_0,t_c- N^{-1+C\e_g}]$,
\be\label{flow_1G_res_2}
 \E \langle \left(G_t-M_t\right) E_a\rangle   \prec \frac{N^{-\e_g}}{N(t_c-t)}.  \ee
 \end{lemma} 

Armed with the above three lemmas, we can prove the following key result for the proof of \Cref{main_lemma_G}.

 \begin{lemma}\label{lem_flow_3}  
 Under the assumptions of  \Cref{lem_flow_1}, we have that for any $t\in [t_0,t_c- N^{-1+C\e_g}]$,
\be\label{flow_3_result}
\left\|\E L_{(1,2),t}- K_{(1,2),t}\right\|_2\prec \frac{(t_c-t_0)^2}{(t_c-t)^2}\left\|\E L_{(1,2),t_0}- K_{(1,2),t_0}\right\|_2+ \frac{N^{-\e_g}}{N(t_c-t)^2}+\frac{N^{2\e_g}}{N^2(t_c-t)^3}.
\ee
Together with \eqref{largG_Ex}, it implies that for any $t\in [t_0,t_c- N^{-1+C\e_g}]$,
\be\label{flow_3_result_2}
\left\| \E L_{(1,2),t}- K_{(1,2),t}\right\|_2\prec  \frac{N^{-(\e_g\wedge \delta_A)}}{N(t_c-t)^2}.
 \ee
 \end{lemma}

\begin{proof}[Proof of \Cref{main_lemma_G}] 
Suppose we want to prove \eqref{main_G_es} for $z=E+\ii \eta$ and $ \Lambda$ satisfying the assumptions of \Cref{main_lemma_G}. In particular, we have $|E|\le 2-\kappa$ and $N^{-1+C\e_g } \le \eta \le N^{-C\e_g }$ for some constant $C>4$ such that \Cref{lem_flow_3} holds. 
Now, take $t_f=0$ and let $t_0=t_f- N^{-\e_g}/2$. We can find initial values $z_{t_0}$ and $\Lambda_{t_0}$ such that $z_{t_f}=z$ and $\Lambda_{t_f}=\Lambda$ at $t=t_f$. (In fact, we can first solve the second equation in \eqref{eq:deterflow} as $\Lambda_t=e^{(t_f-t)/2}\Lambda$ and then plug it into the first equation in \eqref{eq:deterflow}. In the resulting equation, the RHS is a locally Lipschitz function in $t$ and $z$, so there exists a solution $z_{t_0}$ at $t=t_0$.) 
Since $m_t(z)=m_{sc}(z)+\oo(1)$ under $\|\Lambda_t\|=\oo(1)$ (see \eqref{eq:msc} in the appendix), we have $\im m_t(z_t)\sim 1$. Thus, by \eqref{t_eta}, we know that $t_c-t_f\sim \eta $, which also gives $t_c-t_0=(t_f-t_0)(1+\oo(1))=N^{-\e_g}(1/2+\oo(1))$. 

Under the assumption \eqref{Gcom_a} for $H$, we can find another Gaussian divisible matrix $H_{t_0}$ such that $H_{t_f}$ has the same law as $H $ under the evolution \eqref{eq:Ht}. With the above choices of $H_{t_0}$, $z_{t_0}$ and $\Lambda_{t_0}$, we have 
$$ \E L_{(1,2)}= \E L_{(1,2),t_f} ,\quad  K_{(1,2)}=  K_{(1,2),t_f}. $$
Finally, using \eqref{flow_3_result_2} at $t=t_f$, 
we conclude the desired result \eqref{main_G_es}. 
\end{proof}

\subsection{Proof of \Cref{lemma_G_large}}\label{sec_simpleLK}

Our proof relies on the following formula derived from the definitions of $G$ and $M$ in \eqref{def_M}, 
\be\label{eq:G-M} 
G - M= - G(H+m)M= - M(H+m)G,
\ee
and the following complex cumulant expansion formula. We adopt the form stated in \cite[Lemma 7.1]{He:2017wm}.
	\begin{lemma}(Complex cumulant expansion) \label{lem:complex_cumu}
		Let $h$ be a complex random variable with all its moments exist. The $(p,q)$-cumulant of $h$ is defined as
		$$
		\mathcal{C}^{(p,q)}(h)\deq (-\ii)^{p+q} \cdot \left(\frac{\partial^{p+q}}{\partial {s^p} \partial {t^q}} \log \E e^{\mathrm{i}sh+\mathrm{i}t\bar{h}}\right) \bigg{|}_{s=t=0}\,.
		$$
		Let $f:\mathbb C^2 \to \C$ be a smooth function, and we denote its holomorphic  derivatives by
		$$
		f^{(p,q)}(z_1,z_2)\deq \frac{\partial^{p+q}}{\partial z_1^p \partial z_2^q} f(z_1,z_2)\,.
		$$ Then, for any fixed $l \in \N$, we have
		\begin{equation} \label{5.16}
		\E f(h,\bar{h})\bar{h}=\sum\limits_{p+q=0}^l \frac{1}{p!\,q!}\mathcal{C}^{(p,q+1)}(h)\E f^{(p,q)}(h,\bar{h}) + R_{l+1}\,,
		\end{equation}
		given all integrals in (\ref{5.16}) exist. Here, $R_{l+1}$ is the remainder term depending on $f$ and $h$, and for any $\tau>0$, we have the estimate
		$$
			\begin{aligned}
			R_{l+1}=&\, \OO(1)\cdot \E \big|h^{l+2}\mathbf{1}_{\{|h|>N^{\tau-1/2}\}}\big|\cdot \max\limits_{p+q=l+1}\big\| f^{(p,q)}(z,\bar{z})\big\|_{\infty} \\
			+&\,\OO(1) \cdot \E |h|^{l+2} \cdot \max\limits_{p+q=l+1}\big\| f^{(p,q)}(z,\bar{z})\cdot \mathbf{1}_{\{|z|\le N^{\tau-1/2}\}}\big\|_{\infty}\,.
			\end{aligned}
		$$
	\end{lemma}

With assumptions \eqref{eq:meanvar}, \eqref{eq:meanvar2}, and \eqref{eq:highmoment}, we can show that for $i, j\in \cI$,
$$ \mathcal{C}^{(0,1)}(H_{ij})=\mathcal{C}^{(1,0)}(H_{ij})=0,\quad \mathcal{C}^{(1,1)}(H_{ij}) =s_{ij},\quad \mathcal{C}^{(0,2)}(H_{ij})=\mathcal{C}^{(2,0)}(H_{ij})=s_{ij}\delta_{ij},$$
and that for any fixed $p,q\in \N$ with $p+q\ge 3$, there exists a constant $C>0$ such that
\be\label{eq:bdd_cumu}
\max_{i,j\in \cI}|\mathcal{C}^{(p,q)}(H_{ij})| \le \left(CN\right)^{-(p+q)/2}.
\ee

We also adopt the following notation from \cite[equation (42)]{cipolloni-erdos2021}.

\begin{definition}\label{def:underline}
Suppose that $f$ and $g$ are matrix-valued functions. Define
\be
\underline{g(H)Hf(H)} := g(H)Hf(H)-\widetilde{\bbe}g(H)\widetilde{H}(\partial_{\widetilde{H}}f)(H) - \widetilde{\bbe}(\partial_{\widetilde{H}}g)(H)\widetilde{H}f(H),
\ee
where $\widetilde{H}$ is an indepdent copy of $H$, $\widetilde{\bbe}$ denotes the partial expectation with respect to $\wt H$, and $(\partial_{\widetilde{H}}f)(H)$ denotes the directional derivative of the function $f$ in the direction $\wt H$ at the point $H$, i.e.,
\be
[(\partial_{\widetilde{H}}f)(H)]_{xy} = (\widetilde{H}\cdot \nabla f(H))_{xy} := \sum_{\alpha,\beta\in \cI}\widetilde{H}_{\alpha\beta}\frac{\partial f(H)_{xy}}{\partial H_{\alpha\beta}}.
\ee
\end{definition}

The terms subtracted from $g(H)Hf(H)$ are precisely the second-order term in the cumulant expansion. In particular, if all entries of $H$ are Gaussian, we have $\bbe\underline{g(H)Hf(H)} = 0$. Moreover, if we take $g(H) = I$ and $f(H) = G$, we have that
\be\label{eq:HGline}
\underline{HG} = HG + \widetilde{\bbe}[\widetilde{H}G\widetilde{H}]G,\quad \text{with}\quad \widetilde{\bbe}[\widetilde{H}G\widetilde{H}] = \sum_{a=1}^D D\langle GE_a\rangle E_a.
\ee

First, we prove the estimate \eqref{largG_Ent}. Here, we restrict our consideration to the case where $z_1=z$ and $z_2=\bar z$; the case where $z_1=z_2\in \{z,\bar z\}$ is simpler and can be treated in a similar manner. 
For simplicity, we drop the subscripts and abbreviate $\wh M\equiv \wh M_{(1,2)}$, $L\equiv L_{(1,2)}$, and $K=K_{(1,2)}$. 
Using \eqref{eq:G-M} and \eqref{eq:HGline}, we write that 
\be\label{basGM}G - M = -M(m + H)G = -M\underline{HG} + M(\widetilde{\bbe}[\widetilde{H}G\widetilde{H}] - m)G.
\ee
Applying it to $G_2$ and multiplying by $G_1E_a$ from the left gives 
\[G_1E_aG_2  = G_1E_aM_2  - G_1E_aM_2\underline{HG_2}  + G_1E_aM_2(\widetilde{\bbe}[\widetilde{H}G_2\widetilde{H}] - m_2)G_2 .\]
On the other hand, by definition,
\begin{align}\label{basGM2}
\underline{G_1E_aM_2HG_2 }  
&= G_1E_aM_2\underline{HG_2}  + G_1\widetilde{\bbe}[\widetilde{H}G_1E_aM_2\widetilde{H}]G_2 ,
\end{align} 
so we have that
\[G_1E_aG_2  = G_1E_aM_2  - \underline{G_1E_aM_2HG_2  } + G_1\widetilde{\bbe}[\widetilde{H}G_1E_aM_2\widetilde{H}]G_2  + G_1E_aM_2(\widetilde{\bbe}[\widetilde{H}G_2\widetilde{H}] - m_2)G_2 .\] 
Now, multiplying $E_b$ from the right and taking the trace, we get that 
\begin{align}
\begin{split}
L_{ab} = &\,D\langle G_1E_aM_2E_b\rangle - D\langle \underline{G_1E_aM_2HG_2E_b}\rangle \\
&+ D \sum_{x=1}^D \langle G_1E_aM_2E_x\rangle L_{xb} + D^2\sum_{x=1}^D \langle (G_2 - M_2) E_x\rangle \langle G_1E_aM_2E_xG_2E_b\rangle.
\end{split}\label{eq:aftertrace}
\end{align} 
Then, using the averaged local law \eqref{eq:aver_local} and Lemma \ref{appen1}, we can estimate the RHS as
\begin{align*}
L_{ab} = &\,\wh M_{ab} - D\langle \underline{G_1E_aM_2HG_2E_b}\rangle + (\wh M L)_{ab}  +\opr{\eta^{-1}(N\eta)^{-1}},\quad a,b\in \qqD.
\end{align*} 
With the bound on $\|(1-\wh M)^{-1}\|$ in \eqref{1-M} below, we get from the above equation that
\begin{align}\label{eq_Lab} 
L_{ab} = &K _{ab} - D\sum_{x=1}^D\big(1-\wh M\big)^{-1}_{ax}\langle \underline{G_1E_xM_2HG_2E_b}\rangle   +\opr{\eta^{-2}(N\eta)^{-1}}.
\end{align} 
To conclude \eqref{largG_Ent}, we only need to prove that given any $a,b\in \qqD$, there is
\begin{equation}\nonumber
Z:=\langle \underline{G_1E_aM_2HG_2E_b}\rangle \prec  (N\eta^2)^{-1}.
\end{equation}
By Markov's inequality, it suffices to prove the following high-moment bound: for any fixed $p\in 2\N $ with $p\ge 4$,
\be
\E |Z|^p \prec (N\eta^2)^{-p } .\label{eq:i_ii_cumulant}
\ee

To prove \eqref{eq:i_ii_cumulant}, we use the cumulant expansion formula in \Cref{lem:complex_cumu}:
\be\label{eq:Zcumu}
\begin{aligned}
\E |Z|^p &= \E \langle \underline{G_1E_aM_2HG_2E_b}\rangle \overline{Z}|Z|^{p-2} = \frac{1}{DN^2}\sum_{x=1}^D\sum_{ \alpha,\beta \in \cI_x} \E (G_2E_bG_1E_aM_2)_{\beta\alpha}\partial_{\beta\alpha}(\overline{Z}|Z|^{p-2}) 
\\
&\hspace{0.5cm}+ \sum_{ m+n=2}^l
   \frac{1}{DN} \sum_{x=1}^D\sum_{\alpha, \beta\in \cI_x}
    \frac{1}{m!n!} \mathcal{C}^{(n,m+1)}_{\beta\alpha}  \E\partial_{\alpha\beta}^m\partial_{\beta\alpha}^{n}
    \big((G_2E_bG_1E_aM_2)_{\beta\alpha}\overline{Z}|Z|^{p-2}\big) 
    + \mathcal{R}_{l+1}
    \\
&=: {X_1} + \sum_{m+n=2}^l X_{m,n} + \mathcal{R}_{l+1},
\end{aligned}
\ee
where $\partial_{\beta\al}\equiv \partial_{h_{\beta\al}}$, $\mathcal{C}^{(n,m)}_{\beta\al}=\OO(N^{-(n+m)/2})$ denotes the $(n,m)$-cumulant of $h_{\beta\al}$, and $\cal R_{l+1}$ denotes the error term involving the $(l+2)$-th order moment of $h_{\beta\al}$. By definitions, we can write \(Z\) as 
\begin{align}\label{eq:Z}
Z &= \langle G_1E_aM_2HG_2E_b\rangle + D\sum_{x=1}^D \langle G_2E_x\rangle \langle G_1E_aM_2E_xG_2E_b\rangle + D\sum_{x=1}^D \langle G_1E_xG_2E_b\rangle\langle G_1E_aM_2E_x\rangle .
\end{align}
With this formula and $HG_2=I-(\Lambda-z)G_2$, we can check with direct calculations that
\be\label{eq:X1}
 X_1=N^{-2} \E \left(\cal X_{11}|Z|^{2}+\cal X_{12}(\overline Z)^2\right)|Z|^{p-4},
\ee
where $\cal X_{11}$ and $\cal X_{12}$ are linear combinations of $\OO(1)$ many terms taking the following forms with coefficients of order $\OO(1)$: 
\begin{itemize}
    \item $\langle \prod_{i=1}^k G_{j_i}B_{i}\rangle $ for some $k\le 5$, $j_i \in \{1,2\}$, and deterministic matrices $B_i$ with $\|B_i\|=\OO(1)$;
    \item $\langle \prod_{i=1}^k G_{j_i}B_{i}\rangle $$\langle \prod_{i=1}^{k'} G_{j_{i'}}B_{i'}\rangle $ for some $k,k'\ge 1$ with $k+k'\le 6$, $j_i,j_{i'}\in \{1,2\}$, and deterministic matrices $B_i, B_{i'}$ with $\|B_i\|=\OO(1)$, $\|B_{i'}\|=\OO(1)$.
\end{itemize}
Then, using \Cref{appen1}, we can bound \eqref{eq:X1} by 
\be\label{eq:X1_2}
X_1 \prec   (N\eta^2)^{-2} \E |Z|^{p-2} .
\ee

Next, we consider the terms $X_{m,n}$ with $2\le m+n\le l$. Using \eqref{eq:Z}, the anisotropic local law \eqref{eq:aniso_local}, and \Cref{appen1},  we can check that for any fixed $m,n\in \N$ with $m+n\ge 1$, 
\be\label{pq}\partial_{\alpha\beta}^m\partial_{\beta\alpha}^n
\big( G_2E_bG_1E_aM_2\big)_{\beta\alpha}\prec \frac{1}{\eta}
,\quad \quad 
\partial_{\alpha\beta}^m\partial_{\beta\alpha}^n
Z\prec \frac{1}{N \eta^2} .
\ee
With these estimates, we obtain that 
\be\nonumber
\partial_{\alpha\beta}^m\partial_{\beta\alpha}^n
\left(\left( G_2E_bG_1E_aM_2\right)_{\beta\alpha}\overline Z |Z|^{p-2}
\right)\prec \sum_{q=0}^{p-1}\frac{1}{\eta}\left(\frac{1}{N\eta^2}\right)^q|Z|^{p-1-q},
 \ee
which gives that for any $3\le m+n\le l$, 
\be\label{eq:Xmn>2}
X_{m,n}\prec
\sum_{q=1}^p N^{2-({n+m+1})/{2}}\left(\frac{1}{N\eta^2}\right)^q \E|Z|^{p-q} \le \sum_{q=1}^p \left(\frac{1}{N\eta^2}\right)^q \E|Z|^{p-q}.
\ee
It remains to consider the $m+n=2$ case: 
 \begin{itemize}
     \item[(i)] 
     If no derivative or both derivatives act on $(G_2E_bG_1E_aM_2)_{\beta\alpha}$, using the anisotropic local law \eqref{eq:aniso_local} and \Cref{appen1}, we can check directly that
    $$ \sum_{\al,\beta\in \cI_x} \left|
\partial_{\alpha\beta}^m\partial_{\beta\alpha}^n
\big( G_2E_bG_1E_aM_2\big)_{\beta\alpha}
\right|^2\prec {\eta^{-3}{N}}.$$
Together with the Cauchy-Schwarz inequality, it implies that
   \be\label{deronG}
    \sum_{\al,\beta\in \cI_x} 
\left|
\partial_{\alpha\beta}^m\partial_{\beta\alpha}^n
\big( G_2E_bG_1E_aM_2\big)_{\beta\alpha}
\right| \prec {\eta^{-3/2}}{N^{3/2}}.
\ee

\item[(ii)] If one derivative acts on $(G_2E_bG_1E_aM_2)_{\beta\alpha}$ and the other on $Z$ or $\overline{Z}$, then we can check that
\be\label{eq:dervZZZ}
\left|\partial_{\alpha\beta}\big( G_2E_bG_1E_aM_2\big)_{\beta\alpha}
\right|+\left|
\partial_{\beta\alpha}\big( G_2E_bG_1E_aM_2\big)_{\beta\alpha}
\right|\prec \eta^{-1},\quad 
\sum_{\al,\beta} \left|\partial_{\alpha\beta} Z\right|^2 \prec {N}^{-1}\eta^{-5}.
\ee
To see the second estimate, we express $\partial_{\alpha\beta} Z$ as 
\begin{align*}
\partial_{\alpha\beta} Z=\frac{1}{DN   }\left(\cal Z^{(1)}_{\beta\al}+\cal Z^{(2)}_{\beta\al}\right),
\end{align*}
where $\cal Z^{(1)}_{\beta\al}$ comes from the derivative of the first term in \eqref{eq:Z}:
$$\cal Z^{(1)}_{\beta\al}:= - (G_1E_aM_2E_bG_1)_{\beta\al}+ ( G_1E_aM_2(\Lambda-z)G_2E_bG_1)_{\beta\al}+ ( G_2E_bG_1E_aM_2(\Lambda-z)G_2)_{\beta\al},$$
and $\cal Z^{(2)}_{\beta\al}$ comes from the derivative of the remaining two terms. Moreover, $\cal Z^{(2)}_{\beta\al}$ is a sum of $6$ terms with coefficients of order $\OO(1)$, taking the forms: 
$$D\sum_{x=1}^D( G_{j_1}B_{1}G_{j_2}B_2G_{j_1})_{\beta\al}\avg{G_{j_3} B_3},\quad \text{or}\quad D\sum_{x=1}^D\langle G_{j_1}B_{1}G_{j_2}B_{2}\rangle (G_{j_3}B_3 G_{j_3})_{\beta\al}, $$
where $j_i \in \{1,2\}$ and $B_i$ are deterministic matrices such that $\|B_i\|=\OO(1)$ for $i\in \{1,2,3\}$. By using \Cref{appen1}, we can establish the bounds:
\begin{align*}
& \sum_{\al,\beta}|\cal Z^{(1)}_{\beta\al}|^2\prec N\eta^{-5},\quad \sum_{\al,\beta}|( G_{j_1}B_{1}G_{j_2}B_2G_{j_1})_{\beta\al}\avg{G_{j_3} B_3}|^2\prec N\eta^{-5},\\
& \sum_{\al,\beta}|\langle G_{j_1}B_{1}G_{j_2}B_{2}\rangle (G_{j_3}B_3 G_{j_3})_{\beta\al}|^2\prec \eta^{-2} \sum_{\al,\beta}|(G_{j_3}B_3 G_{j_3})_{\beta\al}|^2 \prec N\eta^{-5} .   
\end{align*}
These conclude the second estimate in \eqref{eq:dervZZZ}.
Now, combining \eqref{eq:dervZZZ} with the Cauchy-Schwarz inequality, we obtain that
\be\label{deronG2}
\sum_{\al,\beta\in \cI_x}
\left(\left| \partial_{\alpha\beta}\big( G_2E_bG_1E_aM_2\big)_{\beta\alpha} \right|+ \left| \partial_{\beta\alpha}\big( G_2E_bG_1E_aM_2\big)_{\beta\alpha} \right|\right)
\left(\left|\partial_{\alpha\beta} Z\right|+\left|\partial_{\beta\alpha} Z\right|\right) \prec \frac{N^{1/2}}{\eta^{7/2}}.
\ee 
 \end{itemize}
Therefore, combining \eqref{pq}, \eqref{deronG}, and \eqref{deronG2}, we obtain that for $m+n=2$, 
\be\label{eq:Xmn=2}
X_{m,n}\prec 
\sum_{q=1}^3 \left(\frac{1}{N\eta^2}\right)^q \E|Z|^{p-q}.
\ee

Finally, for the remainder term $\cal R_{l+1}$, with the argument in e.g., \cite[Lemma 4.6]{He:2017wm}, it is easy to prove the following estimate: as long as $l$ is sufficiently large depending on $\tau$ and $p$, there is 
\be\label{eq:X_remainder}
\cal R_{l+1} \prec N^{-p}.
\ee
Since the argument is standard, we omit the details here. 

Now, plugging \eqref{eq:X1_2}, \eqref{eq:Xmn>2}, \eqref{eq:Xmn=2}, and \eqref{eq:X_remainder} into \eqref{eq:Zcumu} yields that for any small constant $\e>0$,
\[\bbe|Z|^p \prec  \sum_{q=1}^{p}\left(N^{-1}\eta^{-2}\right)^q \bbe |Z|^{p-q} \le \sum_{q=1}^{p}\left(N^{-1}\eta^{-2}\right)^q \left(\bbe |Z|^p\right)^{\frac{p-q}{p}} \prec \left(N^{-1+\e}\eta^{-2}\right)^p+ N^{-\e} \bbe |Z|^p,\] 
where we have applied H\"older's and Young's inequalities in the second and third steps. 
This implies \eqref{eq:i_ii_cumulant} since $\e$ is arbitrary, which further completes the proof of \eqref{largG_Ent}.

\begin{remark}
In the above proof, the matrix $E_b$ can be replaced by any deterministic $B$ with $\|B\|=\OO(1)$. Thus, the above arguments lead to the following estimate that is similar to \eqref{largG_Ent}: 
\be\label{largG_Ent_BB}
\langle G_1E_{a}G_2B\rangle -\sum_{b\in \qqD} \big(1-\wh M_{(1,2)}\big)^{-1}_{ab} \langle M_1 E_{b} M_2 B\rangle \prec N^{-1}\eta^{-3}.
\ee
Furthermore, in the case where $z_1=z_2\in \{z,\bar z\}$, since $\|(1-\wh M)^{-1}\|\lesssim 1$ due to \eqref{1-M-2}, the error term in \eqref{eq_Lab} becomes $\opr{N^{-1}\eta^{-2}}$. Thus, we get a better bound 
\be\label{largG_Ent_same}
\|L_{(1,2 )}-K_{(1,2)}\|_{2}\prec N^{-1 }\eta^{-2}\quad \text{when}\quad z_1=z_2.
\ee
\end{remark}

Next, we prove the estimate \eqref{largG_3G}. Note that two of $z_1, z_2, z_3$ must take the same value. We assume without loss of generality that $z_1 = z_3\in \{z,\bar z\}$. Then, using \eqref{eq:G-M} and \Cref{def:underline}, we write that 
\begin{align}
G_1E_{a_1}G_2E_{a_2}G_3E_{a_3} &= G_1E_{a_1}G_2E_{a_2}M_3E_{a_3} - \underline{G_1E_{a_1}G_2E_{a_2}M_3HG_3E_{a_3}} \nonumber\\
&+ G_1E_{a_1}G_2E_{a_2}M_3\widetilde{\bbe}[\widetilde{H}(G_3 - M_3)\widetilde{H}]G_3E_{a_3}\nonumber\\
&+ G_1\widetilde{\bbe}[\widetilde{H}G_1E_{a_1}G_2E_{a_2}M_3\widetilde{H}]G_3E_{a_3} + G_1E_{a_1}G_2\widetilde{\bbe}[\widetilde{H}G_2E_{a_2}M_3\widetilde{H}]G_3E_{a_3}. \nonumber
\end{align} 
Taking the trace and applying \eqref{eq:aver_local} and \Cref{appen1}, we derive from this equation that 
\begin{align}
(L_{(1,2,3)})_{a_1a_2a_3}&= D\langle G_1E_{a_1}G_2E_{a_2}M_3E_{a_3}\rangle - D\langle\underline{G_1E_{a_1}G_2E_{a_2}M_3HG_3E_{a_3}}\rangle  \nonumber\\
&+ \sum_{x=1}^D (\wh{M}_{(2,3)})_{a_2x}(L_{(1,2,3)})_{a_1 x a_3} 
+ \sum_{x=1}^D D\langle G_1E_{a_1}G_2E_{a_2}M_3E_x\rangle (K_{(3,1)})_{a_3 x}\nonumber\\
&+ \sum_{x=1}^D D\langle G_1E_{a_1}G_2E_{a_2}M_3E_x\rangle\left(L_{(3,1)}  - K_{(3,1)}\right)_{a_3 x} + \opr{\eta^{-2}(N\eta)^{-1}} .\label{eq:3GGG}
\end{align} 
Applying \eqref{largG_Ent_same} to the last sum on the RHS and using \Cref{appen1}, we get that 
\begin{align*}
    D\sum_{x=1}^D \langle G_1E_{a_1}G_2E_{a_2}M_3E_x\rangle \left(L_{(3,1)}  - (K_{(3,1)})\right)_{xa_3} \prec N^{-1}\eta^{-3} .
\end{align*}
Applying \eqref{largG_Ent_BB} to the first term on the RHS of \eqref{eq:3GGG}, we get that 
\begin{align}\label{eq:3GGG2}
D\langle G_1E_{a_1}G_2E_{a_2}M_3E_{a_3}\rangle &= \sum_x (I - \wh{M}_{(1, 2)})^{-1}_{a_1x}D\langle M_1E_xM_2E_{a_2}ME_{a_3}\rangle + \opr{N^{-1}\eta^{-3}}.
\end{align}
For simplicity, we will abbreviate the first term on the RHS of \eqref{eq:3GGG2} as $\nu_{a_1a_2a_3}$.
The equation \eqref{eq:3GGG2} also applies to the fourth term on the RHS of \eqref{eq:3GGG} and gives that
$$ D\sum_{x=1}^D \langle G_1E_{a_1}G_2E_{a_2}M_3E_x\rangle(K_{(3,1)})_{a_3 x} =  \sum_{x=1}^D \nu_{a_1a_2 x}(K_{(3,1)})_{a_3 x} + \opr{N^{-1}\eta^{-3}},$$
where we used that $\|K_{(3,1)}\|\lesssim 1$ by \eqref{1-M-2} below. Finally, we can prove that
$$\langle\underline{G_1E_{a_1}G_2E_{a_2}M_3HG_3E_{a_3}}\rangle \prec N^{-1}\eta^{-3}$$ 
using a similar argument as that for the proof of \eqref{eq:i_ii_cumulant} with cumulant expansions. We omit the details here. 
In sum, we have derived from \eqref{eq:3GGG} that
\begin{align*}
(L_{(1,2,3)})_{a_1a_2a_3} &=\nu_{a_1a_2a_3} +\sum_{x=1}^D \nu_{a_1a_2 x}(K_{(3,1)})_{a_3 x}  + \sum_{x=1}^D (\wh{M}_{(2,3)})_{a_2 x}(L_{(1,2,3)})_{a_1x a_3} + \opr{N^{-1}\eta^{-3}}.
\end{align*} 
Now, solving for $L_{(1,2,3)}$, using \eqref{1-M} and that $I + K_{(3,1)} = (I - \wh{M}_{(3,1)})^{-1}$, the estimate \eqref{largG_3G} follows.

\subsection{Proof of \Cref{lemma_G_large_2}}\label{sec_hardLK}

For any $z_1, z_2\in \{z, \overline{z}\}$, we abbreviate that 
$$\wh M\equiv \wh M_{(1,2)},\quad L\equiv L_{(1,2)},\quad K\equiv K_{(1,2)},\quad \text{and}\quad \wt M\equiv \wh M_{(2,1)},\quad \wt L\equiv L_{(2,1)}, \quad \wt K\equiv K_{(2,1)}.$$
Moreover, given any deterministic matrix $B\in \C^{DN\times DN}$, we denote 
$$L_{ab}(B):=D\langle G_1 E_a G_2 E_b B\rangle,\quad K_{ab}(B):=\sum_x (1-\wh M)^{-1}_{ax}D\langle M_1 E_x M_2 E_b B\rangle. $$
Similarly, we define $\wt L_{ab}(B)$ and $\wt K_{ab}(B)$ by exchanging $1$ and $2$. Now, taking the expectation of \eqref{eq:aftertrace}, we obtain that   
\begin{align}
\E L_{ab} &= \wh{M}_{ab} + D\E\langle (G_1-M_1)E_aM_2E_b\rangle - D\E\langle \underline{G_1E_aM_2HG_2E_b}\rangle + \sum_{x=1}^D \wh M_{ax} \E L_{xb} \nonumber\\
&\quad + D\sum_{x=1}^D \E\langle (G_1 - M_1)E_aM_2E_x\rangle L_{xb} + D^2\sum_{x=1}^D \E\langle (G_2-M_2) E_x\rangle \langle G_1E_aM_2E_xG_2E_b\rangle \nonumber\\
&= \wh{M}_{ab} + D\E\langle (G_1-M_1)E_aM_2E_b\rangle - D\E\langle \underline{G_1E_aM_2HG_2E_b}\rangle + \sum_{x=1}^D \wh M_{ax} \E L_{xb} \label{ELK}\\
&\quad + D\sum_{x=1}^D \E\langle (G_1 - M_1)E_aM_2E_x\rangle K_{xb} + D^2\sum_{x=1}^D \E\langle (G_2-M_2) E_x\rangle \wt K_{ba}(M_2E_x)+\OO_\prec (N^{-2}\eta^{-4}),\nonumber
\end{align} 
where we used the averaged local law \eqref{eq:aver_local} and the estimates \eqref{largG_Ent} and \eqref{largG_Ent_BB} in the second step. The proof of \Cref{lemma_G_large_2} is based on the next two lemmas, Lemmas \ref{lem:fourthcumu1} and \ref{lem:fourthcumu2}, which give precise estimates of the expectations on the RHS of \eqref{ELK}.

\begin{lemma}\label{lem:fourthcumu1}
Under the setting of \Cref{lemma_G_large_2}, we have that     
 \begin{align}
&- D\E\langle \underline{G_1E_aM_2HG_2E_b}\rangle=\opr{\eta^{-2} N^{-3/2}+\eta^{-3}N^{-2}} \nonumber\\ &+\frac{D\kappa^{(2, 2)}}{N}\sum_{x=1}^D \left[  { \langle \diag(M_2)^2E_x\rangle }\wt K_{ba}(M_2\diag(M_2)E_x) +\langle M_1 \diag(M_2)E_x\rangle  \wt K_{bx}(\diag(M_1E_aM_2))\right]\nonumber\\
&+\frac{D\kappa^{(2, 2)}}{N}\sum_{x=1}^D   \left[\langle M_1E_aM_2 \diag(M_1)E_x\rangle  \wt K_{bx}(\diag(M_1))+\langle M_1E_aM_2 \diag(M_2)E_x\rangle  \wt K_{bx}(\diag(M_2))\right]  ,\label{eq:G1HG2}
    \end{align}
where $\kappa^{(2, 2)}$ is the normalized $(2,2)$-cumulant of $h_{12}$ defined as $\kappa^{(2, 2)}:=N^2\cal C_{12}^{(2,2)}$, and $\diag(B)$ is the diagonal matrix consisting of the diagonal entries of the given matrix $B$.  
\end{lemma}
\begin{proof}
We again use the cumulant expansion in \Cref{lem:complex_cumu}: 
\[-D\E\langle \underline{G_1E_aM_2HG_2E_b}\rangle = \frac{-1}{N}\sum_{m+n=2}^{l} \sum_{\alpha,\beta}\frac{1}{m!}\frac{1}{n!}\mathcal{C}^{(n, m+1)}_{\beta\alpha}\E \partial_{\alpha\beta}^m\partial_{\beta\alpha}^n(G_2E_bG_1E_aM_2)_{\beta\alpha} + \mathcal{R}'_{l+1},
\] 
where we choose $l$ large enough such that the remainder term satisfies $\mathcal{R}'_{l+1} \prec N^{-2}$. Using Lemma \ref{appen1}, it is straightforward to check that $\partial_{\alpha\beta}^m\partial_{\beta\alpha}^n(G_2E_bG_1E_aM_2)_{\beta\alpha} \prec \eta^{-1}$. We thus obtain the rough bound
\[Y_{m,n} := \frac{-1}{N}\sum_{\alpha,\beta}\frac{1}{m!}\frac{1}{n!}\mathcal{C}^{(n, m+1)}_{\beta\alpha}\E \partial_{\alpha\beta}^m\partial_{\beta\alpha}^n(G_2E_bG_1E_aM_2)_{\beta\alpha} \prec N^{-(m+n-1)/2}\eta^{-1},\] 
which is good enough when $m+n \geq 4$. It remains to deal with the cases $m+n \in \{ 2, 3\}$. 

For the $m+n=2$ case, we take $m=n=1$ as an example:
\begin{align*}
	Y_{1,1} =& -\frac{1}{N}\sum_{\alpha,\beta}\mathcal{C}^{(1, 2)}_{\beta\alpha}\E\left[(G_2)_{\beta\beta}(G_2)_{\beta\alpha}(G_2E_bG_1E_aM_2)_{\alpha\alpha} + (G_2)_{\beta\beta}(G_2)_{\al\alpha}(G_2E_bG_1E_aM_2)_{\beta\alpha}\right]\\
	&-\frac{1}{N}\sum_{\alpha,\beta}\mathcal{C}^{(1, 2)}_{\beta\alpha}\E\left[(G_2)_{\beta\beta} (G_2E_bG_1)_{\al\al}(G_1E_aM_2)_{\beta\alpha} +  (G_2)_{\beta\al}(G_2E_bG_1)_{\beta\beta}(G_1E_aM_2)_{\al\alpha}\right]\\
	&-\frac{1}{N}\sum_{\alpha,\beta}\mathcal{C}^{(1, 2)}_{\beta\alpha}\E\left[(G_1)_{\beta\beta} (G_2E_bG_1)_{\beta\al}(G_1E_aM_2)_{\al\alpha} +  (G_1)_{\al\al}(G_2E_bG_1)_{\beta\beta}(G_1E_aM_2)_{\beta\alpha}\right].
\end{align*}
For the first term on the RHS, we write it as
\be\label{Y11_first}\frac{1}{N^{5/2}}\E \sum_{x=1}^D\sum_{\alpha,\beta\in \cI_x}\left(\kappa^{(1,2)} + u\delta_{\al\beta}\right) (G_2)_{\beta\beta}(G_2)_{\beta\alpha}(G_2E_bG_1E_aM_2)_{\alpha\alpha}
\ee
where we used that  $\mathcal{C}^{(1,2)}_{\beta\alpha}=N^{-3/2}(\kappa^{(1,2)}+u\delta_{\al\beta})$ with $\kappa^{(1,2)}:=N^{3/2}\mathcal{C}^{(1,2)}_{12}$ and $u$ being some fixed number. Then, using $(G_2)_{\beta\beta}\prec 1$ and $(G_2E_bG_1E_aM_2)_{\alpha\alpha}\prec \eta^{-1}$ by \Cref{appen1}, applying the trivial bound $\|G_2\|\le \eta^{-1}$ gives that
$$
\eqref{Y11_first} \prec \frac{1}{N^{5/2}}\cdot \frac{N}{\eta}+ \frac{1}{N^{5/2}} \cdot \frac{1}{\eta} \frac{N}{\eta} \le 2 \eta^{-2}N^{-3/2}.
$$
The other terms in $Y_{1,1}$ and the terms in the $(m,n)\in \{(0, 2), (2,0)\}$ cases can be handled in a similar way.

For the $m + n = 3$ case, through direct calculations, we see that when $(m,n)\in \{(0, 3), (2,1), (3,0)\}$, the corresponding $Y_{m,n}$ consists of terms of the form
\be\label{eq:414_mn3_otherterms}
J_1=\frac{1}{N^{3}}\sum_{x=1}^D\sum_{\al,\beta\in \cI_x}(G_sB)_{ij}\cal F_1  \ \ \text{ or }\ \ J_2=\frac{1}{N^{3}}\sum_{x=1}^D\sum_{\al,\beta\in \cI_x}(G_2E_bG_1B)_{ij}\cal F_2 ,
\ee 
where $s\in \{1,2\},$ $ (i,j)\in \{(\al,\beta),(\beta,\al)\}$, $B$ is some deterministic matrix with $\|B\|=\OO(1)$, and $\cal F_1$ and $\cal F_2$ are expressions satisfying  $\cal F_1=\opr{\eta^{-1}}$ and $\cal F_2=\opr{1}$. Using Cauchy-Schwarz inequality and \Cref{appen1}, we get that
$$\sum_{\beta\in \cI_x}|(G_sB)_{ij}| \prec \eta^{-1/2}N^{1/2},\quad \sum_{\beta\in \cI_x}|(G_2E_bG_1B)_{ij}| \prec \eta^{-3/2}N^{1/2}, $$
with which we can bound $J_1$ and $J_2$ by $ J_1 + J_2 \prec (N\eta)^{-3/2}.$ 

Finally, we consider the remaining case with $m = 1$ and $n = 2$. In this case, some terms in $Y_{1,2}$ will be of the form \eqref{eq:414_mn3_otherterms}, so they can be bounded by $\OO_\prec((N\eta)^{-3/2})$. However, $Y_{1,2}$ also contains some leading terms that contribute to \eqref{eq:G1HG2}. We can expand $Y_{1,2}$ as
\begin{align*}
	 Y_{1,2}&=-\frac{\kappa^{(2, 2)} }{2N^3}\sum_{x=1}^D\sum_{\alpha,\beta\in \cI_x}\E \partial_{\alpha\beta}\partial_{\beta\alpha}^2(G_2E_bG_1E_aM_2)_{\beta\alpha}\\
	 &=\frac{ \kappa^{(2, 2)} }{2N^3}\sum_{x=1}^D\sum_{\alpha,\beta\in \cI_x}\E \partial_{\alpha\beta}\partial_{\beta\alpha}\left[(G_2)_{\beta\beta}(G_2E_bG_1E_aM_2)_{\al\alpha}+(G_2E_bG_1)_{\beta\beta}(G_1E_aM_2)_{\al\alpha}\right]\\
	 &=-\frac{\kappa^{(2, 2)} }{N^3}\sum_{x=1}^D\sum_{\alpha,\beta\in \cI_x}\E \partial_{\alpha\beta} \left[(G_2)_{\beta\beta}(G_2)_{\al\beta}(G_2E_bG_1E_aM_2)_{\al\alpha}+(G_2)_{\beta\beta}(G_2E_bG_1)_{\al\beta}(G_1E_aM_2)_{\al\alpha}\right]\\
	 &\quad -\frac{ \kappa^{(2, 2)} }{N^3}\sum_{x=1}^D\sum_{\alpha,\beta\in \cI_x}\E \partial_{\alpha\beta}\left[(G_1)_{\al\beta}(G_2E_bG_1)_{\beta\beta}(G_1E_aM_2)_{\al\alpha}\right].
\end{align*}
When $\partial_{\alpha\beta}$ acts on the off-diagonal terms $(G_2)_{\al\beta}$, $(G_2E_bG_1)_{\al\beta}$, and $(G_1)_{\al\beta}$, it yields 4 leading terms consisting of diagonal terms only. The partial derivative $\partial_{\alpha\beta}$ acting on other terms gives errors, each containing two off-diagonal entries of the form $(G_{s} B)_{ij}$ and $(G_{s}E_\gamma G_{s'}B)_{ij}$, where $s,s'\in \{1,2\}$, $ (i,j)\in \{(\al,\beta),(\beta,\al)\}$, $\gamma\in \{a,b\}$, and $B$ is some deterministic matrix with $\|B_i\|=\OO(1)$. By \Cref{appen1}, we have 
$\sum_{\beta\in \cI_x} |(G_{s} B)_{ij}|^2 \prec \eta^{-1}$ and $ \sum_{\beta\in \cI_x}|(G_{s}E_\gamma G_{s'}B)_{ij}|^2 \prec \eta^{-3}.$
Combining these bounds with Cauchy-Schwarz inequality, we can bound the error terms by  $\OO_\prec((N\eta)^{-2}) $. In sum, we obtain that 
\begin{align*}
Y_{1,2}&=\frac{\kappa^{(2, 2)}}{N^3}\sum_{x=1}^D\sum_{\alpha,\beta\in \cI_x}\E(G_2)_{\beta\beta}\left[ (G_2)_{\beta\beta}  (G_2)_{\al\al} (G_2E_bG_1E_aM_2)_{\al\alpha}+   (G_2)_{\al\al} (G_2E_bG_1)_{\beta\beta} (G_1E_aM_2)_{\al\alpha}\right]\\
&+\frac{\kappa^{(2, 2)}}{N^3}\sum_{x=1}^D\sum_{\alpha,\beta\in \cI_x}\E(G_1)_{\beta\beta}\left[ (G_2)_{\beta\beta} (G_2E_bG_1)_{\al\al}  (G_1E_aM_2)_{\al\al}+ (G_1)_{\al\al}(G_2E_bG_1)_{\beta\beta}  (G_1E_aM_2)_{\al\al}\right]\\
& + \OO_\prec((N\eta)^{-2}).
\end{align*}
Then, using the anisotropic local law \eqref{eq:aniso_local}, \Cref{appen1}, and the estimate \eqref{largG_Ent_BB}, we get from the above expression that 
\begin{align*} 
	Y_{1,2}&=\frac{\kappa^{(2, 2)}}{N^2}\sum_{x=1}^D \bigg\{ \sum_{\beta\in \cI_x} \left[(M_2)_{\beta\beta}\right]^2   \wt K_{ba}(M_2\diag(M_2)E_x)  + \sum_{\al\in \cI_x}  (M_2)_{\al\al} (M_1E_aM_2)_{\al\alpha} \wt K_{bx}(\diag(M_2)) \bigg\} \\
	&+\frac{\kappa^{(2, 2)}}{N^2}\sum_{x=1}^D\bigg\{\sum_{\beta\in \cI_x} (M_1)_{\beta\beta}(M_2)_{\beta\beta}\wt K_{bx}(\diag(M_1E_aM_2))  + \sum_{\al\in \cI_x} (M_1)_{\al\al}(M_1E_aM_2)_{\al\al} \wt K_{bx}(\diag(M_1)) \bigg\}\\
	& + \OO_\prec\big((N\eta)^{-3/2}+\eta^{-3}N^{-2}\big).
\end{align*}
Rewriting this expression concludes \eqref{eq:G1HG2}. 
\end{proof}

\begin{lemma}\label{lem:fourthcumu2}
Under the setting of \Cref{lemma_G_large_2}, let $B$ be an arbitrary deterministic matrix with $\|B\|\le 1$. Then, we have that
    \be\label{eq:Exp_G1-M1}
    \begin{aligned}
    \bbe \langle(G_1 - M_1)B\rangle 
    =&\frac{\kappa^{(2,2)}  \langle \diag(M_1)^2 \rangle }{N}\left[ \langle M_1BM_1\diag(M_1)\rangle   +   \frac{ \langle M_1^2 \diag(M_1)\rangle}{1-\langle M_1^2\rangle } \langle M_1^2B\rangle\right]\\
    & + \OO_\prec\big({\eta^{-1} N^{-3/2}+(N\eta)^{-2}}\big).
    \end{aligned}
    \ee
\end{lemma}
\begin{proof}
Using \eqref{eq:G-M} and \eqref{eq:HGline}, we get that
\be\label{eq:g1b1}\begin{aligned}
	\bbe \langle(G_1 - M_1)B\rangle 
	&= -\bbe \langle M_1\underline{HG_1}B\rangle + D\sum_{x=1}^D \bbe \langle (G_1 -M_1)E_x\rangle\langle M_1E_xG_1B\rangle \\
	&= -\bbe \langle M_1\underline{HG_1}B\rangle + D\sum_{x=1}^D \bbe \langle (G_1 -M_1)E_x\rangle\langle M_1E_xM_1B\rangle  + \opr{(N\eta)^{-2}},
\end{aligned}\ee
where we used the averaged local law \eqref{eq:aver_local} in the second step. 
Taking $B=E_a$ and using \eqref{1-M-2} below, we can solve this equation to get that 
\begin{align}\label{eq:g1b2}	\bbe \langle(G_1 - M_1)E_a\rangle=-  \sum_{x=1}^D (1-\wh M_{(1,1)})^{-1}_{ax} \bbe \langle M_1\underline{HG_1}E_x\rangle + \opr{(N\eta)^{-2}}.
\end{align}
With the cumulant expansion in \Cref{lem:complex_cumu}, using a similar argument as in the proof of \Cref{lem:fourthcumu1}, we can derive that 
\begin{align}
	-\bbe \langle M_1\underline{HG_1}E_x\rangle 
	&=\frac{\kappa^{(2,2)}}{DN^3} \sum_{y=1}^D \sum_{\al,\beta\in \cI_y} \bbe \left[(G_1)_{\beta \beta}\right]^2(G_1)_{\al\al} (G_1E_xM_1)_{\al\al} +  \OO_\prec\big(\eta^{-1} N^{-3/2} \big) \nonumber\\
	&=\frac{\kappa^{(2,2)}}{N} \langle \diag(M_1)^2 \rangle \langle M_1E_xM_1 \diag(M_1)\rangle +  \OO_\prec\big(\eta^{-1} N^{-3/2} \big)\nonumber\\
	&=\frac{\kappa^{(2,2)}}{DN} \langle \diag(M_1)^2 \rangle \langle M_1^2 \diag(M_1)\rangle +  \OO_\prec\big(\eta^{-1} N^{-3/2} \big),\label{eq:G1-M1Ea_b}
\end{align}
where we used the averaged local law \eqref{eq:aver_local} and the block translation symmetry of $M_1$ in the above derivation. A similar argument gives that 
\begin{align}
	-\bbe \langle M_1\underline{HG_1}B\rangle=\frac{\kappa^{(2,2)} }{N}  \langle \diag(M_1)^2 \rangle \langle M_1BM_1\diag(M_1)\rangle  +  \OO_\prec\big(\eta^{-1} N^{-3/2} \big).\label{eq:G1-M1Ea0}
\end{align}

Now, plugging \eqref{eq:G1-M1Ea_b} into \eqref{eq:g1b2} and using 
$$ \sum_{x=1}^D (1-\wh M_{(1,1)})^{-1}_{ax} = \frac{1}{1-\sum_{x=1}^D (M_{(1,1)})_{ax}}=\frac{1}{1-\langle M_1^2\rangle },$$
we obtain that  
\begin{align} 
	\label{eq:g1b3}	
	D\bbe \langle(G_1 - M_1)E_a\rangle=\frac{\kappa^{(2,2)}}{N} \frac{\langle \diag(M_1)^2 \rangle \langle M_1^2 \diag(M_1)\rangle}{1-\langle M_1^2\rangle }  + \OO_\prec\big({\eta^{-1} N^{-3/2}+(N\eta)^{-2}}\big).
\end{align}
Finally, applying \eqref{eq:G1-M1Ea0} and \eqref{eq:g1b3} to \eqref{eq:g1b1},  we conclude \eqref{eq:Exp_G1-M1}.
\end{proof}

Now, we are ready to complete the proof of \Cref{lemma_G_large_2} using Lemmas \ref{lem:fourthcumu1} and \ref{lem:fourthcumu2}.
\begin{proof}[Proof of \Cref{lemma_G_large_2}]
Notice that $(1-\langle M(z)^2\rangle)^{-1}\lesssim 1$ by  \eqref{eq:msc0.5} and \eqref{eq:msc2} below. Hence, the estimate \eqref{largG_1G} follows from \Cref{lem:fourthcumu2} immediately for $\eta\sim N^{-\e_g}$. It remains to show \eqref{largG_Ex}.

First, when $z_1=z_2\in \{z,\bar z\}$, by \eqref{1-M-2} below, we have $\max_{a,b\in \qqD}\| K_{ab}(B)\|\lesssim 1$ for any matrix $B$ with $\|B\|=\OO(1)$. Then, applying \Cref{lem:fourthcumu1} and \Cref{lem:fourthcumu2} to \eqref{ELK}, we obtain that 
\begin{align*}
	\E L_{ab}&= \wh{M}_{ab} + \sum_{x=1}^D \wh M_{ax} \E L_{xb} + \opr{N^{-1}+\eta^{-2} N^{-3/2}+\eta^{-4}N^{-2}},
\end{align*} 
solving which gives \eqref{largG_Ex}.

Next, we consider the case $z_1=\bar z_2\in \{z,\bar z\}$. Without loss of generality, we assume that $z_1=\bar z_2=z$. Again, applying \Cref{lem:fourthcumu1} and \Cref{lem:fourthcumu2} to \eqref{ELK}, we obtain that  
\begin{align}
	\E L_{ab} &= \wh{M}_{ab} + \sum_{x=1}^D \wh M_{ax} \E L_{xb} + \OO_\prec\big({N^{-\delta_A}(N\eta)^{-1}+N^{-1}+\eta^{-2} N^{-3/2}+\eta^{-4}N^{-2}}\big)\nonumber\\
	&+\frac{\kappa^{(2, 2)}}{N} \left[ m_{sc}(z_2)^4 + |m_{sc}(z_1)|^4+ |m_{sc}(z_1)|^2 m_{sc}(z_1)^2+ |m_{sc}(z_1)|^2 m_{sc}(z_2)^2 \right]K_{ab}\nonumber\\
	&+ \frac{\kappa^{(2,2)}  }{N}\left(\frac{m_{sc}(z_1)^4|m_{sc}(z_1)|^2}{1-m_{sc}(z_1)^2}  +\frac{m_{sc}(z_2)^6}{1-m_{sc}(z_2)^2}\right) K_{ab},\label{ELK222}
\end{align} 
where we have used the estimate \eqref{eq:msc0.5} below to replace $M_1$ and $M_2$ by scalar matrices $m_{sc}(z_1)I$ and $m_{sc}(z_2)I$ up to an error of order $\OO(\|A\|)=\OO(N^{-\delta_A})$. Now, using $|m_{sc}(z_1)|=1-\OO(\eta)$ and $m_{sc}(z_2)=\overline m_{sc}(z_1)$, we can check that  
\begin{align}
& m_{sc}(z_2)^4 + |m_{sc}(z_1)|^4+ |m_{sc}(z_1)|^2 m_{sc}(z_1)^2+ |m_{sc}(z_1)|^2 m_{sc}(z_2)^2 + \frac{m_{sc}(z_1)^4|m_{sc}(z_1)|^2}{1-m_{sc}(z_1)^2}  +\frac{m_{sc}(z_2)^6}{1-m_{sc}(z_2)^2}\nonumber\\
& =m_{sc}(z_2)^4+ m_{sc}(z_1)^2+ m_{sc}(z_2)^2+ 1+ \frac{m_{sc}(z_1)^4}{1-m_{sc}(z_1)^2} + \frac{m_{sc}(z_2)^6}{1-m_{sc}(z_2)^2} +  \OO(\eta) =\OO(\eta).\label{eq:sumzero}
\end{align}
Plugging it back to \eqref{ELK222} and using that $\eta\sim N^{-\e_g}$ with $\e_g\in (0,1/4)$, we get
\begin{align*}
	\E L_{ab} &= \wh{M}_{ab} + \sum_{x=1}^D \wh M_{ax} \E L_{xb} + \OO_{\prec}\big(N^{-\delta_A}(N\eta)^{-1}+N^{-1}\big).
\end{align*}
Solving $\E L$ from this equation and using \eqref{1-M} below concludes \eqref{largG_Ex}.
\end{proof}

\begin{remark}\label{rem:self}
We remark that the cancellation in \eqref{eq:sumzero} is not a mere coincidence. Instead, there is a robust mechanism behind it that does not depend on the specific properties of $m_{sc}$. This mechanism has been thoroughly investigated and elucidated in the context of random band matrices in \cite{yang2021delocalization}, where it is referred to as the \emph{sum zero property} or \emph{self-energy renormalization}. To explain the core idea, we recall that by Lemmas \ref{lem:fourthcumu1} and \ref{lem:fourthcumu2}, there is 
\be\label{eq:ELab}
\E L_{ab} = \wh{M}_{ab} + \sum_{x=1}^D \wh M_{ax} \E L_{xb} + \cal D_{ab} + \OO_\prec (N^{-1}),
\ee
for some deterministic matrix $\cal D$ satisfying the ``rough bound" $\|\cal D\|\prec (N\eta)^{-1}$. Now, summing both sides of this equation over $a\in \qqD$, we get that 
\begin{align*}
     \left( 1-\frac{\im m}{\im m + \eta}\right)\frac{\E D\im \langle GE_{b}\rangle}{\eta} = \frac{\im m}{\im m + \eta}  + \sum_{a=1}^D \cal D_{ab} + \OO_\prec (N^{-1}),
\end{align*}
where we have used equation \eqref{sumwtM} below and applied Ward's identity \eqref{eq_Ward} to $\sum_x L_{xb}$. By \Cref{lem:fourthcumu2}, this equation gives 
\be\label{eq:nontrivial_cancel} \sum_{a=1}^D \cal D_{ab} = \frac{D\im \E \langle (G-M)E_{b}\rangle}{\im m+ \eta} + \OO_\prec(N^{-1})=\OO_\prec(N^{-1}).
\ee
This already indicates a non-trivial cancellation compared to the rough bound $(N\eta)^{-1}$. In particular, if we can write $\cal D_{ab}$ as some coefficient $c_N$ times $K_{ab}$ (up to a small error) as shown in the proof of \Cref{lemma_G_large_2} above, then $c_N$ must be of order $\OO_\prec(\eta/N)$. This gives another proof of \eqref{eq:sumzero}.
\end{remark}

\section{Characteristic flow estimates}\label{sec:flow}

In this section, we complete the proofs of Lemmas \ref{lem_flow_1}--\ref{lem_flow_3} using the characteristic flow method. Let $\mathbf{B}_t=\left(b_{i j}(t)\right)_{i,j\in \cI}$ be a $D  \times D $ block matrix Brownian motion consisting of the diagonal blocks $(B_a)_t$ in \eqref{eq:Ht}. Then, by \eqref{eq:Ht}, $H_t=(h_{ij}(t))_{i,j\in \cI}$ satisfies the equation 
$$
\dd h_{i j}=-\frac{1}{2} h_{i j} \dd t+\frac{1}{\sqrt{N}} \dd b_{i j}(t) ,
$$
with initial data $H_{t_0}=H_0$. 
Let $F$ be any function of $t$ and $H$ with continuous second-order derivatives. Then, by Itô's formula, we have that
\be\label{Ito}
\dd F=\partial_t F \dd t+ \sum_{a=1}^D \sum_{k,l \in \mathcal{I}_a} \partial_{h_{kl}} F \dd h_{kl}+\frac{1}{2 N} \sum_{a=1}^D \sum_{k,l \in \mathcal{I}_a} \partial_{h_{kl}}\partial_{ h_{lk}}  F \dd t.
\ee
We will apply this equation to functions of the resolvents $G_{i,t}\equiv (G_i)_t =(H_t-Z_{i,t})^{-1}$ with $Z_{i,t}=(z_{i})_{t}-\Lambda_t$ for $z_{i}\in \{z, \bar z\}$. 
Using the formula (with the simplified notation $\partial_{kl}\equiv \partial_{h_{kl}}$)
\begin{equation} \label{Ito_pro_1}
        \partial_{kl} \left(G_{i,t}\right)_{k'l'}= -\left(G_{i,t}\right)_{k' k}\left(G_{i,t}\right)_{ll'},\quad k',l' \in \cI, \ k,l\in \cal I_a, \ a\in \qqD,
\end{equation}
we can easily obtain the following identities (with $M_{i,t}\equiv (M_i)_t$): 
\begin{equation} \label{Ito_pro_2}
        \partial_{t}  G_{i,t}= G_{i,t}  \left(\frac{\dd}{\dd t}Z _{i,t}\right) G_{i,t},\quad \text{with}\quad \frac{\dd}{\dd t}Z_{i,t}=-\frac12 Z_{i,t}-\langle M_{i,t}\rangle; 
\end{equation}
\begin{equation} \label{Ito_pro_3}
 \sum_{a=1}^D \sum_{k,l \in \mathcal{I}_a}  h_{kl}\partial_{kl} G_{i,t}  =-G_{i,t}H_t G_{i,t}=-G_{i,t}-G_{i,t}Z_{i,t}G_{i,t};
\end{equation}
\begin{equation} \label{Ito_pro_4}
        \sum_{k,l \in \mathcal{I}_a} \partial_{kl} \left(G_{i,t}\right)_{k_1l_1} \cdot \partial_{lk} \left(G_{i',t}\right)_{k_2l_2}  = \left(G_{i,t}E_{a}G_{i',t}\right)_{k_1 l_2} \left(G_{i',t}E_{a}G_{i,t}\right)_{k_2l_1},\quad k_1,l_1,k_2,l_2 \in \cI.
\end{equation}

\subsection{Proof of \Cref{lem_flow_1}}

In this subsection, we give the proof of \Cref{lem_flow_1}. For simplicity of notations, we abbreviate $\wh M_{(1,2),t}$, $L_{(1,2),t}$, and $K_{(1,2),t}$ as $\wh M_{t}$, $L_t$, and $K_t$. Moreover, we denote $z_t=E_t+\ii \eta_t$ and
\be\label{eq:wtL12t}
\wt L_t\equiv \wt L_{(1,2),t}:= (t_c-t) L_t, \quad \wt K_t\equiv \wt K_{(1,2),t}:= (t_c-t) K_t.
\ee
Using Itô's formula \eqref{Ito} and the identities \eqref{Ito_pro_1}--\eqref{Ito_pro_4}, we can calculate that for $ x,y \in \qqD$, 
\begin{align*}
\dd (\widetilde{L}_{t})_{xy} &= - (L_t)_{xy} \dd t  + \frac{1}{\sqrt{N}} \sum_{a=1}^D \sum_{k, l \in \mathcal{I}_a} \partial_{k l} (\widetilde{L}_{t})_{x y} \dd b_{k l} + D(t_c-t)\left\langle G_{1,t} E_x G_{2,t} E_y\right\rangle \dd t \\
&+ D^2(t_c-t) \sum_{a=1}^D\left\langle G_{1,t} E_x G_{2,t} E_a\right\rangle\left\langle G_{2,t} E_y G_{1,t} E_a\right\rangle \dd t  \\
& + D^2(t_c-t) \sum_{a=1}^D
 \left\langle \left(G_{1,t}-M_{1,t} \right) E_a\right\rangle\left\langle G_{1,t} E_x G_{2,t} E_y G_{1,t} E_a\right\rangle \dd t\\
 & + D^2(t_c-t) \sum_{a=1}^D \left\langle \left(G_{2,t}-M_{2,t} \right) E_a\right\rangle\left\langle G_{2,t} E_y G_{1,t} E_x G_{2,t} E_a\right\rangle \dd t  .
 \end{align*}
Using the definitions of $\wt L_t$ and $L_{(1,2,3),t}$, we can rewrite the above equation as 
 \begin{align}
\dd (\widetilde{L}_{t})_{xy}&= \frac{1}{\sqrt{N}} \sum_{a=1}^D \sum_{k, l \in \mathcal{I}_a} \partial_{k l} (\widetilde{L}_{t})_{x y} \dd b_{k l} + \left(1 - \frac{1}{t_c-t}\right) (\widetilde{L}_{t})_{x y} \dd t + \frac{1}{t_c-t} \sum_{a=1}^D (\widetilde{L}_{t})_{x a} (\widetilde{L}_{t})_{a y} \dd t  \nonumber\\
&+ D(t_c-t) \sum_{a=1}^D \left\{\left\langle \left(G_{1,t}-M_{1,t} \right) E_a\right\rangle [L_{(1,2,1),t}]_{xya} + \left\langle \left(G_{2,t}-M_{2,t} \right) E_a\right\rangle [L_{(2,1,2),t}]_{yxa} \right\} \dd t . \label{ItowtL}
\end{align}
Next, with the averaged local law \eqref{eq:aver_local} and \Cref{appen1}, we can bound the last term by $\OO_\prec((t_c-t)\cdot N^{-1}\eta_t^{-3})=\OO_\prec(N^{-1}(t_c-t)^{-2})$, where we used $\eta_t\sim t_c-t$ by \eqref{t_eta}. Hence, we can rewrite \eqref{ItowtL} as 
\be\label{mainwtL}
\begin{aligned}
\dd  \widetilde{L}_{t }
=\frac{1}{\sqrt{N}} \sum_{a=1}^D \sum_{k, l \in \mathcal{I}_a} \partial_{k l} \widetilde{L}_{t } \dd  b_{k l} +\left[  \left(1 - \frac{1}{t_c-t}\right) \widetilde{L}_{t} + \frac{1}{t_c-t} (\widetilde{L}_{t })^2   \right] \dd  t+\OO_\prec\left(\frac1{N(t_c-t)^2}\right)\dd t .
\end{aligned}
\ee
On the other hand, by \eqref{eq:whMK}, we see that $\wt K_t$ satisfies the following equation:
\be\label{mainwtK}
\frac{\dd}{\dd t} \wt K_t= \left(1 - \frac{1}{t_c-t}\right) \widetilde{K}_{t} + \frac{1}{t_c-t} (\widetilde{K}_{t })^2, 
\ee
which matches the drift term in \eqref{mainwtL}.

We now study the martingale term in \eqref{mainwtL}, which is denoted as $\cal L_t$: 
$$ \dd\mathcal{L}_t= \frac{1}{\sqrt{N}}\sum_{a=1}^D\sum_{k,l\in\cal I_a}\partial_{kl}\wt L_{t}\dd b_{kl} \quad \text{with}\quad \cal L_{t_0}=0.
$$
The quadratic variation of $(\cal L_t)_{xy}$, $x,y \in \qqD$, is given by
\begin{align}\label{eq:quad_var}
[ \cal L_{xy}]_t &= \frac{1}{N}\int_{t_0}^t \sum_{a=1}^D\sum_{k,l\in\cal I_a}|\partial_{kl}(\wt L_s)_{xy}|^2\dd s.
\end{align}
Using \eqref{Ito_pro_1}, we can calculate the integrand as
\begin{align*}
\sum_{a=1}^D\sum_{k,l\in\cal I_a}|\partial_{kl}(\wt L_s)_{xy}|^2  =& \frac{(t_c-s)^2}{N^2}\sum_{a=1}^D\sum_{k,l\in\cal I_a}\Big(\left|(G_{1,s}E_xG_{2,s}E_yG_{1,s})_{lk}\right|^2 + \left|(G_{2,s}E_yG_{1,s}E_xG_{2,s})_{lk}\right|^2\\
&\qquad \qquad \qquad + 2\re\left[(G_{1,s}E_xG_{2,s}E_yG_{1,s})_{lk}\overline{(G_{2,s}E_yG_{1,s}E_xG_{2,s})_{lk}}\right]\Big)\\
=& \frac{D(t_c-s)^2}{N }\sum_{a=1}^D \Big( \langle G_{1,s}E_xG_{2,s}E_yG_{1,s}E_aG_{1,s}^*E_yG_{2,s}^*E_xG_{1,s}^*E_a\rangle\\
&\qquad \qquad \qquad +  \langle G_{2,s}E_yG_{1,s}E_xG_{2,s}E_aG_{2,s}^*E_xG_{1,s}^*E_yG_{2,s}^*E_a\rangle\\
&\qquad \qquad \qquad + 2 \re\langle G_{1,s}E_xG_{2,s}E_yG_{1,s}E_aG_{2,s}^*E_xG_{1,s}^*E_yG_{2,s}^*E_a\rangle\Big).
\end{align*}
Applying \Cref{appen1} and \eqref{t_eta}, we obtain that if $t_0 \le s \le t_c- N^{-1+\tau}$ for a constant $\tau>0$,
\be\label{eq:quad_integrand}
\sum_{a=1}^D\sum_{k,l\in\cal I_a}|\partial_{kl}(\wt L_s)_{xy}|^2  \prec \frac{|t_c-s|^2}{N }\cdot \frac{1}{\eta_s^5} \lesssim \frac{1}{N(t_c-s)^3} .
\ee
With a standard continuity argument, we obtain that this estimate holds uniformly in $s\in [t_0 , t_c- N^{-1+\tau}]$ (i.e., we first show that \eqref{eq:quad_integrand} holds uniformly in $t$ belonging to an $N^{-C}$-net of $[t_0 , t_c- N^{-1+\tau}]$ and then extend it uniformly to the whole interval using the Lipschitz continuity in $t$). Plugging \eqref{eq:quad_integrand} into \eqref{eq:quad_var}, we get the estimate
\be\label{eq:Ltxy}
[ \cal L_{xy}]_t  \prec \frac{1}{N^2(t_c-t)^2}, \quad \text{if}\quad t_0 \le t \le t_c- N^{-1+\tau}.
\ee
On the other hand, we have the trivial bound $|[ \cal L_{xy}]_t |\le N$ by using $\|G_{i,s}\|\le \eta_s^{-1} \ll N$ for $t\in [t_0 , t_c- N^{-1+\tau}]$. Together with \eqref{eq:Ltxy} and \Cref{stoch_domination}, it implies that for any constant $c>0$ and fixed $p\in \N$, 
$$\mathbb E \left|[ \cal L_{xy}]_t\right|^p \le \left(\frac{N^c}{N^2(t_c-t)^2}\right)^p, \quad \text{if}\quad t_0 \le t \le t_c- N^{-1+\tau}. $$
Applying the Burkholder-Davis-Gundy inequality, we obtain a $p$-th moment bound on $\sup_{s\in[ t_0,t] }\left|(\cal L_s)_{xy}\right|$. Then, applying Markov's inequality yields that for any $t\in [t_0 , t_c- N^{-1+\tau}]$ and $x,y \in \qqD$,
\be\label{CIest}
\sup_{s\in[ t_0,t] }\left|(\cal L_s)_{xy}\right| \prec \frac{1}{  N(t_c-t) }.
\ee

Inserting \eqref{CIest} back to \eqref{mainwtL}, we obtain that for any $t\in [t_0 , t_c- N^{-1+\tau}]$ and $x,y \in \qqD$,
\be\label{mainwtL2}
 \widetilde{L}_{t }- \widetilde{L}_{t_0 }
=\int_{t_0}^{t} \left[\left(1 - \frac{1}{t_c-s}\right) \widetilde{L}_{s}  + \frac{1}{t_c-s}   (\widetilde{L}_{s} )^2 \right] \dd s+\OO_\prec \left(\frac1{N(t_c-t) }\right) .
\ee
On the other hand, by \eqref{mainwtK}, we have
\be\label{mainwtL3} 
 \widetilde{K}_{t }- \widetilde{K}_{t_0 }
=\int_{t_0}^{t} \left[\left(1 - \frac{1}{t_c-s}\right) \widetilde{K}_{s}  + \frac{1}{t_c-s} (\widetilde{K}_{s})^2\right] \dd s.
\ee
For simplicity, we introduce the notation $ \wt \LK_t:= \wt L_t-\wt K_t$ and define the linear operator 
$\cal T_t$ acting on $D\times D$ matrices as 
\be\label{eq:TtV}
\cal T_t(V):=\wt K_t  V+V\wt K_t - [1-(t_c-t)] V,\quad V\in \C^{D\times D}.
\ee
Then, subtracting \eqref{mainwtL3} from \eqref{mainwtL2}, we obtain that 
\begin{align}
\widetilde{\LK}_{t }- \widetilde{\LK}_{t_0} &=\int_{t_0}^{t} \left[\left(1 - \frac{1}{t_c-s}\right) \widetilde{\LK}_{s}  + \frac{1}{t_c-s} \left(\widetilde{K}_{s}\widetilde{\LK}_{s}+ \widetilde{\LK}_{s}\widetilde{K}_{s} + (\widetilde{\LK}_{s})^2\right)\right] \dd s  +\OO_\prec \left(\frac1{N(t_c-t) }\right)\nonumber \\
&=\int_{t_0}^{t} \left(\left[(t_c-s) -1\right] \widetilde{\LK}_{s}  +  \widetilde{K}_{s}\widetilde{\LK}_{s}+ \widetilde{\LK}_{s}\widetilde{K}_{s} + (\widetilde{\LK}_{s})^2 \right) \frac{\dd s}{t_c-s}  +\OO_\prec \left(\frac1{N(t_c-t) }\right)\nonumber\\
&=\int_{t_0}^{t} \left( \cal T_s( \widetilde{\LK}_s)  +  (\widetilde{\LK}_s)^2\right) 
 \frac{\dd s}{t_c-s}+\cal E_{t},\label{wtLKint}
\end{align}
where $\cal E_{t}$ is a $D\times D$ random matrix satisfying that $\|\cal E_{t}\|_{\HS} \prec [N(t_c-t)]^{-1} $ uniformly in $t\in [t_0 , t_c- N^{-1+\tau}]$.  Denoting $\LKE_t:=\wt\LK_t-\cal E_t$ and noticing that $\cal E_{t_0}=0$, we can rewrite \eqref{wtLKint} as
\be\label{eq:deltat}
{\LKE}_{t} - \LKE_{t_0} = \int_{t_0}^{t} \left( \mathcal{T}_s(\LKE_s)+ \mathcal{T}_s(\mathcal{E}_s) + (\LKE_s + \mathcal{E}_s)^2 \right) \frac{\dd s}{t_c-s }.
\ee
Let $\Phi\left(t ; t_0\right)$ be the standard Peano-Baker series corresponding to the linear operator $\cal  T_t/(t_c-t)$, i.e., it is the unique solution to the following linear integral equation
\be\label{eq:PBintegral}
\Phi \left(t ; t_0\right)=\mathbf{1}+\int_{t_0}^t \frac{\cal T_s}{t_c-s}\circ \Phi \left(s ; t_0\right) \mathrm{d} s , 
\ee
where $\mathbf{1}$ denotes the identity operator. Equivalently, $\Phi\left(t ; t_0\right)$ satisfies the differential equation 
$$ \frac{\dd}{\dd t}\Phi \left(t ; t_0\right)=\frac{\cal T_t}{t_c-t}\circ \Phi \left(t ; t_0\right), \quad \text {with}\quad \Phi \left(t_0 ; t_0\right)=\mathbf{1}. $$
By Duhamel's principle, the solution $\LKE_t$ to \eqref{eq:deltat} can be expressed as 
\be\label{LKE}
\LKE_{t }
=\Phi \left(t  ; t_0\right) \LKE_{t_0}+\int_{t_0}^t \Phi \left(t ; s\right) \left(\frac{\mathcal{T}_s(\mathcal{E}_s) + (\LKE_s + \mathcal{E}_s)^2}{t_c-s} \right) \mathrm{d} s.\ee

Suppose the space $\mathbb{C}^{D\times D}$ of $D\times D$ matrices is equipped with the Hilbert-Schmidt norm. Then, we claim that, as a linear operator on $\mathbb{C}^{D\times D}$, 
$\mathcal T_t$ has operator norm at most $1 + \oo(1)$:
\be\label{Tnorm} 
 \|\cal T_t\|_{op}\le 1+\oo(1).
 \ee
Before proving this estimate, we first use it to prove \eqref{flow_1_result}.
With \eqref{Tnorm}, we get from \eqref{eq:PBintegral} that 
$$\frac{\dd}{\dd t}\|\Phi(t;s)\|_{op}\le \frac{1+\oo(1)}{t_c-t} \|\Phi(t;s)\|_{op}.$$
Using Gr{\"o}nwall's inequality, we conclude that for $t_0\le s \le t\le t_c- N^{-1+\tau} $,
\be\label{Phinorm} 
\|\Phi \left(t ; s\right)\|_{op}\prec  \frac{t_c-s}{t_c-t}  .
\ee
Applying \eqref{Tnorm} and \eqref{Phinorm} to \eqref{LKE} and using the bound on $\|\cal E_t\|_{\HS}$, we obtain that 
$$
\|{\LKE}_{t}\|_{2}\prec \frac{t_c-t_0}{t_c-t} \|{\LKE}_{t_0}\|_2 +\frac{1}{t_c-t}\int_{t_0}^t \|{\LKE}_s+\cal E_s\|_2^2 \dd s+\int_{t_0}^t \frac{\dd s}{N(t_c-t)(t_c-s)}.$$
From this estimate, writing $\LKE_t=\wt\LK_t-\cal E_t$, we obtain that for $t_c-t_0\sim N^{-\e_g}$ and $t_c-t\ge N^{-1+\tau}$,
\be\label{eq:wtLKt}
\|\wt {\LK}_{t }\|_2\prec \frac{t_c-t_0}{t_c-t} \|\wt {\LK}_{t_0}\|_2 +\frac{1}{t_c-t}\int_{t_0}^t \|\wt{\LK}_s\|_2^2 \dd s+\frac{1}{N(t_c-t)}.
\ee
By \eqref{largG_Ent}, we have $\|\widetilde{\LK}_{t_0 }\|_2\prec N^{-1+2\e_g}$. Then, from \eqref{eq:wtLKt}, we derive the the following self-improving estimate for $t\in [t_0,t_c- N^{-1+C\e_g}]$ when $C>2$:  
\be\label{eq:self_imp}
\sup_{ s \in [t_0, t]}N(t_c-s)\|\wt\LK_s\|_2\le N^{2\e_g} \ \Rightarrow \ N(t_c-t)\|\wt\LK_{t}\|_2\prec N^{\e_g}+N^{(4-C)\e_g}.
\ee
Moreover, defining the stopping time $T=\inf_{t\ge t_0} \{N(t_c-t)\|\widetilde{\LK}_t\|_2\ge  N^{2\e_g}\}$, we obtain from \eqref{eq:wtLKt} that 
  $$
\|\wt {\LK}_{t }\|_2\prec \frac{t_c-t_0}{t_c-t} \|\wt {\LK}_{t_0}\|_2
 +\frac{1}{N(t_c-t)}$$
if $t\le T$ and $t_c-t\ge N^{-1+C\e_g}$ with $C>4$. Now, applying a standard continuity argument with \eqref{eq:self_imp} gives that $T\ge t_c- N^{-1+C\e_g}$ with high probability when $C>4$ and hence concludes the desired result \eqref{flow_1_result}.  

Finally, we prove the bound \eqref{Tnorm}. Notice that 
\be\label{eq:wtKt}
\| \wt K_t\|=(t_c-t)\| K_t\|\le (t_c-t)\| (1-\wh M_t)^{-1}\|\|\wh M_t\|\lesssim (t_c-t)
\| (1-\wh M_t)^{-1}\| .
\ee
Hence, when $(z_1)_t=(z_2)_t \in \{z_t,\bar z_t\}$, we have $\| \wt K_t\|\lesssim t_c-t$ by \eqref{1-M-2} below, with which we readily derive \eqref{Tnorm}. 
It remains to consider the case $(z_1)_t=(\bar z_2)_t \in \{z_t,\bar z_t\}$. Since $\wh M_t$ is a circulant matrix, it has an eigendecomposition $\wh M_t = U_t D_t U_t^*$, where $D_t$ is the diagonal matrix of eigenvalues and $U_t$ is a $D\times D$ unitary matrix. Then, $\wt K_t$ can be written as 
$$\wt K_t = U_t\Xi_t U_t^*,\quad \Xi_t:=(t_c-t)\frac{D_t}{1-D_t}.$$
Now, we define the linear operator $\wt{\cal T}_t$ as 
$$
\wt{\cal T}_t(V):=\Xi_t  V+V\Xi_t - [1-(t_c-t)] V,\quad V\in \C^{D\times D}.
$$
It is easy to see ${\cal T}_t(V)= U_t[\wt{\cal T}_t(U_t^* VU_t)]U_t^*$, which implies that $\|\cal T_t\|_{op}=\|\wt{\cal   T }_t\|_{op}$. From the definition of $\wt{\cal T}_t$, we see that 
  \be\label{T2aa}
  \|\wt{\cal T}_t\|_{op} \le 
    \max_{k,l\in \qqD} \left|(\Xi_t)_{kk}+(\Xi_t)_{ll}-1\right| + |t_c-t|.
  \ee
It remains to estimate the eigenvalues of $\wt K_t$. 

Since the entries of \smash{$\wh M_t$} are all non-negative when $(z_1)_t=(\bar z_2)_t$, it has a Perron–Frobenius eigenvalue 
$$d_1=\frac{\im m_t(z_t)}{\im m_t(z_t) + \eta_t}$$ 
by equation \eqref{sumwtM} below. Moreover, by equation \eqref{eq:otherM}, the eigenvalues $d_k$ of \smash{$\wh M_t$} satisfy $d_k=d_1-a_k-\ii b_k$, $k \in \qqD$, for some $a_k\ge 0$ and $a_k+|b_k|=\oo(1)$ (specifically, $a_1=b_1=0$). Thus, 
\begin{align*}
(\Xi_t)_{kk}+(\Xi_t)_{ll}-1 & = (t_c-t)\left[\frac{d_1-a_k-\ii b_k}{(1-d_1)+a_k+\ii b_k} + \frac{d_1-a_l-\ii b_l}{(1-d_1)+a_l+\ii b_l}\right]-1 \\
& =\frac{\eta_t}{\eta_t + a'_k + \ii b'_k} + \frac{\eta_t}{\eta_t + a'_l + \ii b'_l} -1 +\oo(1), 
\end{align*} 
where we used \eqref{t_eta} in the second step and abbreviated that $a'_k:=(\im m_t(z_t)+\eta_t)a_k$ and $b'_k:=(\im m_t(z_t)+\eta_t)b_k$. 
Together with the simple fact $|1/(1+z)-1/2|\le 1/2$ when $\re z\ge 0$, this equation implies $|(\Xi_t)_{kk}+(\Xi_t)_{ll}-1|\le 1+\oo(1)$. Plugging it into \eqref{T2aa} concludes \eqref{Tnorm}.

\subsection{Proof of \Cref{lem_flow_2}} 

The proof of \Cref{lem_flow_2} is similar to that of \Cref{lem_flow_1} above. Hence, we only describe an outline of the proof without giving all the details. For simplicity, we abuse the notations a little bit and abbreviate $L_{(1,2,3),t}$ and $K_{(1,2,3),t}$ as $L_t$ and $K_t$. Moreover, we denote $z_t=E_t+\ii \eta_t$ and
 $$
\wt L_t:= (t_c-t)^{3/2} L_t, \quad \wt K_t:= (t_c-t)^{3/2} K_t, \quad \wt L_{(1,2),t}=(t_c-t)L_{(1,2),t}, \quad \wt K_{(1,2),t}=(t_c-t)K_{(1,2),t}.
$$
Similar to \eqref{mainwtL}, using It{\^o}'s formula \eqref{Ito}, the identities \eqref{Ito_pro_1}--\eqref{Ito_pro_4}, and Lemmas \ref{lem_loc} and \ref{appen1}, we  obtain that for $x,y,w\in \qqD$, 
\begin{align}
\dd (\widetilde{L}_{t})_{x y w}
 &= \frac{1}{\sqrt{N}} \sum_{a=1}^D \sum_{k, l \in \mathcal{I}_a} \partial_{k l} (\widetilde{L}_{t})_{xy w} \dd b_{k l} 
 + \frac{3}{2}\left( 1- \frac{1}{t_c-t}\right)(\widetilde{L}_{t})_{x y w} \dd t  
 + \OO_\prec \bigg(\frac{1}{N\left(t_c-t\right)^{5/2}}\bigg)\dd t 
 \nonumber\\
& + \frac{1}{t_c-t} 
\sum_{a=1}^D \left[  (\widetilde{L}_{(1,2),t})_{x a} (\widetilde{L}_{t})_{a y w}
+ (\widetilde{L}_{(2,3),t})_{y a} (\widetilde{L}_{t})_{ x a  w}
+ (\widetilde{L}_{(3,1),t} )_{ w a} (\widetilde{L}_{t})_{x y a}
\right]  \dd t. \label{eq:SDEwtL3}
\end{align}

We again denote the martingale term by 
 $$
 \dd \cal L_t:=\frac{1}{\sqrt{N}} \sum_{a=1}^D \sum_{k, l \in \mathcal{I}_a} \partial_{k l} (\widetilde{L}_{t})_{x yz} \dd b_{k l}. 
$$
Similar to \eqref{CIest}, we can prove that uniformly in $t\in [t_0 , t_c- N^{-1+\tau}]$, 
\be\label{CIest3G}
\sup_{x,y,w \in \qqD}\sup_{s\in[ t_0,t] }|(\cal L_s)_{xyw}|\prec \frac{1}{  N(t_c-t)^{3/2} }.
\ee
Plugging it into \eqref{eq:SDEwtL3}, we obtain the following estimate for $x,y,w \in \qqD$ and $t\in [t_0 , t_c- N^{-1+\tau}]$:
\begin{align}
 \left(\widetilde{L}_{t}- \widetilde{L}_{t_0 }\right)_{xyw}
=&\int_{t_0}^{t}  \frac32\left( 1 - \frac{1}{t_c-s} \right)(\widetilde{L}_{s })_{xyw} \dd s+\OO_\prec \left(\frac1{N(t_c-t)^{3/2} }\right) \label{eq:SDEwtL2}\\
 +& \int_{t_0}^{t} \frac{1}{t_c-s}  
 \sum_{a=1}^D\left[  (\widetilde{L}_{(1,2),s})_{x a} (\widetilde{L}_{s})_{a y w}
+ (\widetilde{L}_{(2,3),s})_{y a} (\widetilde{L}_{s})_{ x a  w}
+ (\widetilde{L}_{(3,1),s} )_{ w a} (\widetilde{L}_{s})_{x y a}
\right]\dd s .\nonumber
\end{align}
Using \eqref{flow_1_result_2} and applying \Cref{appen1} to $\wt L_s$, we obtain from \eqref{eq:SDEwtL2} that
\be\label{mainwtL3G}
\begin{aligned}
  \widetilde{L}_{t}- \widetilde{L}_{t_0 } 
=&\int_{t_0}^{t} 
 \mathscr T_s(\wt L_{s})  \frac{\dd s}{t_c-s}+\OO_\prec \left(\frac{N^{\e_g}}{N(t_c-t)^{3/2} }\right),
\end{aligned}
\ee
where $\mathscr T_t:\C^{D\times D\times D}  \to \C^{D\times D\times D} $ is a linear operator defined as follows: for $V\in \C^{D\times D\times D}$, 
$$
\mathscr T_t(V)_{xyw}
=  
 \sum_{a=1}^D 
 \left[ 
  (\widetilde{K}_{(1,2),t} )_{x a} V_{  a y w}
+ (\widetilde{K}_{(2,3),t} )_{y a} V_{  x a w}
+ (\widetilde{K}_{(3,1),t} )_{w a} V_{  x y a}
\right]- \frac32[1-(t_c-t)] V_{xyw}.
$$

On the other hand, by \eqref{eq:whMK3}, we have 
\be\label{mainwtL3K}
 \widetilde{K}_{t}- \widetilde{K}_{t_0}
=\int_{t_0}^{t} 
\mathscr T_s(\wt K_{s})  \frac{\dd s}{t_c-s}.
\ee
Now, subtracting \eqref{mainwtL3K} from \eqref{mainwtL3G} and denoting $ \wt {\LK}_t:= \wt L_t-\wt K_t$, we obtain that
\be\label{eq:L3t}
  \widetilde{\LK}_{t}- \widetilde{\LK}_{t_0 }
 =\int_{t_0}^{t}  \mathscr T_s( \widetilde{\LK}_s) 
 \frac{\dd s}{t_c-s}+\cal E_{t},
\ee
 where $\cal  E_{t}$ is a $D\times D\times D$ random tensor satisfying that $\|\cal  E_{t}\|_2 \prec  {N^{\e_g}}/[N(t_c-t)^{3/2} ]$, where $\|\cdot\|_2$ denotes the $\ell^2$-norm by regarding $\cal  E_{t}$ as a $D^3$-dimensional vector. 
We now claim that 
 \be\label{eq:boundT3/2}
 \|\mathscr T_t\|_{op}\le 3/2+\oo(1).
 \ee
This bound is trivial when $(z_1)_t=(z_2)_t=(z_3)_t$. Next, we assume that $(z_1)_t=(z_3)_t=(\bar z_2)_t\in \{z_t,\bar z_t\}$ without loss of generality. Then, $\|\widetilde{K}_{(3,1),t} \|=\oo(1)$ and $(\widetilde{K}_{(2,3),t})_{ya}=(\widetilde{K}_{(1,2),t})_{ay}$. Hence, we can write $\mathscr T_t$ as $$\mathscr T_t(V)=\mathscr T^{(1)}_t(V) + \sum_{a=1}^D (\widetilde{K}_{(3,1),t} )_{w a} V_{  x y a}- \frac12[1-(t_c-t)] V_{xyw},$$
where the operator $\mathscr T^{(1)}_t$ is defined as
$$
\mathscr T^{(1)}_t(V)_{xyw}:=\sum_{a=1}^D 
 \left[
  (\widetilde{K}_{(1,2),t} )_{x a} V_{  a y w}
+  V_{  x a w}(\widetilde{K}_{(1,2),t} )_{ay}
\right] - [1-(t_c-t)]V_{xyw} .
$$
Note that when $w$ is fixed, $\mathscr T^{(1)}_t$ reduces to the operator \eqref{eq:TtV}. With \eqref{Tnorm}, we can easily conclude \eqref{eq:boundT3/2}. 
 
Finally, we adopt a similar argument as that below \eqref{Tnorm}. Let $\Phi\left(t ; t_0\right)$ be the standard Peano-Baker series corresponding to the linear operator $\mathscr T_t/(t_c-t)$. From \eqref{eq:boundT3/2}, we can derive that for $t_0\le s \le t\le t_c- N^{-1+\tau} $,
\be\label{Phinorm2} 
\|\Phi \left(t ; s\right)\|_{op}\prec  \frac{(t_c-s)^{3/2}}{(t_c-t)^{3/2}}  .
\ee
Then, from \eqref{eq:L3t} and \eqref{Phinorm2}, we obtain that 
$$
\| \widetilde{\LK}_{t }\|_2
 \prec \frac{(t_c-t_0)^{3/2}}{(t_c-t)^{3/2}} \|\widetilde{\LK}_{t_0 }\|_2+\frac{N^{\e_g}}{N(t_c-t)^{3/2}},
$$
which implies \eqref{flow_3G_res} and completes the proof of \Cref{lem_flow_2}. 


\subsection{Proof of \Cref{lem_flow_1.5}}

To mimic our proofs of Lemmas \ref{lem_flow_1} and \ref{lem_flow_2} above, we abuse the notations again and denote $\wt L_t,\wt K_t\in \C^D$ by
$$
(\wt L_t)_{x}:= (t_c-t)^{1/2}\langle G E_x\rangle , \quad 
(\wt K_t)_{x}:= (t_c-t)^{1/2}\langle M E_x\rangle,\quad x\in \qqD. 
$$
Again, using It{\^o}'s formula \eqref{Ito}, the identities \eqref{Ito_pro_1}--\eqref{Ito_pro_4}, we get that for $x\in \qqD$ and $(z_{1})_t\equiv z_t=E_t+\ii \eta_t$, 
\be\label{EwtLt} 
\dd\, \E (\widetilde{L}_{t})_{x}   
=  \frac{1}{2}\left(1- \frac{1}{t_c-t} \right)\E (\widetilde{L}_{t})_{x} \dd t + \frac{1}{t_c-t}  \sum_{a=1}^D
\E (\wt L_{(1,1),t})_{xa} (\wt L_t-\wt K_t)_a 
\dd t ,
\ee
where recall that $\wt L_{(1,1),t}$ was defined in \eqref{eq:wtL12t} with $(z_1)_t=z_t$. By \eqref{eq:aver_local}, \eqref{flow_1_result_2}, and \eqref{t_eta}, we have that for $t\in [t_0,t_c- N^{-1+C\e_g}]$, 
$$
(\wt L_t-\wt K_t)_a\prec\frac{1}{N (t_c-t)^{1/2}},\quad 
 \wt L_{(1,1),t} -\wt K_{(1,1),t}\prec \frac{N^{\e_g}}{N(t_c-t) }.
$$
Again, these estimates hold uniformly in $t$ due to the $N^{-C}$-net argument. 
Therefore, \eqref{EwtLt} now writes
$$ 
\begin{aligned}
\dd  \,\E \widetilde{L}_{t }   
&= \frac{1}{2}\left(1- \frac{1}{t_c-t} \right)\widetilde{K}_{t}\dd t + \frac{1}{t_c-t}T_t \left(\E\wt L_t-\wt K_t\right) \dd  t +\OO_\prec \left(\frac{N^{\e_g}}{N^2 (t_c-t)^{5/2}}\right)\dd t,
\end{aligned}
$$
 where $T_t:\C^D\to \C^D$ is a linear operator defined as 
$$ 
\begin{aligned}
 T_t (V) 
 = \wt K_{(1,1),t}V -\frac{1}{2}\left[1  - (t_c-t)\right] V ,\quad V\in \C^D.
\end{aligned}
$$
On the other hand, with \eqref{devM}, we can derive that 
$$\dd \widetilde{K}_{t }   
=  \frac{1}{2}\left(1- \frac{1}{t_c-t} \right)\widetilde{K}_{t}\dd t. $$
Subtracting the above two differential equations and denoting $\wt\LK_t:=\E \wt L_t-\wt K_t$, we get 
\be\label{eq:SDE_simple}
\wt\LK_t -\wt\LK_{t_0} = \int_{t_0}^t T_s (\wt\LK_s) \frac{\dd  s}{t_c-s} +\cal E_t,
\ee
where $\cal E_t$ is a $D$-dimensional random vector satisfying $\|\cal E_t\|_2 \prec  {N^{\e_g}}/[{N^2 (t_c-t)^{3/2}}]$. 
Since $\|\wt K_{(1,1),t}\|=\oo(1)$, we easily see that $\|T_t\|_{op}\le 1/2+\oo(1)$.  Again, using the argument with Peano-Baker series, we obtain from \eqref{eq:SDE_simple} that 
$$
 \|\widetilde{L }_{t  }- \widetilde{K }_{t}\|_2
 \prec  \frac{(t_c-t_0)^{1/2}}{(t_c-t)^{1/2}} \|\widetilde{L }_{t_0  }- \widetilde{K }_{t_0  }\|_2+\frac{N^{\e_g}}{N^2(t_c-t)^{3/2}},
$$
which implies \eqref{flow_1G_res} and completes the proof of \Cref{lem_flow_1.5}. 

\subsection{Proof of \Cref{lem_flow_3}} 

Finally, in this subsection, we complete the proof of \Cref{lem_flow_3} with Lemmas \ref{lem_flow_2} and \ref{lem_flow_1.5} and the proof of \Cref{lem_flow_1} as main inputs. Taking the expectation of \eqref{ItowtL}, we obtain that 
\be  \label{eq:ESDELt}
\frac{\dd}{\dd t} \E \widetilde{L}_{t}
=   \left(1 - \frac{1}{t_c-t}\right)\E  \widetilde{L}_{t} + \frac{1}{t_c-t} \E (\widetilde{L}_{t })^2      +\OO_\prec\left(\frac{N^{-\e_g}}{N (t_c-t)^2}+\frac{N^{\e_g}}{N^2(t_c-t)^3}\right). 
\ee
In the derivation of this equation, we have applied \eqref{flow_3G_res_2} and \eqref{eq:aver_local} to get that 
\begin{align*}
  \E\left\langle \left(G_{1,t}-M_{1,t} \right) E_a\right\rangle [L_{(1,2,1),t}]_{xya}&=\E\left\langle \left(G_{1,t}-M_{1,t} \right) E_a\right\rangle [K_{(1,2,1),t}]_{xya}+\OO_\prec\left( \frac{N^{\e_g}}{N^2(t_c-t)^4}\right)\\
  &\prec \frac{N^{-\e_g}}{N (t_c-t)^3}+\frac{N^{\e_g}}{N^2(t_c-t)^4},
\end{align*}
where in the second step we have used \eqref{flow_1G_res_2} to control $\E \langle \left(G_{1,t}-M_{1,t} \right) E_a\rangle$ and the bound $|(K_{(1,2,1),t})_{xya}| \lesssim  \eta_t^{-2}\lesssim (t_c-t)^{-2} $ due to the estimates \eqref{1-M} and \eqref{1-M-2} below. The other term $ \E \left\langle \left(G_{2,t}-M_{2,t} \right) E_a\right\rangle [L_{(2,1,2),t}]_{yxa}$ can be bounded in a similar way.

Next, the estimate \eqref{flow_1_result_2} implies that for any $x,y,a\in \qqD$, 
\be\nonumber
\E (\widetilde{L}_{t})_{x a} (\widetilde{L}_{t})_{a y}
-\E (\widetilde{L}_{t})_{x a} \E (\widetilde{L}_{t})_{a y}
 \prec \frac{N^{2\e_g}}{N^2(t_c-t)^2} .
 \ee
Plugging this estimate into \eqref{eq:ESDELt}, we can improve the expectation of equation \eqref{wtLKint} to
$$
 \E \widetilde{\LK}_{t }- \E \widetilde{\LK}_{t_0 }
 =\int_{t_0}^{t } \cal T_s(\E \widetilde{\LK}_s) 
 \frac{\dd s}{t_c-s}+\cal E_{t },\quad \cal E_{t}=\OO_\prec\left(\frac{N^{-\e_g}}{N (t_c-t)}+\frac{N^{2\e_g}}{N^2(t_c-t)^2 }\right) .
 $$ 
Then, following the proof below \eqref{eq:deltat}, we can derive that
$$
\|\E \widetilde{\LK}_{t}\|_2\prec \frac{t_c-t_0}{t_c-t} \|\E \widetilde{\LK}_{t_0}\|_2+ \frac{N^{-\e_g}}{N(t_c-t)}+\frac{N^{2\e_g}}{N^2(t_c-t)^2 } ,
$$
which implies \eqref{flow_3_result} and completes the proof of \Cref{lem_flow_3}.

\section{Chaotic regime: eigenvalues}\label{sec:mix_evalue}

Consider the matrix OU process $\VV(t)=H_t +\Lambda$, where $H_t=(h_{ij}(t))_{i,j\in \cI}$ satisfies the OU equation 
\be\label{defVt}
\dd h_{i j}=-\frac{1}{2} h_{i j} \dd t+\frac{1}{\sqrt{DN}} \dd b_{i j}(t),\quad \text{with}\quad H_{0}=H,
\ee
where $B_t=(b_{ij}(t))_{i,j\in \cI}$ denotes a Hermitian matrix whose upper triangular entries are independent complex Brownian motions with variance $t$. Then, Theorem \ref{MixEV} follows immediately from Lemmas \ref{CVtG} and \ref{CVtV} below. 

\begin{lemma}\label{CVtG} 
Under the assumptions of Theorem \ref{MixEV}, suppose $\ft= N^{-1+\fc}$ for a constant $\fc \in (0,1/10)$. Then, for any fixed $n\in \N$, there exist a constant $c_n=c_n(\fc,\delta_A)>0$ such that 
\begin{equation}\label{ls_CVtG}
\left|\int_{\mathbb{R}^n} \mathrm{~d} \boldsymbol{\alpha}\; O(\boldsymbol{\alpha})
\left[p_{\VV(\ft)}^{(n)} \left(E+\frac{\boldsymbol{\alpha}}{DN}\right)
-p_{GUE}^{(n)} \left( E+\frac{\boldsymbol{\alpha}}{DN}\right)
\right]\right|
\le  N^{-c_n}   , 
\end{equation}
where $p_{\VV(\ft)}^{(n)}$ denotes the $n$-point correlation function of $\VV(\ft)$ and we have abbreviated that $E+\frac{\boldsymbol{\alpha}}{DN}=\left(E+\frac{\alpha_1}{DN  }, \ldots, E+\frac{\alpha_n}{DN  }\right).$
\end{lemma}

\begin{proof} 
This lemma is a simple consequence of Theorem 2.2 in \cite{LANDON20191137}. More precisely, $H_t$ in \eqref{defVt} has law 
\be\label{eq:eq_in_law_H}
H_t \stackrel{d}{=} e^{-t/2}\cdot H+ \sqrt{1-e^{-t}}\cdot W,
\ee
where $\stackrel{d}{=}$ means ``equal in distribution" and $W$ is a $DN\times DN$ GUE independent of $H$. Since the matrix $H+e^{\ft/2}\Lambda$ also satisfies \Cref{main_assm}, its resolvent satisfies the averaged local law \eqref{eq:aver_local} with $\Lambda$ replaced by $e^{\ft/2}\Lambda$. Then, applying \cite[Theorem 2.2]{LANDON20191137} with $V=e^{-\ft/2}H+\Lambda$ gives that
\be\label{eq:compare_statistics}\int_{\mathbb{R}^n} \mathrm{~d} \boldsymbol{\alpha}\; O(\boldsymbol{\alpha})
\left[p_{\VV(\ft)}^{(n)} \left(E+\frac{\boldsymbol{\alpha}}{DN\wh\rho_{\mathrm{fc},\ft}(E)}\right)
-p_{GUE}^{(n)} \left( E+\frac{\boldsymbol{\alpha}}{DN\rho_{sc}(E)}\right)
\right]\le N^{-c}
\ee
for a constant $c>0$, where $\wh\rho_{\mathrm{fc},\ft}$ denotes the density for the free convolution of the empirical measure of $V$ and the semicircle law with variance $1-e^{-\ft}$. More precisely, $\wh\rho_{\mathrm{fc},\ft}$ can be defined through its Stieltjes transform $\wh m_{\mathrm{fc},\ft}(z)$, which is the unique solution to 
$$ \wh m_{\mathrm{fc},\ft}(z) = \frac{1}{DN}\tr \frac{1}{V - z - (1-e^{-\ft})\wh m_{\mathrm{fc},\ft}(z)}\quad \text{with}\quad \im \wh m_{\mathrm{fc},\ft}(z) > 0 \ \text{whenever} \ \im z> 0. $$
Using the averaged local law for $H+e^{\ft/2}\Lambda$, we get that 
\begin{align*}
    \wh m_{\mathrm{fc},\ft}(z) &=e^{\ft/2} m\big(e^{\ft/2}z + (e^{\ft/2}-e^{-\ft/2})\wh m_{\mathrm{fc},\ft}(z),e^{\ft/2}\Lambda\big)+\OO_\prec \big((N\ft)^{-1}\big) \\
    &=m_{sc}(z)+\OO_\prec \big((N\ft)^{-1}+\|A\|+\ft\big)
\end{align*} 
for $z=E+\ii \eta$ with $E\in [-2+\kappa,2-\kappa]$ and $ 0\le \eta\le 1$, where in the second step, we used that $m$ is well-approximated by $m_{sc}$ due to the estimate  \eqref{eq:msc0.5} below. Taking the imaginary part of the above equation when $\eta=0$, we obtain that $\wh\rho_{\mathrm{fc},\ft}(E)=\rho_{sc}(E)+\OO_\prec(N^{-\fc\wedge \delta_A})$. Plugging it into \eqref{eq:compare_statistics}, applying the change of variables, and using the smoothness of the test function $O$, we conclude \eqref{ls_CVtG}.
\end{proof}

\begin{lemma}\label{CVtV} 
Under the assumptions of Theorem \ref{MixEV}, there exists a constant $\fc>0$ depending on $\e_A$ such that the following holds for $\ft= N^{-1+\fc}$.  
For any fixed $n\in \N$, there exists a constant $c_n=c_n(\fc,\e_A)$ such that 
\begin{equation}\label{ls_CVtV}
\int_{\mathbb{R}^n} \mathrm{~d} \boldsymbol{\alpha}\; O(\boldsymbol{\alpha})
\left(p_{\VV(\ft)}^{(n)}-p_{\VV}^{(n)} \right) 
\left(E+\frac{\boldsymbol{\alpha}}{N }\right) 
\leqslant N^{-c_n} . 
\end{equation}
\end{lemma}

The rest of this section is devoted to the proof \Cref{CVtV}. We will use the following correlation function comparison theorem stated in \cite[Theorem 15.3]{Erds2017ADA}, which is a slightly modified version of \cite[Theorem 6.4]{erdHos2012bulk}. 

\begin{lemma}[Theorem 15.3 of \cite{Erds2017ADA}]\label{rescomY}
  Under the assumptions of Theorem \ref{MixEV}, let $G$ and $G_{\ft}$ denote the resolvents of $\VV$ and $\VV(\ft)$, respectively. 
 Suppose that for some small constants $\sigma, \delta>0$, the following two conditions hold.
 \begin{itemize}
     \item [(i)] Fix any $\varepsilon>0$ and $k\in \N$, there is
\be\label{eq:roughmoment}
\mathbb{E} [\im \langle G(z)\rangle] ^k
+
\mathbb{E}[\im \langle G_\ft(z)\rangle] ^k \lesssim 1,\quad \text{for}\quad z=E +\ii N^{-1+\varepsilon} \ \ \text{with}\ \ |E|\le 2-\kappa. 
\ee

\item [(ii)] For any sequence $z_j=E_j + \ii \eta_j$, $j=1, \ldots, n$, with $\left|E_j\right| \leqslant 2-\kappa$ and $\eta_j=N^{-1-\sigma_j}$ for some constants $0<\sigma_j \leqslant \sigma$, we have
\begin{equation}\label{2comVVt}
    \left|
\mathbb{E} \prod_{i=1}^n \im \langle G(z_i)\rangle 
-\mathbb{E} \prod_{i=1}^n \im \langle G_\ft(z_i)\rangle 
\right| \leqslant N^{-\delta}.
\end{equation}
 \end{itemize}
Then, for any integer $n \geqslant 1$, there is a constant $c_n=c_n(\sigma, \delta)>0$ such that for any $|E| \leqslant 2-2 \kappa$ and any $C^1$-function $O: \mathbb{R}^n \rightarrow \mathbb{R}$ with compact support,
$$
\int_{\mathbb{R}^n} \, \mathrm{~d} \boldsymbol{\alpha}\; O(\boldsymbol{\alpha})\left(p_{ \VV(\ft)}^{(n)}-p_{\VV}^{(n)}\right)\left(E+\frac{\boldsymbol{\alpha}}{N}\right) \leqslant  N^{-c_n}.
$$
\end{lemma}

Next, we bound the quantities in \eqref{eq:roughmoment} and \eqref{2comVVt} using \Cref{lem_loc} and \Cref{mix}. Note that we have only proved \Cref{lem_loc} and \Cref{mix} for $\VV$, but they can be extended to any $\VV(t)$ with $t\in [0,\ft]$. (Heuristically, adding a GUE component will ``help" the QUE of eigenvectors, so there is no essential difficulty in making this extension.) Let $M_t(z)$ be the solution to the matrix Dyson equation \eqref{def_M} with the operator $\cal S$ replaced by $\cal S_t$: 
$$ \cal S_t(M_t):=e^{-t}\cal S(M_t)+(1-e^{-t}) \langle M_t\rangle.$$
However, note that the self-consistent equation \eqref{self_m} for $m_t(z):=\langle M_t(z)\rangle$ is unchanged, so we have $m_t(z)=m(z)$ and $M_t(z)=M(z)$ as given by \eqref{def_G0}.

\begin{lemma}\label{lem:samefort}
For any $t\in [0,\ft]$, under the assumptions of \Cref{lem_loc}, the local laws \eqref{eq:aniso_local}--\eqref{eq:aver_local} holds for $G_t$ and the eigenvalue rigidity estimate \eqref{eq:rigidity} holds. Under the assumptions of \Cref{mix}, \eqref{eq:extend:main_evector1} holds for the bulk eigenvectors of $\VV(t)$. 
\end{lemma}
\begin{proof}
The local laws and eigenvalue rigidity can be proved in the same way as \Cref{lem_loc} using the methods developed in e.g., \cite{He2018,AEK_PTRF,EKS_Forum}; more details will be explained in \Cref{sec:pf_locallaw} below. The proof of \eqref{eq:extend:main_evector1} is similar to that for \Cref{mix}, and we omit the details.
\end{proof}

\begin{proof}[Proof of \Cref{CVtV}]
The estimate \eqref{eq:roughmoment} follows directly from the local law \eqref{eq:aver_local} for $G_t(z)$ established in \Cref{lem:samefort}. Now, to apply \Cref{rescomY} to conclude \Cref{CVtV}, we only need to prove the Green's function comparison estimate \eqref{2comVVt}. Using the local law of $G_t(z)$ and the fact that $\eta \im (G_t)_{xx}(E+\ii \eta)$ is an increasing function of $\eta$ for fixed $x\in \cI$ and $E\in \R$, we obtain that for any constant $\e>0$,
$$ \max_{x\in \cI}\sup_{|E|\le  2-\kappa} \sup_{N^{-1-\sigma} \le \eta\le N^{-1+\e}}\im (G_t)_{xx}(E+\ii \eta) \le \frac{N^{-1+\e}}{N^{-1-\sigma}} \max_{x\in \cI}\sup_{|E|\le  2-\kappa}\im (G_t)_{xx}(E+\ii N^{-1+\e})\prec N^{\sigma + \e}.$$
Then, using \cite[Lemma 15.5]{Erds2017ADA}, we can derive that 
\be\label{eq:local_belowN}
\max_{x,y\in \cI}\sup_{|E|\le  2-\kappa}\sup_{\eta\ge N^{-1-\sigma}}|(G_t)_{xy}(E+\ii \eta)| \prec N^{\sigma + \e}.  
\ee
Since $\im \avg{G_t(z_i)}=(\avg{G_t(z_i)}-\avg{G_t(\bar z_i)})/(2\ii\eta_i)$, to show \eqref{2comVVt}, it suffices to prove that for $z_j=E_j\pm \ii \eta_j$, $j=1, \ldots, n$, we have
\begin{equation}\label{2comVVt_2}
    \left|
\mathbb{E} \prod_{i=1}^n  \langle G(z_i)\rangle 
-\mathbb{E} \prod_{i=1}^n \langle G_\ft(z_i)\rangle 
\right| \leqslant N^{-\delta}.
\end{equation}

To prove \eqref{2comVVt_2}, we apply the It{\^o}'s formula and get that 
\begin{align*}
    &\partial_{t}\E  \prod_{i=1}^n \langle G_t(z_i)\rangle  = \frac{1}{2DN}\E\sum_{x,y\in \cI} \partial_{xy}\partial_{yx} \left(\prod_{i=1}^n \langle G_t(z_i)\rangle \right) - \frac{1}{2}\E\sum_{x,y\in \cI} h_{xy}(t)\partial_{xy} \left(\prod_{i=1}^n \langle G_t(z_i)\rangle \right),
\end{align*}
where $\partial_{xy}$ denotes the partial derivative $\partial/\partial_{h_{xy}(t)}$. Then, applying the cumulant expansion in \Cref{lem:complex_cumu} to the second term on the RHS, we get that 
\begin{align}\label{eq:prodnGt}
    &\partial_{t}\E  \prod_{i=1}^n \langle G_t(z_i)\rangle  = \frac{e^{-t}}{2}\E\sum_{x,y\in \cI} \left( \frac{1}{DN} - s_{xy}\right)\partial_{xy}\partial_{yx} \left( \prod_{i=1}^n \langle G_t(z_i)\rangle \right) + \sum_{k=3}^l \cal F_k + \cal E_{l+1} ,
\end{align}
where we used that $E|h_{xy}(t)|^2=e^{-t}s_{xy}+(1-e^{-t})(DN)^{-1}$ by \eqref{eq:eq_in_law_H} (recall that $s_{xy}$ was defined in \eqref{eq:sij}), $\cal F_k$ is the sum of terms involving the cumulants $\cal C^{(m,n)}(h_{xy}(t))$ with $m+n=k$, and $\cal E_{l+1}$ is the remainder term. We can choose $l$ sufficiently large such that $ \cal E_{l+1} \le 1$. For the terms $\cal F_k$, using the estimates in \eqref{eq:local_belowN} and the fact that $\e$ is arbitrarily small, it is easy to check that 
\be\label{eq:Fk}\cal F_k \prec N^{-k/2+2+(n+k)\sigma},\quad 3\le k\le l.\ee
It remains to bound the first term on the RHS of \eqref{eq:prodnGt}.
To simplify notations, for a fixed $t\in [0,\ft]$, we abbreviate $G_i\equiv G_t(z_i)$. Then, we can write the first term on the RHS of \eqref{eq:prodnGt} as $e^{-t}$ times
\begin{align*}
\cal F_2:=& \frac{1}{DN^2}\E\sum_{x,y\in \cI} \left(  {D}^{-1} -N s_{xy}\right)\bigg\{ \sum_{i\in \qq{n}}(G_i^2)_{xx}(G_i)_{yy} \prod_{k\ne i} \langle G_k\rangle + \frac{1}{DN} \sum_{i< j\in \qq{n}}(G_i^2)_{xy}(G_j^2)_{yx} \prod_{k\notin \{i,j\}} \langle G_k\rangle\bigg\} \\
=& D \E\sum_{a\in \qqD} \langle G_i^2(D^{-1}-E_a) \rangle \langle G_iE_a \rangle \prod_{k\ne i} \langle G_k\rangle +\frac{1}{DN^2}\E\sum_{i< j\in \qq{n}} \langle G_i^2 (D^{-1}-E_a)G_j^2 E_a\rangle \prod_{k\notin \{i,j\}} \langle G_k\rangle.
\end{align*}
It remains to bound the two terms 
\begin{align*}
	X(i;a)&:=\langle G_i^2(D^{-1}-E_a) \rangle \langle G_iE_a \rangle \prod_{k\ne i} \langle G_k\rangle,\\
	Y({i,j;a})&:=N^{-2}\langle G_i^2 (D^{-1}-E_a)G_j^2 E_a\rangle \prod_{k\notin \{i,j\}} \langle G_k\rangle.
\end{align*}

With the bound \eqref{eq:local_belowN}, we can derive that  
\begin{align}\label{eq:roughXY0}
	\langle G_k\rangle \prec N^\sigma,\quad \langle G_iE_a \rangle\prec N^{\sigma},\quad \langle G_i^2(D^{-1}-E_a) \rangle \prec N^{1+2\sigma},\quad \langle G_i^2 (D^{-1}-E_a)G_j^2 E_a\rangle \prec N^{3+4\sigma}. 
\end{align}
With these estimates, we get the following rough bounds on $X$ and $Y$:
\begin{align}\label{eq:roughXY}
	X(i;a) \prec N^{1+(n+2)\sigma},\quad 
	Y({i,j;a}) \prec N^{1+(n+2)\sigma}.
\end{align}
To improve these estimates, we consider the eigendecompositions 
   \begin{align}
        \left\langle G_i^2 \left(D^{-1}-E_a\right) \right\rangle &=\frac1{DN}\sum_k \frac{ {\bf v}_k^* (D^{-1}-E_a){\bf v}_k }{(\lambda_k-z_i)^2},\label{GEaD1}\\
        \left\langle G_i^2 \left(D^{-1}-E_a\right)G_j^2 E_a \right\rangle
      & =\frac1{DN} \sum_{k,l} \frac{  {\bf v}_k^* (D^{-1}-E_a){\bf v}_l \cdot  {\bf v}_l^* E_a{\bf v}_k}
       {(\lambda_k-z_i)^2(\lambda_l-z_j)^2},\label{GEaD2}
   \end{align}
where $\lambda_k\equiv \lambda_k(t)$ and $\bv_k\equiv \bv_k(t)$ denote the eigenvalues and eigenvectors of $H_t + \Lambda$, respectively. Now, let $k_0\in \cal I$ be the integer such that $|\gamma_{k_0}-E_i|=\min_{k\in \cI}|\gamma_k-E_i|$. Using the eigenvalue rigidity \eqref{eq:rigidity} and the QUE estimate \eqref{eq:extend:main_evector1} for $\VV(t)$ in \Cref{lem:samefort}, we can bound \eqref{GEaD1} as follows: with probability $1-\OO(N^{-c})$, 
\begin{align}
	\eqref{GEaD1}
	&\lesssim \frac1{DN}\sum_{|k-k_0|\le N^\e} \frac{N^{-c} }{\eta_i^2} +\frac1{DN}\sum_{N^\e < |k-k_0|\le N^c} \frac{N^{-c} }{|k-k_0|^2/N^2} + \frac1{DN}\sum_{|k-k_0|> N^c} \frac{1 }{|k-k_0|^2/N^2} \nonumber\\
	&\lesssim N^{1-c+2\sigma+\e},\label{GEaD3}
\end{align}
for any small constant $\e \in (0,c)$. Similarly, we can bound \eqref{GEaD2} as 
\begin{equation}\label{GEaD4}
\P\left(\left|\langle G_i^2 \left(D^{-1}-E_a\right)G_j^2 E_a \rangle\right| \ge N^{3-c+4\sigma+\e}\right)  \lesssim N^{-c} .
\end{equation}
Combining \eqref{GEaD3} and \eqref{GEaD4} with \eqref{eq:roughXY0}, we obtain that for any constant $\e>0$, 
\[\mathbb P\left( |X(i;a)|+ |Y(i,j;a)|\ge N^{1-c+(n+2)\sigma+\e} \right)\le N^{-c}.\]
Together with the rough bound \eqref{eq:roughXY}, it yields that 
\be\label{eq:F2}
\cal F_2 \lesssim N^{1-c+(n+2)\sigma+\e} +  N^{1+(n+2)\sigma + \e}\cdot N^{-c} \le  2N^{1-c+(n+2)\sigma+\e}.
\ee

Finally, plugging \eqref{eq:Fk} and \eqref{eq:F2} into \eqref{eq:prodnGt}, we conclude that for any small constant $\e >0$,
\begin{align*} 
	&\left|\partial_{t}\E  \prod_{i=1}^n \langle G_t(z_i)\rangle \right| \le  N^{1-c+(n+2)\sigma+\e} + N^{1/2+(n+3)\sigma+\e},\quad t\in [0,\ft],
\end{align*}
if we choose $\sigma\in (0,1/2)$. Integrating over $t$, we conclude \eqref{2comVVt} with $\delta = (c-(n+2)\sigma-\e-\fc)\wedge (1/2-(n+3)\sigma-\e-\fc)$ as long as we choose $\fc$, $\sigma$, and $\e$ sufficiently small depending on $c$.     
\end{proof}

\section{Integrable regime: eigenvectors}\label{sec:localization}

In this section, we prove \Cref{nonmix}. For notational convenience, we will assume $D=2$ in the subsequent proof, although the argument remains the same for the general case of $D\ge 2$.

Take the $k$-th bulk eigenvalue $\lambda_k$ with $k\in \qq{DN, (1-\kappa)DN}$. 
Denote the corresponding eigenvector by $\bv_k=\begin{pmatrix}
\mathbf{u}_k \\
\mathbf{w}_k
\end{pmatrix}$. 
Then, we have the eigenvalue equation
\[
H \begin{pmatrix}
\mathbf{u}_k \\
\mathbf{w}_k
\end{pmatrix}=\begin{pmatrix}
H_1 & A \\
A^\dagger & H_2
\end{pmatrix}
\begin{pmatrix}
\mathbf{u}_k \\
\mathbf{w}_k
\end{pmatrix}
=
\lambda_k
\begin{pmatrix}
\mathbf{u}_k \\
\mathbf{w}_k
\end{pmatrix}.
\]
From this equation, we derive that  
$$\bw_k = -\cG_2(\lambda_k)  A^\dagger\bu_k, \quad \bu_k= -\cG_1(\lambda_k) A\bw_k,$$
where we denote the resolvents of $H_1$ and $H_2$ by  
$$\cG_1(z):=(H_1-z)^{-1}, \quad \cG_2(z):=(H_2-z)^{-1}.$$

Now, given an arbitrarily small constant $\delta>0$, we define the events 
$$\mathscr E_1:=\left\{\dist(\lambda_k, \spec(H_1))\ge N^{-1-\delta}\right\},\quad \mathscr E_2:=\left\{\dist(\lambda_k, \spec(H_2))\ge N^{-1-\delta}\right\}.$$
We claim that 
\be\label{eq:probab12}
\P \left( \mathscr E_1\cup \mathscr E_2\right) = 1-\OO(N^{-\delta/2}).
\ee
To prove this claim, notice that 
$$ \P\left( (\mathscr E_1\cup \mathscr E_2)^c \right) \le \P \left( \exists i,j \in \qq{N} \text{ such that } |\lambda_i^{(1)}-\lambda_j^{(2)}|\le 2N^{-1-\delta}\right),$$
where $\lambda_i^{(1)}$ and $\lambda_j^{(2)}$ denote the eigenvalues of $H_1$ and $H_2$, respectively. Using the rigidity of eigenvalues for Wigner matrices \cite[Theorem 2.2]{erdHos2012rigidity} (or using \eqref{eq:rigidity} in the case of $D=1$), we get
\be\label{rigidity12} |\lambda_i^{(1)}-\gamma_i^{sc}|+|\lambda_i^{(2)}-\gamma_i^{sc}| \prec N^{-2/3}\min(i,N+1-i)^{-1/3},\quad i\in \qq{N}, \ee
where $\gamma_i^{sc}$, $i\in \qq{N}$, denote the quantiles of the semicircle law:
 $$\gamma_i^{sc}:=\sup_{x\in \R}\left\{\int_{-\infty}^{x}\rho_{sc}(x)\dd x < \frac{i}{N}\right\}.$$
By the bulk universality of $H_1$, there exists a constant $\delta>0$ such that for any constant $\tau>0$ and $j\in \qq{N}$, 
\be\label{Wegner12} \P \left( \exists i\in \qq{N}, \left. |\lambda_i^{(1)}-\lambda_j^{(2)}|\le 2N^{-1-\delta} \right| H_2\right) \le N^{-\delta+\tau}.\ee
In fact, \cite[Proposition B.1]{bourgade2016fixed} shows that \eqref{Wegner12} holds for a Gaussian divisible ensemble; then, applying the Green's function comparison theorem \cite[Theorem 15.3 and Theorem 16.1]{Erds2017ADA} concludes \eqref{Wegner12}. 
Let $k_0:=\mathrm{arg\,min}_{i=1}^N|\gamma_{i}^{sc}-\gamma_k|$. Now, using \eqref{rigidity12} and \eqref{Wegner12}, we obtain that for any constants $\tau,C>0$, 
\begin{align*}
 \P\left( (\mathscr E_1\cup \mathscr E_2)^c \right) &\le \P \left( {\exists i,j \in \qq{k_0-N^\tau , k_0+ N^\tau}  \text{ such that } |\lambda_i^{(1)}-\lambda_j^{(2)}|\le 2N^{-1-\delta} }\right) + \OO(N^{-C})\\
&\le \sum_{j\in \qq{k_0-N^\tau,k_0+ N^\tau}} \P \left(\exists i\in \qq{N}, |\lambda_i^{(1)}-\lambda_j^{(2)}|\le 2N^{-1-\delta}\right) + \OO(N^{-C}) =\OO(N^{-\delta+2\tau}).
\end{align*}
Taking $\tau<\delta/4$ concludes \eqref{eq:probab12}.

Without loss of generality, suppose $\mathscr E_1$ holds. Then, we claim the following estimate. 
\begin{lemma}\label{prop:local}
Under the assumptions of \Cref{nonmix} (with $D=2$) and the notations defined above, the following estimate holds for any constants $\e,C>0$:
\begin{align}
    \label{eq:E1}
    &\P \left(\| \cG_1(\lambda_k) A \bw_k\|\ge N^{\delta+\e} \|A\|_{\HS} + N^{-C}; \mathscr E_1\right) \le N^{-\e}.
\end{align}
\end{lemma}
\begin{remark}
The intuition behind the estimates \eqref{eq:E1} is that we believe $\bw_k$ is delocalized in the eigenbasis of $H_1$ on $\mathscr E_1$, since $H_1$ and $\bw_k$ should be almost independent in the integrable regime. As a consequence, we expect that on $\mathscr E_1$, 
\begin{align*}
    \|\bu_k\|^2 &\prec  \frac{1}{N}\tr\left[A^\dag  \cG_1(\lambda_k)^* \cG_1(\lambda_k) A\right] \prec \|A\|_{\HS}^2 \max_i \frac{(\bu^a_i)^* \cG_1(\lambda_k)^* \cG_1(\lambda_k)\bu^a_i }{N}  \\
    & \prec \frac{\|A\|_{\HS}^2}{N^2} \sum_{i=1}^N \frac{1}{|\lambda_i^{(1)}-\lambda_k|^2}\prec {N}^{2\delta} {\|A\|_{\HS}^2},
\end{align*}
where $\bu_i^a$ are the eigenvectors of $AA^*$. Above, in the second step, we used the definition of $\|A\|_{\HS}^2$ in terms of the singular values of $A$; in the third step, we used the delocalization of $\bu^a_i$ in the eigenbasis of $H_1$; in the last step, we used the definition of $\mathscr E_1$ and the rigidity of eigenvalues \eqref{rigidity12}. 
\end{remark}

\begin{proof}[Proof of \Cref{prop:local}]
Let $z=E+\ii \eta$ with $E=\gamma_k$ and $\eta=N^{-1+c}$ for a small constant $c\in (0,1/2)$. We claim that  
\begin{equation}\label{eq:boundE1}
    \E\left(\| \cG_1(\lambda_k) A \bw_k\|^2;\mathscr E_1\right)\lesssim N^{2(c+\delta)} \E \tr\left[\begin{pmatrix}
0 & 0 \\
0 & A^*  \im \cG_1(z) A
\end{pmatrix}
 \im G(z)\right] .
\end{equation}
To see why \eqref{eq:boundE1} holds, with the spectral decomposition of $\im G$, we obtain that 
\begin{align*}
    \E \tr\left[\begin{pmatrix}
0 & 0 \\
0 & A^*  \im \cG_1(z) A
\end{pmatrix}
 \im G(z)\right] =\E \sum_{j\in \cI} \frac{\eta}{(\lambda_j-\gamma_k)^2+\eta^2} \bw_j^*   A^*  \im \cG_1(z) A \bw_j\\
\ge  \E   \frac{\eta}{(\lambda_k-\gamma_k)^2+\eta^2} \bw_k^*   A^*  \im \cG_1(z) A \bw_k  \gtrsim \eta^{-1} \E \bw_k^*   A^*  \im \cG_1(z) A \bw_k ,
\end{align*}
where in the last step, we used the rigidity of $\lambda_k$ given by \eqref{eq:rigidity}. On the other hand, with the spectral decomposition of $\cG_1$, we obtain that on the event $\mathscr E_1$, with high probability,
\begin{align*}
\eta^2\| \cG_1(\lambda_k) A \bw_k\|^2 &= \sum_{j} \frac{\eta^2|(\bu_j^{(1)})^* A \bw_k|^2}{(\lambda_j^{(1)}-\lambda_k)^2} \lesssim N^{2(c+\delta)}\sum_{j} \frac{\eta^2|(\bu_j^{(1)})^* A \bw_k|^2}{(\lambda_j^{(1)}-\lambda_k)^2+\eta^2} \\
&\lesssim N^{2(c+\delta)}\sum_{j} \frac{\eta^2|(\bu_j^{(1)})^* A \bw_k|^2}{(\lambda_j^{(1)}-\gamma_k)^2+\eta^2} = N^{2(c+\delta)} \cdot \eta \bw_k^*   A^*  \im \cG_1(z) A \bw_k ,
\end{align*}
where $\bu_j^{(1)}$, $j\in \qq{N}$, denote the eigenvectors of $H_1$, and we used the definition of $\mathscr E_1$ in the second step and the rigidity of $\lambda_k$ in the third step. The above two estimates together conclude \eqref{eq:boundE1}.

Using \eqref{eq:boundE1} and the trivial bound  
$$ \E \tr\left[\begin{pmatrix}
0 & 0 \\
0 & A^*  \im \cG_1(z) A
\end{pmatrix}
 \im G(z)\right] 
\le \E \tr\left[\im G_0(z) \Lambda \im G(z) \Lambda \right]\quad \text{with}\quad G_0(z):=(H-z)^{-1},$$
we obtain that 
\begin{align}\label{eq:boundE1.5}
    \E\left(\| \cG_1(\lambda_k) A \bw_k\|^2;\mathscr E_1\right)\lesssim N^{2(c+\delta)} \E \tr\left[\im G_0(z) \Lambda \im G(z) \Lambda \right].
\end{align}
The estimate \eqref{eq:E1} then follows from the next lemma together with Markov's inequality. 

\begin{lemma}\label{lem:localization}
    Under the setting of \Cref{nonmix} (without requiring that $D=2$), let $z=E+\ii \eta$ with $E=\gamma_k$ and $\eta=N^{-1+c}$. Then, we have that for any constant $C>0$,
    \begin{align}\label{key_local_bdd}
        \E \tr\left[\im G_0(z) \Lambda \im G(z)\Lambda \right] \prec   
  \|A\|_{\HS}^2 + N^{-C}.
    \end{align}
\end{lemma}
 The proof of this lemma is the same as that for \Cref{lem:nonmix_key2} below, which will be presented in full detail in \Cref{sec_nonmix_evalue}. Hence, we omit the proof of \Cref{lem:localization} here.
 
Combining \eqref{key_local_bdd} with \eqref{eq:boundE1.5} gives that 
\begin{equation*}
    \E\left(\| \cG_1(\lambda_k) A \bw_k\|^2;\mathscr E_1\right) 
    \prec N^{2(c+\delta)}(\|A\|^2_{\HS}+N^{-C}).
\end{equation*}
Then, applying Markov's inequality concludes \eqref{eq:E1} since $c$ is arbitrarily small.
\end{proof}

Now, we are ready to conclude \Cref{nonmix}.
\begin{proof}[Proof of \Cref{nonmix}]
   When $D=2$, by \Cref{prop:local}, we have that for any constant $\e>0$,   
   $$\P \left(\| \bu_k\| \ge N^{-\e_A+\delta+\e}; \mathscr E_1\right) \le N^{-\e}.$$
   By symmetry, a similar estimate holds for $\bw_k$ on $\mathscr E_2$. Together with \eqref{eq:probab12}, these estimates imply that 
   $$\P \left(\min\{\| \bu_k\|, \|\bw_k\|\}\ge N^{-\e_A+\delta+\e}\right) \lesssim N^{-\e}+N^{-\delta/2},$$
   which concludes \eqref{eq:main_evector2} as long as we choose $\delta+\e<\e_A$.  
   
    For the general case of $D\ge 2$, the proof is similar. More precisely, define events $\mathscr E_{(a)}$, $a\in \qqD$, as 
    $$\mathscr E_{(a)}:=\left\{\dist\left(\lambda_k, \cup_{k\in \qqD\setminus \{a\}}\spec(H_k)\right)\ge N^{-1-\delta}\right\}.$$
    Again, using the eigenvalue rigidity and bulk universality for Wigner matrices, we get that $\P(\cup_{a=1}^D \mathscr E_{(a)}) = 1-\OO(N^{-\delta/2})$. Then, similar to \eqref{eq:boundE1.5}, we can show that 
    $$ \sum_{k\in \qqD\setminus\{a\}}\E \left(\|E_k \bv_k\|^2;\mathscr E_{(a)}\right)\lesssim N^{2(c+\delta)} \E \tr\left[\im G_0(z) \Lambda \im G(z)\Lambda \right]. $$
    Applying \Cref{lem:localization} and Markov's inequality, we obtain that for any constant $\e>0$,
    $$ \P\left( 1-\|E_a \bv_k\|^2\ge N^{-2\e_A+2\delta+\e};\mathscr E_{(a)}\right)\lesssim N^{-\e/2}+N^{-\delta/2}.$$
    Finally, taking the union bound over $a\in \qqD$ concludes the proof. 
\end{proof}

\section{Integrable regime: eigenvalues}\label{sec_nonmix_evalue}
Finally, in this section, we give the proof of \Cref{NonMixEV}. Recall the interpolating matrices defined in \eqref{eq:VVtheta}. 
Denote the eigenvalues and corresponding eigenvectors of $\VV(\theta)$ by $\lambda_i(\theta)$ and ${\bf v}_i(\theta)$, $i\in \cI$. 
Then, we have that for any $k\in \cI$,
\begin{equation} \label{ygfyz}
    \lambda_k(1)-\lambda_k(0)
=\int_0^1 \frac{\dd}{\dd\theta}\lambda_{k}(\theta) \dd\theta
=\int_0^1  {\bf v}_{k}(\theta)^* \Lambda {\bf v}_{k}(\theta) \dd\theta,
\end{equation}
from which we derive that
\begin{equation} \label{ygfyz2}
\mathbb E\left|\lambda_k(1)-\lambda_k(0)\right|
\le 
\int_0^1 \mathbb E \left|{\bf v}_{k}(\theta)^* \Lambda {\bf v}_{k}(\theta) \right| \dd\theta
\le \int_0^1 \left( \mathbb E \left| {\bf v}_{k}(\theta)^* \Lambda {\bf v}_{k}(\theta) \right|^2 \right)^{1/2} \dd\theta.
\end{equation}
\begin{remark*}
Rigorously speaking, ${\dd\lambda_{k}(\theta)}/{\dd\theta}={\bf v}_{k}(\theta)^* \Lambda {\bf v}_{k}(\theta)$ holds when $\lambda_k(\theta)$ is not a degenerate eigenvalue of $\VV(\theta)$. To handle this issue, we can add a small Gaussian component to $H$ so that each entry of the diagonal blocks has a continuous distribution. For example, we can take a matrix $H_{\Lambda}'$ whose blocks satisfy the law \eqref{eq:Gauss_div} with $t-t_0=e^{-N}$. Then, almost surely, all the eigenvalues of $H_{\Lambda}'$ are non-degenerate, which, together with Fubini's theorem, implies that  
$$\E[\lambda_k(1)-\lambda_k(0)] =\int_0^1  \E {\bf v}_{k}(\theta)^* \Lambda {\bf v}_{k}(\theta) \dd\theta.$$
Hence, \eqref{ygfyz2} remains valid for the ensemble $H_{\Lambda}'$. After showing \eqref{eq:main_perurb} for $H_{\Lambda}'$, we then remove the small Gaussian component, which leads to a negligible error of order $\OO_\prec(e^{-N})$ to both $\lambda_k$ and $\lambda_k(H)$. For clarity and without loss of generality, it suffices to assume that \eqref{ygfyz2} holds throughout the following proof. 
\end{remark*}

Given \eqref{ygfyz2}, we will prove that there exists a constant $\e_A'>0$ (depending on $\e_A$) such that for any $\theta\in [0,1]$ and $k\in \qq{\kappa DN, (1-\kappa)DN}$, 
\begin{equation} \label{eq:nonmix_key}
\left( \mathbb E \left| {\bf v}_{k}(\theta)^* \Lambda {\bf v}_{k}(\theta) \right|^2 \right)^{1/2} \le N^{-1-\e_A'}.
\end{equation}
Plugging this bound into \eqref{ygfyz2} and applying Markov's inequality, we get that 
$$\P\left( \left|\lambda_k(1)-\lambda_k(0)\right| \ge N^{-1-\e_A'/2}\right)\le N^{-\e_A'/2},$$
which concludes the proof.
\begin{remark}
We now explain heuristically why the estimate \eqref{eq:nonmix_key} should hold. Fix any $\theta\in [0,1]$, suppose ${\bf v}_{k}\equiv {\bf v}_{k}(\theta)$ is decomposed into $D$ sub-vectors ${\bf v}_{k,i}:=(v_k(x):x\in \cI_i)$, $i\in\qqD$. Then, ${\bf v}_{k}^* \Lambda {\bf v}_{k}$ involves inner products ${\bf v}_{k,i}^* A {\bf v}_{k,j}$, $j\ne i$, and their complex conjugates. If we believe that the eigenvectors of different blocks $H_i$, $i\in \qqD$, do not mix (which is partially suggested by \Cref{nonmix}), then ${\bf v}_{k,i}$ and ${\bf v}_{k,j}$ should be approximately independent  random vectors, which implies that  ${\bf v}_{k,i}^* A {\bf v}_{k,j} \prec  {N}^{-1} \|A\|_{\HS} \le N^{-1-\e_A}$.
\end{remark}

The rest of this section is devoted to the proof of \eqref{eq:nonmix_key}. Without loss of generality, we can take $\theta=1$ in the following proof---for a general $\theta$, we can just take $\theta\Lambda$ as our new $\Lambda$. With the spectral decomposition of $\im G$, we get that 
$$
\operatorname{Tr}  \left[  (\operatorname{Im} G) \Lambda (\operatorname{Im} G) \Lambda \right]=
\sum_{k,l\in \cI}\frac{\eta^2 \left| {\bf v}_{k}^* \Lambda {\bf v}_l\right|^2 }{  |z-\lambda_k|^2|z-\lambda_l|^2   },\quad  z=E+\ii\eta.
$$
Let $E=\gamma_k$ and $\eta=N^{-1+c}$ for a small constant $c>0$. Then, we obtain from the above equation that 
\begin{equation}\label{bdwgy}
\E|{\bf v}_{k}^* \Lambda{\bf v}_{k}|^2  
\le \E \frac{\left[(\lambda_k-\gamma_k)^2+\eta^2\right]^2}{\eta^2} \tr\left[  (\operatorname{Im} G) \Lambda (\operatorname{Im} G) \Lambda \right] \prec \eta^2 \E  \tr\left[  (\operatorname{Im} G) \Lambda (\operatorname{Im} G) \Lambda \right] ,
\end{equation}
where we used the eigenvalue rigidity \eqref{eq:rigidity} in the second step. Now, to conclude \eqref{eq:nonmix_key}, it suffices to show that for $z=\gamma_k+\ii N^{-1+c}$,
$$
\E  \tr\left[  (\operatorname{Im} G) \Lambda (\operatorname{Im} G) \Lambda \right] \prec \|A \|^2_{\HS} + N^{-1}.
$$
Writing $\im G=(G-G^*)/(2\ii)$, it suffices to prove the following lemma.

\begin{lemma}\label{lem:nonmix_key2}
Under the setting of \Cref{NonMixEV}, let $c\in (0,1)$ be an arbitrary small constant. Fix any $z=E+\ii \eta$ with $|E|\le 2-\kappa$ and $\eta=N^{-1+c}$. Then, for $z_1,z_2\in \{z,\bar z\}$ and any constant $C>0$, we have that
\be\label{eq:nonmix_key2}
\E  \left\langle G_1 \Lambda  G_2 \Lambda \right\rangle \prec N^{-1}\|A \|^2_{\HS} + N^{-C}.
\ee
\end{lemma}

Again, the proof of this estimate relies on \eqref{eq:G-M} and the cumulant expansion formula, \Cref{lem:complex_cumu}. We first consider a simpler case where the entries of $H$ are complex Gaussian. In this case, \Cref{lem:complex_cumu} can be replaced by the Gaussian integration by parts.

\subsection{Proof of \Cref{lem:nonmix_key2}: Gaussian case} \label{sec:Gauss_exp}

We recall the definitions of the scalar $m$ and the matrix $M$ in \eqref{self_m} and \eqref{def_G0}. In the following proof, we denote $\widetilde m:=-(m+z)^{-1}$ and 
\begin{equation}\label{MLambda}
\wt M:=M-\wt m=-\wt m \Lambda M=- \sum_{k=1}^\infty (m+z)^{-(k+1)}\Lambda^k.
\end{equation}
When $z=z_i$, $i\in \{1,2\}$, we then use the notations \smash{$\wt m_i$ and $\wt M_i$}.

By the averaged local law \eqref{eq:aver_local}, we have that 
$$ \langle M_1\Lambda G_2\Lambda\rangle = \langle \Lambda^2 M_1 G_2\rangle \prec N^{-1}\|A\|_{\HS}^2 .$$
Thus, we only need to bound $\langle (G_1-M_1) \Lambda  G_2   \Lambda \rangle $. With \eqref{eq:G-M} and Gaussian integration by parts, we get that
\begin{align} \label{HL1}
\E \langle (G_1-M_1) \Lambda  G_2   \Lambda\rangle  & = - \E \langle  G_1(H+m_1) \widetilde M_1\Lambda  G_2 \Lambda \rangle
- \widetilde m_1\E \langle  G_1(H+m_1)\Lambda  G \Lambda \rangle = Y_{1}, 
\end{align}
where $Y_1$ is defined as
\begin{align}
Y_1& = D \E \sum_{a=1}^D  \langle G_1 E_a \widetilde M_1 \Lambda  G_2   \Lambda\rangle  \langle (G_1-M_1) E_a\rangle + D\E  \sum_{a=1}^D \langle  G_1 E_a G_2 \Lambda \rangle  \langle \widetilde M_1 \Lambda  G_2 E_a\rangle \nonumber \\ 
&+D\widetilde m_1 \E\sum_{a=1}^D \langle G_1 E_a  \Lambda  G_2   \Lambda\rangle \langle (G_1-M_1) E_a\rangle   +  D\widetilde m_1 \E\sum_{a=1}^D \langle  G_1 E_a G_2 \Lambda \rangle  \langle   \Lambda  G_2 E_a\rangle. \label{HL2}
\end{align}
We can keep expanding $Y_1$ using \eqref{eq:G-M} and Gaussian integration by parts. During the expansions, all our expressions can be written into certain forms with some loop structures, which we now describe in more detail. In the expansions, each term only contains a deterministic coefficient and the following four types of loops (i.e., terms of the form $\langle \cdot \rangle$):
\begin{enumerate}
\item[(i)] {\bf Light weights}: $\langle (G-M)E_{a_i}\rangle$ for some $a_i\in \qqD$. 

\item[(ii)] {\bf$G$-loops}: $\langle \cal G^{(k)}\rangle$ for some $ k\ge 2$, where $\cal G^{(k)}$ represents an expression of the form
\be\label{eq:Gk}
\cal G^{(k)} =  E_{a_0}\prod_{i=1}^k G_{s_i} E_{a_i}
\ee
for some sequences $a_i\in \qqD$ and $s_i\in \{1,2\}$.

\item[(iii)] {\bf$(G,\Lambda)$-loops}: $\langle \cal G^{(k)}\Lambda^{(n)}\rangle$ for some $k\ge 0$ and $n\ge 1$, where $\cal G^{(k)}$ is of the form \eqref{eq:Gk} (with the convention that $\cal G^{(0)}=1$) and $\Lambda^{(n)}$ represents an expression of the form
\be\label{eq:Lambdak}
\Lambda^{(n)}\sim  E_{a_0} \prod_{i=1}^n\Lambda E_{a_i}
\ee
for some sequence $a_i\in \qqD$. 

\item[(iv)] {\bf$(G,\Lambda,G,\Lambda)$-loops}: $\langle \cal G^{(k_1)}\Lambda^{(n_1)} \cal G^{(k_2)}\Lambda ^{(n_2)}\rangle$ for some $k_1,k_2,n_1,n_2\ge 1 $.

\end{enumerate} 
Let ${\cal W}^{(k)}$ denote a product of light weights of the form ${\cal W}^{(k)}=\prod_{i=1}^{k}\langle (G-M)E_{a_i}\rangle $,
and let $\Gamma^{ (m) }_n$ denote a product of $G$-loops of the form
$$
  \Gamma^{ (m) }_n =\prod_{i=1}^n\langle \cal G^{(k_i)} \rangle\quad \text{with}\quad k_i\ge 2, \ \sum_{i=1}^n k_i=m .
$$
Hereafter, we slightly abuse the notation and use $m$ to denote an integer instead of $m(z)$ defined in \eqref{self_m}.

We will show that by following a \emph{specific expansion strategy}, we can continue expanding $Y_1$ until each expression is bounded by $\OO_\prec (N^{-1}\|A \|^2_{\HS}+N^{-C})$ as required. Before introducing the formal expansion strategy, we will examine some low-order expansions as illustrative examples. In equation \eqref{HL2}, the first two terms have ``smaller sizes" compared to the third and fourth terms, owing to the additional \smash{$\wt M_1$} factor. Utilizing the local laws \eqref{eq:aver_local} and \eqref{entprodG} along with the estimates \eqref{GA1}--\eqref{eq:WGamma3} below, we can show that 
\begin{align}
&  \E    \langle G_1 E_a \widetilde M_1 \Lambda  G_2   \Lambda\rangle  \langle (G_1-M_1) E_a\rangle +\E  \langle  G_1 E_a G_2 \Lambda \rangle  \langle \widetilde M_1 \Lambda  G_2 E_a\rangle \prec  \frac{\|A\|_{\HS}}{\sqrt{N}\eta}   \frac{\|A\|_{\HS}^2}{N},\\
&   \E \langle G_1 E_a  \Lambda  G_2   \Lambda\rangle \langle (G_1-M_1) E_a\rangle   +   \E \langle  G_1 E_a G_2 \Lambda \rangle  \langle   \Lambda  G_2 E_a\rangle \prec \frac{\|A\|_{\HS}^2}{N\eta} .\label{largeterms}
\end{align}
Therefore, we will focus on expanding the two terms in \eqref{largeterms}. Using \eqref{eq:G-M}, we can expand them as
\begin{align}
& \E \langle G_1 E_a  \Lambda  M_2   \Lambda\rangle \langle (G_1-M_1) E_a\rangle  - \E \langle G_1 E_a  \Lambda G_2(H+m_2) M_2   \Lambda\rangle \langle (G_1-M_1) E_a\rangle \nonumber\\
+ ~& \E \langle  G_1 E_a M_2 \Lambda \rangle  \langle \Lambda  G_2 E_a\rangle - \E \langle  G_1 E_a G_2(H+m_2)M_2 \Lambda \rangle  \langle  \Lambda  G_2 E_a\rangle .\label{eq:example1}
\end{align}
With the estimate \eqref{eq:WGamma2} below, we can already bound the first and third terms by $\OO_\prec (N^{-1}\|A\|_{HS}^2)$.
For the second and fourth terms in \eqref{eq:example1}, by applying Gaussian integration by parts, we find: 
\begin{align}
&~D\E \sum_{x=1}^D  \langle G_2 E_x M_2 \Lambda G_1 E_a  \Lambda  \rangle \langle (G_1-M_1) E_a\rangle\langle (G_2-M_2) E_x\rangle \label{eq:example1.1}\\
+&~D\E \sum_{x=1}^D  \langle G_2 E_x G_1 E_a \Lambda \rangle \langle G_1 E_x M_2 \Lambda \rangle \langle (G_1-M_1) E_a\rangle \label{eq:example1.2}\\
+&~\frac{1}{DN^2}\E \sum_{x=1}^D  \langle G_2 E_x G_1E_a G_1 E_x M_2   \Lambda G_1 E_a  \Lambda  \rangle \label{eq:example1.3}\\
+&~ D \E  \sum_{x=1}^D \avg{G_1E_a G_2 E_x M_2 \Lambda } \avg{(G_2-M_2)E_x}\langle  \Lambda  G_2 E_a\rangle \label{eq:example1.4}\\
+&~ D \E  \sum_{x=1}^D \avg{ G_1 E_x M_2 \Lambda}\avg{G_1E_aG_2 E_x} \langle  \Lambda  G_2 E_a\rangle \label{eq:example1.5}\\
+&~\frac{1}{DN^2} \E \sum_{x=1}^D \avg{ G_1E_aG_2 E_x G_2 E_a \Lambda G_2E_x  M_2 \Lambda }.\label{eq:example1.6} 
\end{align}
Note the expressions \eqref{eq:example1.1} and \eqref{eq:example1.4} exhibit similar loop structures to the two terms in \eqref{largeterms}, albeit with an additional light weight. Consequently, their sizes are smaller by an extra factor of $(N\eta)^{-1}$: 
$$\eqref{eq:example1.1}+\eqref{eq:example1.4}\prec  {\|A\|_{\HS}^2}/{(N\eta)^2}. $$
Moving from the two terms in \eqref{largeterms} to \eqref{eq:example1.3} and \eqref{eq:example1.6}, the two loops combine into a single loop while gaining an $N^{-2}$ coefficient. Utilizing the estimate \eqref{eq:WGamma3} below, we can bound them by 
$$\eqref{eq:example1.3}+\eqref{eq:example1.6}\prec  {\|A\|_{\HS}^2}/{(N\eta)^2}, $$
which is also better than \eqref{largeterms} by a factor of $(N\eta)^{-1}$. Furthermore, from the first term in \eqref{largeterms} to the expression \eqref{eq:example1.2}, a $(G,\Lambda,G,\Lambda)$-loop is split into two $(G,\Lambda)$-loops. Using the estimate \eqref{eq:WGamma2} below, we can show that  \eqref{eq:example1.2} also satisfies a better bound than \eqref{largeterms}: 
$$ \eqref{eq:example1.2}\prec  {\|A\|_{\HS}^2}/{(N\eta)^2}.  $$

It remains to address the expression \eqref{eq:example1.5}. Utilizing the estimate \eqref{eq:WGamma2}, we can show that 
\be\label{eq:badexample}
(\ref{eq:example1.5}) \prec \|A\|_{\HS}^2/(N\eta).
\ee
While there is no improvement compared to the bound \eqref{largeterms}, we have a change in the loop structure. Specifically, two $(G,\Lambda)$-loops, denoted as $\langle \cal G^{(k_1)}\Lambda^{(n_1)}\rangle \langle \cal G^{(k_2)}\Lambda^{(n_2)}\rangle$ with $k_1=2$, $k_2=1$, and $n_1=n_2=1$, transform into a $\avg{\cal G^{(2)}}$ loop along with two $(G,\Lambda)$-loops with $k_1=k_2=1$ and $n_1=n_2=1$. The key feature of this change is that $k_1+k_2$ is decreased by 1, indicating that the two $(G,\Lambda)$-loops become ``more deterministic". To proceed, we further expand \eqref{eq:example1.5} (without the summation over $x$ and the $D$ factor) as 
\begin{align}\label{eq:remaint}
\E \avg{ M_1 E_x M_2 \Lambda}\langle  G_2 E_a  \Lambda\rangle \avg{G_1E_aG_2 E_x} - \E  \avg{ G_1(H+m_1)M_1 E_x M_2 \Lambda}\langle  G_2 E_a  \Lambda\rangle \avg{G_1E_aG_2 E_x} .
\end{align}
Note that the first term contains two $(G,\Lambda)$-loops $\langle \cal G^{(k_1)}\Lambda^{(n_1)}\rangle \langle \cal G^{(k_2)}\Lambda^{(n_2)}\rangle$ with $k_1=0$, $k_2=1$, and $n_1=n_2=1$, making it even ``more deterministic" than \eqref{eq:example1.5}. Moreover, recalling \eqref{MLambda}, the factor $\avg{ M_1 E_x M_2 \Lambda} $ can be bounded by 
$$\avg{ M_1 E_x M_2 \Lambda} = \wt m_1 \avg{ E_x \wt M_2 \Lambda} + \wt m_2 \avg{\wt M_1 E_x  \Lambda} + \avg{ \wt M_1 E_x \wt M_2 \Lambda}\lesssim N^{-1}\|A\|_{\HS}^2. $$
This estimate, together with \eqref{entprodG} and the estimate \eqref{eq:WGamma2} below, implies that 
$$ \E \avg{ M_1 E_x M_2 \Lambda}\langle  G_2 E_a  \Lambda\rangle \avg{G_1E_aG_2 E_x} \prec \frac{\|A\|_{\HS}}{\sqrt{N}\eta}\cdot \frac{\|A\|_{\HS}^2}{N},$$
an improvement over \eqref{eq:badexample} by a factor of $N^{-1/2}\|A\|_{\HS}$. Subsequently, by applying Gaussian integration by parts to the second term in \eqref{eq:remaint}, we obtain that
\begin{align*}
&~D\E \sum_{y=1}^D \avg{G_1E_y M_1 E_x M_2 \Lambda }\langle  G_2 E_a  \Lambda\rangle \avg{(G_1-M_1)E_y} \avg{G_1E_aG_2 E_x}\\
+&~\frac{1}{DN^2}\E \sum_{y=1}^D \avg{G_2E_y M_1 E_x M_2 \Lambda G_1 E_y G_2 E_a  \Lambda} \avg{G_1E_aG_2 E_x}\\
+&~\frac{1}{DN^2}\E \sum_{y=1}^D \avg{G_1E_y M_1 E_x M_2 \Lambda G_1 E_y G_1E_aG_2 E_x}
\langle  G_2 E_a  \Lambda\rangle  \\
+&~\frac{1}{DN^2}\E \sum_{y=1}^D \avg{G_1E_aG_2 E_y M_1 E_x M_2 \Lambda G_1 E_yG_2 E_x }\langle  G_2 E_a  \Lambda\rangle .
\end{align*}
By using the estimates \eqref{eq:WGamma2} and \eqref{eq:WGamma3} below, we can show that the first two expressions are bounded by $\OO_\prec ({\|A\|_{\HS}^2}/{(N\eta)^2})$, while the latter two expressions are bounded by $\OO_\prec ({\|A\|_{\HS}^2}/{(N\eta)^3})$. These bounds are better than \eqref{eq:badexample} by at least a factor of $(N\eta)^{-1}$.

Inspired by the above calculations, we define the following two types of expressions with $n_1,n_2\ge 1$:
\begin{itemize}
    \item {\bf Type I expressions}: ${\cal W}^{(k)}\Gamma^{ (m) }_n \langle \cal G^{(k_1)}\Lambda^{(n_1)}\rangle \langle \cal G^{(k_2)}\Lambda^{(n_2)}\rangle$ with a deterministic coefficient of order 
\be\label{coef1}
\OO\left(N^{-(m-n+k_1+k_2-3)-{\bf 1}(k_1=0)-{\bf 1}(k_2=0)}\right).
\ee

\item {\bf Type II expressions}: ${\cal W}^{(k)}\Gamma^{ (m) }_n\langle \cal G^{(k_1)} \Lambda^{(n_1)}\cal G^{(k_2)} \Lambda^{(n_2)}\rangle
$ with a deterministic coefficient of order 
\be\label{coef2}
\OO\left(N^{-(m-n+k_1+k_2-2)-{\bf 1}(k_1=0)-{\bf 1}(k_2=0)}\right).
\ee
\end{itemize}
In our expansions, $Y_1$ can be expanded (in the sense of equal expectation) into a linear combination of these two types of expressions. 
The coefficients take the forms \eqref{coef1} and \eqref{coef2} due to the following reasons. In \eqref{HL2}, every term has three $G$ factors and two $N^{-1}$ factors associated with the two loops, so the number of $G$ factors is greater than the number of $N^{-1}$ factors by 1. 
For each application of Gaussian integration by parts, we gain one more $G$ factor and one more $N^{-1}$ factor. Hence, we should have $ (m-n)+(k_1-1)+(k_2-1)-1$ many $N^{-1}$ factors for type I expressions and $ (m-n)+(k_1+k_2-1)-1$ many $N^{-1}$ factors for type II expressions if $k_1\ne 0$ and $k_2\ne 0$. The case of $k_1=0$ or $k_2=0$ happens when we have replaced (at least) one or two $G_i$, $i\in \{1,2\}$, factors with $M_i$, so we should gain one or two more $N^{-1}$ factors in those cases.

We are now prepared to describe our expansion strategy and a systematic approach to bounding all expressions that appear during the expansion process. Prior to that, let us provide a brief overview of our proof. We will see that for a type I expression (resp.~type II expression), if $k_1=k_2=0$ (resp.~$k_1=0$ or $k_2=0$), it can already be bounded by $\OO_\prec (N^{-1}\|A\|_{HS}^2)$. Otherwise, suppose $k_1>0$ without loss of generality. Then, we select the \emph{first $G$} to the right of $\Lambda^{(n_1)}$, denoted as $G_{s}$ with $s\in \{1,2\}$. We expand this $G_s$ with \eqref{eq:G-M}, i.e., $G_s=M_s-G_s(H+m_s)M_s$, and subsequently apply Gaussian integration by parts. Every resulting expression from this expansion satisfies one of the following properties:
\begin{itemize}
    \item It can be bounded by $\OO_\prec (N^{-1}\|A \|^2_{\HS})$.
    \item It satisfies a better bound compared to the original expression by at least a factor of $(N\eta)^{-1}$.
    \item It is ``more deterministic" than the original expression in the sense that $k_1+k_2$ becomes strictly smaller.  
\end{itemize}
This expansion process can be iterated until every resulting expression can be bounded by $\OO_\prec (N^{-1}\|A \|^2_{\HS}+N^{-C})$. The number of iterations is at most $\OO(1)$.

With the above idea in mind, we now begin with the formal proof.  
With the local laws in \Cref{lem_loc} and \Cref{appen1}, we immediately obtain that 
\be\label{eq:WGamma} {\cal W}^{(k)}\prec  (N\eta)^{-k} ,\quad   \Gamma^{ (m) }_n\prec \eta^{-(m-n)}.
\ee
Furthermore, let $G^{(k)}$ be an $N\times N$ block of the matrix $B_0\prod_{i=1}^k G_i B_i$, where $(B_i)_{0\leq i \leq k}$ is an arbitrary sequence of deterministic matrices satisfying $\|B_i\|\le 1$. Then, using \Cref{appen1} and the singular value decomposition $A=\sum_{i}\lambda_i^a \bu_i^a (\bv_i^a)^*$, we obtain that  
\begin{align}
   & \tr(G^{(k)} A) = \sum_i \lambda_i^a (\bv_i^a)^* G^{(k)} \bu_i^a \prec \eta^{-(k-1)}\sqrt{N}\|A\|_{\HS},\label{GA1}\\ 
   & \tr(G^{(k)} AA^*) = \sum_i (\lambda_i^a)^2 (\bu_i^a)^* G^{(k)} \bu_i^a \prec \eta^{-(k-1)} \|A\|_{\HS}^{2},\label{GA2}\\
   & \tr(G^{(k)} A^2) = \sum_{i,j} \lambda_i^a\lambda_j^a (\bv_j^a)^* G^{(k)} \bu_i^a (\bv_i^a)^*\bu_j^a\prec \eta^{-(k-1)} \sum_j \lambda_j^a \Big\|\sum_i \lambda_i^a \bu_i^a (\bv_i^a)^*\bu_j^a\Big\| \prec \eta^{-(k-1)} \|A\|_{\HS}^{2}.\label{GA3}
\end{align}
Similar bounds also hold if we switch $A$ and $A^*$. With these bounds, we can readily derive that 
\begin{equation}\label{eq:WGamma2}
 \langle {\cal G}^{(k)} \Lambda^{(n)}  \rangle
 \prec  \eta^{-(k-1)}\|A\|^{n-[1+{\bf 1}(n \ge 2)]}\cdot ( N^{-1/2}\|A\|_{\HS})^{1+{\bf 1}(n \ge 2)} 
\end{equation}
for any ${\cal G}^{(k)}$ and $\Lambda^{(n)}$ with $k,n\ge 1$. Following a similar argument based on \Cref{appen1}, we also obtain that 
\begin{equation}\label{eq:WGamma3}
 \langle {\cal G}^{(k_1)} \Lambda^{(n_1)} {\cal G}^{(k_2)} \Lambda^{(n_2)}  \rangle \prec \frac{\|A\|^{n_1+n_2-2-({\bf 1}(n_1\ge 2)+{\bf 1}(n_2\ge 2))}}{\eta^{k_1+k_2-2}}\cdot 
  (N^{-1/2}\|A\|_{\HS})^{{\bf 1}(n_1\ge 2)+{\bf 1}(n_2\ge 2)}\|A\|_{\HS}^2
\end{equation}
for any ${\cal G}^{(k_1)}, {\cal G}^{(k_2)}, \Lambda^{(n_1)} , \Lambda^{(n_2)}$ with $k_1,k_2,n_1,n_2\ge 1$. When $k=0$ and $n = 1$, we trivially have $\langle {\cal G}^{(0)}  \Lambda^{(n)}  \rangle=0$; when $k=0$ and $n \ge 2$, we have
\begin{equation}\label{eq:WGamma4}
 \langle {\cal G}^{(0)}  \Lambda^{(n)}  \rangle \prec \|A\|^{n-2}\cdot N^{-1}\|A\|_{\HS}^2 .
 \end{equation} 
When $k_1=0$ and $k_2,n_1,n_2\ge 1$, we have 
\begin{equation}\label{eq:WGamma5}
  \langle {\cal G}^{(0)} \Lambda^{(n_1)} {\cal G}^{(k_2)} \Lambda^{(n_2)}  \rangle
 \prec \eta^{-(k_2 -1)}\|A\|^{n_1+n_2-2} \cdot N^{-1 }\|A\|_{\HS}^2 .
\end{equation}

For simplicity of presentation, we use notations $\cal T_1(k,m,n;k_1, k_2, n_1, n_2)$ and $\cal T_2(k,m,n;k_1, k_2, n_1, n_2)$ to denote generic type I and type II expressions, together with their deterministic coefficients satisfying \eqref{coef1} and \eqref{coef2}. Their exact forms may change from one line to the next. Using the above bounds \eqref{eq:WGamma}--\eqref{eq:WGamma5}, we readily obtain some ``rough bounds" on type I and type II expressions: for $i\in \{1,2\}$ and $n_1,n_2\ge 1$, 
\begin{align}\label{eq:roughT12}
\cal T_i(k,m,n;k_1, k_2, n_1, n_2)
& \prec
 \left(\frac{1}{N\eta}\right)^{k+m-n+k_1+k_2-2}
\cdot  (N^{-1/2})^{{\bf 1}(n_1\ge 2)+{\bf 1}(n_2\ge 2)}\|A\|_{\HS}^{n_1+n_2} .
\end{align}
(Note we allow that $k_1=0$ or $k_2=0$ in this bound.) When $k_1=0$, we have better bounds: 
\begin{align}
 \cal T_1(k,m,n;0, k_2, n_1, n_2)
& \prec
 \left(\frac{1}{N\eta}\right)^{k+m-n+k_2-1}
 \cdot (N^{-1/2})^{1+ {\bf 1}(n_2 \ge 2)} \|A\|_{\HS}^{n_1+n_2} ,\quad k_2\ge 1,\label{eq:rough_k1=0_1}
\\
\cal T_2(k,m,n;0, k_2, n_1, n_2)
& \prec
 \left(\frac{1}{N\eta}\right)^{k+m-n+k_2-1}
 \cdot N^{-1 }\|A\|_{\HS}^{n_1+n_2} , \quad k_2\ge 1, \label{eq:rough_k1=0_2}\\
\cal T_i(k,m,n;0, 0, n_1, n_2)
& \prec
 \left(\frac{1}{N\eta}\right)^{k+m-n}
 \cdot N^{-1 }\|A\|_{\HS}^{n_1+n_2} ,\quad i\in \{1,2\}.\label{eq:rough_k1=0_3}
 \end{align}
Similar bounds also hold when $k_2=0$. The bounds \eqref{eq:roughT12}--\eqref{eq:rough_k1=0_3} show that the following expressions can be bounded by $\OO_\prec( N^{-1 }\|A\|_{\HS}^2)$ as desired:
\begin{enumerate}
    \item[(E1)] $\cal T_i(k,m,n;k_1, k_2, n_1, n_2)$, $i\in \{1,2\}$, with $ n_1,n_2\ge 2$; 
    \item[(E2)] $\cal T_i(k,m,n;k_1, k_2, n_1, n_2)$, $i\in \{1,2\}$, with $ k_1=k_2=0$; 
    \item[(E3)] $\cal T_2(k,m,n;k_1, k_2, n_1, n_2)$  with $k_1=0$ or $k_2=0$;
    \item[(E4)] $\cal T_1(k,m,n;k_1, k_2, n_1, n_2)$ with $k_1=0$, $n_2\ge 2$ or $k_2=0$, $n_1\ge 2$;
    \item[(E5)] $N^{-1}\cal T_i(k,m,n;k_1, k_2, n_1, n_2)$, $i\in \{1,2\}$, with $ n_1,n_2\ge 1$.
\end{enumerate}

Now, the estimate \eqref{eq:nonmix_key2} in the Gaussian case follows immediately from the next claim. 

\begin{claim}\label{claim:Gauss}
For any constant $C_0>0$, we can expand $Y_1$ in \eqref{HL2} as the expectation of a linear combination of $\OO(1)$ many expressions satisfying (E1)--(E5) and some error expressions of order $\OO_\prec(N^{-C_0})$. 
\end{claim}

\begin{proof}
In the following proof, corresponding to \eqref{eq:roughT12}--\eqref{eq:rough_k1=0_3}, we define the \emph{sizes} of the expressions on the LHS by their RHS. For example, if $k_1,k_2\ge 1$, we define  
$$\size\left[\cal T_i(k,m,n;k_1, k_2, n_1, n_2)\right]:=\left(\frac{1}{N\eta}\right)^{k+m-n+k_1+k_2-2}\cdot  (N^{-1/2})^{{\bf 1}(n_1\ge 2)+{\bf 1}(n_2\ge 2)}\|A\|_{\HS}^{n_1+n_2} .$$
If $k_1=0$ and $k_2\ge 1$ (or $k_2=0$ and $k_1\ge 1$), we define the sizes of $\cal T_i(k,m,n;k_1, k_2, n_1, n_2)$, $i\in \{1,2\}$, through \eqref{eq:rough_k1=0_1} and \eqref{eq:rough_k1=0_2}; if $k_1=k_2=0$, we define the sizes through \eqref{eq:rough_k1=0_3}. 
If we have an expression, say $a_N \cal T + b_N\cal T'$, for some deterministic coefficients $a_N, b_N$ and type I or type II expressions $\cal T, \cal T'$, then 
$$\size(a_N \cal T + b_N\cal T'):=a_N \size(\cal T) + b_N\size(\cal T'). $$  
We set the stopping rule of our expansions as follows: \emph{we stop expanding an expression $\E \cal T$ if it satisfies (E1)--(E5) or $\size(\cal T)\prec N^{-C_0}$}.

Given $Y_1$ in \eqref{HL2}, we first expand $\wt M_1$ using the Taylor expansion \eqref{MLambda}:
\be\label{eq:wtM1} \wt M_1= - \sum_{k=1}^{n_0} (m_1+z_1)^{-(k+1)}\Lambda^k + (m_1+z_1)^{-(n_0+1)}\Lambda^{n_0+1}  M_1,
\ee
where $n_0$ is chosen sufficiently large such that $\|A\|^{n_0}\le N^{-C_0-10}$. In every resulting expression, if it contains the error term $(m_1+z_1)^{-(n_0+1)}\Lambda^{n_0+1}  M_1$, then we can easily check that it is of order $\OO_\prec(N^{-C_0})$; otherwise, between each product of two matrices, we decompose the identity matrix as \smash{$I=\sum_{a=1}^D E_a$}, i.e., for every matrix product $ B_1B_2$, we write it as $\sum_a B_1 E_a B_2$. In this way, we can expand \eqref{HL2} into the expectation of a linear combination of $\OO(1)$ many expressions in $\cal T_1(k=0,m=0,n=0; k_1=2,k_2=1,n_1=1,n_2)$ and $\cal T_2(k=1, m=0, n=0;k_1=1,k_2=1, n_1=1, n_2)$ with $n_2\ge 1$ plus error expressions of size $\OO_\prec( N^{-C_0})$.

Now, given any type I or type II expression that is not an error of size $\OO_\prec( N^{-C_0})$ and does not satisfy (E1)--(E4), say $\cal T(k,m,n;k_1, k_2, n_1, n_2)$ with $n_1,n_2\ge 1$, we expand it according to the following strategy. 

\medskip

\noindent{\bf Step 1}: In this expression, we have either $n_1=1$ or $n_2=1$ since it does not satisfy (E1). Without loss of generality, suppose $n_1=1$. Then, we must have $k_1\ge 1$. Otherwise, if $k_1=0$ and $\cal T$ is a type I expression, then we trivially have $\langle \Lambda^{(1)}  \rangle=0$; if $k_1=0$ and $\cal T$ is a type II expression, then $\cal T$ satisfies (E3). Now, we pick the first $G$ to the right of this $\Lambda^{(1)}$, say $G_{s}$ with $s\in \{1,2\}$, i.e., there is a factor $G_{s}E_{a_0}\Lambda E_{a_1}$ in the loop containing $ \Lambda^{(1)}=E_{a_0}\Lambda E_{a_1} $, $a_0,a_1\in \qqD$. We then expand $G_s$ with \eqref{eq:G-M}, i.e., $G_s=M_s-G_s(H+m_s)M_s$. 

\medskip

\noindent{\bf Step 2}: In a new expression from step 1 that does not satisfy the stopping rule, suppose we have replaced this $G_s$ with $M_s$. By expanding \smash{$\wt M_s$} with \eqref{eq:wtM1} and adding {$I=\sum_{a} E_a$} between matrix products, we get some error expressions of size $\OO_\prec(N^{-C_0})$ plus expressions of the form $\cal T'=(N^{-1})^{\mathbf 1(k_1\ge 2)}\cdot \cal T_i (k,m,n;k_1-1, k_2, n_1', n_2)$, $i\in \{1,2\}$, with $n_1'\ge n_1 =1$. Then:
\begin{itemize}
    \item If $k_1\ge 2$, then $\cal T'$ satisfies (E5).
    \item If $k_1=1$ and $i=2$, then $\cal T'$ satisfies (E3).
    \item If $k_1=1$, $i=1$, and $n_2\ge 2$, then $\cal T'$ satisfies (E4).
    \item If $k_1=1$, $i=1$, $n_2=1$, and $k_2=0$, then $\cal T'$ satisfies (E2).
\end{itemize}
In the remaining cases with $i=1$, $k_1=1$, $k_2\ge 1$, $n_1'\ge n_1 =1$, and $n_2=1$, $\cal T'$ is a type I expression $\cal T_1(k,m,n;0,k_2,n_1',1)$ with 
$\size(\cal T')\le N^{-1/2}\size(\cal T).$ We send this $\cal T'$ back to Step 1.

\medskip

\noindent{\bf Step 3}: In a new expression from step 1 that does not satisfy the stopping rule, suppose we have replaced the $G_s$ with $-G_s(H+m_s)M_s$. Again, by expanding $\wt M_s$ with \eqref{eq:wtM1} and adding $I=\sum_{a} E_a$ between matrix products, we get errors of size $\OO_\prec(N^{-C_0})$ plus expressions of the form $\E \cal T_1$ or $\E \cal T_2$, where for some $\wt n_1\ge n_1$,
\be\nonumber
\cal T_1= c_N {\cal W}^{(k)}\Gamma^{ (m) }_n \langle \cal G^{(k_1-1)} G_s(H+m_s) \Lambda^{(\wt n_1)}\rangle \langle \cal G^{(k_2)}\Lambda^{(n_2)}\rangle
\ee
with a deterministic coefficient $c_N =\OO (N^{-(m-n+k_1+k_2-3)-{\bf 1}(k_1=0)-{\bf 1}(k_2=0)} )$
or 
\be\nonumber 
\cal T_2=c_N {\cal W}^{(k)}\Gamma^{ (m) }_n\langle \cal G^{(k_1-1)}G_s(H+m_s) \Lambda^{(\wt n_1)}\cal G^{(k_2)} \Lambda^{(n_2)}\rangle
\ee
with a deterministic coefficient $c_N=\OO (N^{-(m-n+k_1+k_2-2)-{\bf 1}(k_1=0)-{\bf 1}(k_2=0)} )$.

First, we apply the Gaussian integration by parts to $\E\cal T_1$:
\begin{align}
    \E\cal T_1 &= \frac{c_N}{DN^2} \E \sum_{x=1}^D\sum_{\al,\beta\in I_x} 
  \partial_{\beta\al}\left[\left(\Lambda^{(\wt n_1)} \cal G^{(k_1-1)} G_s\right)_{\beta\al} {\cal W}^{(k)}\Gamma^{ (m) }_n \langle \cal G^{(k_2)}\Lambda^{(n_2)}\rangle\right] \nonumber\\
  &+ c_N m_s \E  {\cal W}^{(k)}\Gamma^{ (m) }_n \langle \cal G^{(k_1-1)} G_s \Lambda^{(\wt n_1)}\rangle \langle \cal G^{(k_2)}\Lambda^{(n_2)}\rangle. \label{eq:calT1}
\end{align}
Taking the derivative $\partial_{\beta\al}$, we get the following cases. 
\begin{itemize}
    \item If $\partial_{\beta\al}$ acts on $G_s$, then the resulting expression together with the second term on the RHS of \eqref{eq:calT1} gives a sum of expressions in $\cal T_1(k+1,m,n;k_1, k_2, \wt n_1, n_2)$. These new expressions have sizes $\lesssim (N\eta)^{-1}\size(\cal T_1)$. 

    \item If $\partial_{\beta\al}$ acts on a light weight in ${\cal W}^{(k)}$, then we get a sum of expressions in $\cal T_1(k-1,m,n;k_1+2, k_2, \wt n_1, n_2)$. These new expressions have sizes $\lesssim (N\eta)^{-1}\size(\cal T_1)$. 

    \item If $\partial_{\beta\al}$ acts on a $G$-loop in $\Gamma^{ (m) }_n$, then we get a sum of expressions in $\cal T_1(k,m-k',n-1;k_1+k'+1, k_2, \wt n_1, n_2)$ for some $2\le k' \le m$. These new expressions have sizes $\lesssim (N\eta)^{-2}\size(\cal T_1)$. 

    \item If $\partial_{\beta\al}$ acts on $\cal G^{(k_1-1)}$, then we create a new $G$-loop and get a sum of expressions in $\cal T_1(k,m +k',n+1 ;k_1-k'+1, k_2, \wt n_1, n_2)$ for some $2\le k' \le k_1$. These new expressions have sizes $\lesssim \size(\cal T_1)$, but the value of $k_1$ is decreased at least by 1. 

    \item If $\partial_{\beta\al}$ acts on $\cal G^{(k_2)}$, then we create a $(G,\Lambda, G,\Lambda)$-loop and get a sum of expressions in $\cal T_2(k,m,n; k'_1, k_1+k_2-k'_1+1, \wt n_1, n_2)$ for some $1\le k'_1 \le k_2$. These new expressions have sizes $\lesssim (N\eta)^{-1}\size(\cal T_1)$. 
\end{itemize}

Second, we apply the Gaussian integration by parts to $\E\cal T_2$:
\begin{align}
    \E\cal T_2 & = \frac{c_N}{DN^2} \E \sum_{x=1}^D\sum_{\al,\beta\in I_x} 
  \partial_{\beta\al}\left[\left( \Lambda^{(\wt n_1)}\cal G^{(k_2)} \Lambda^{(n_2)} \cal G^{(k_1-1)} G_s\right)_{\beta\al} {\cal W}^{(k)}\Gamma^{ (m) }_n \right] \nonumber\\
  &+ c_N m_s \E  {\cal W}^{(k)}\Gamma^{ (m) }_n\langle \cal G^{(k_1-1)} G_s \Lambda^{(\wt n_1)}\cal G^{(k_2)} \Lambda^{(n_2)}\rangle. \label{eq:calT2}
\end{align}
Taking the derivative $\partial_{\beta\al}$, we get the following cases.  
\begin{itemize}
    \item If $\partial_{\beta\al}$ acts on $G_s$, then the resulting expression together with the second term on the RHS of \eqref{eq:calT2} gives a sum of expressions in $\cal T_2(k+1,m,n;k_1, k_2, \wt n_1, n_2)$. These new expressions have sizes $\lesssim (N\eta)^{-1}\size(\cal T_2)$. 

    \item If $\partial_{\beta\al}$ acts on a light weight in ${\cal W}^{(k)}$, then we get a sum of expressions in $\cal T_2(k-1,m,n;k_1+2, k_2, \wt n_1, n_2)$. These new expressions have sizes $\lesssim (N\eta)^{-1}\size(\cal T_2)$. 

    \item If $\partial_{\beta\al}$ acts on a $G$-loop in $\Gamma^{ (m) }_n$, then we get a sum of expressions in $\cal T_2(k,m-k',n-1;k_1+k'+1, k_2, \wt n_1, n_2)$ for some $2\le k' \le m$. These new expressions have sizes $\lesssim (N\eta)^{-2}\size(\cal T_2)$. 

    \item  If $\partial_{\beta\al}$ acts on $\cal G^{(k_1-1)}$, then we create a new $G$-loop and get a sum of expressions in $\cal T_2(k,m +k',n+1 ;k_1-k'+1, k_2, \wt n_1, n_2)$ for some $2\le k' \le k_1$. These new expressions have sizes $\lesssim\size(\cal T_2)$, but the value of $k_1$ is decreased at least by 1. 

    \item If $\partial_{\beta\al}$ acts on $\cal G^{(k_2)}$, then we create two $(G,\Lambda)$-loops and get a sum of expressions in $\cal T_1(k,m,n;k'_1, k_1+k_2-k'_1+1, \wt n_1, n_2)$ for some $1\le k_1'\le k_2$. These new expressions have sizes $\lesssim (N\eta)^{-1}\size(\cal T_2)$. 
\end{itemize}

For every new expression from Step 3 that does not satisfy the stopping rule, we send it back to Step 1. 

\medskip 

Now, we iterate the above expansion strategy, Steps 1--3, until all resulting expressions satisfy the stopping rule. It then suffices to show that this expansion process will finally stop after $\OO(1)$ many iterations. In fact, we see that after one iteration of Steps 1--3, for each new expression, either it has a strictly smaller size than the input graph by a factor $(N\eta)^{-1}\le N^{-c}$, or the value of $k_1+k_2$ decreases at least by 1. With this observation, we can show that all expressions will satisfy the stopping rule after at most $(\lceil C_0/c\rceil+10)^2$ many iterations. This concludes the proof of \Cref{lem:nonmix_key2} in the Gaussian case.
\end{proof}

\subsection{Proof of \Cref{lem:nonmix_key2}: non-Gaussian case} 

In the non-Gaussian case, we apply the cumulant expansion, \Cref{lem:complex_cumu}, to \eqref{HL1} and get some new terms containing higher-order cumulants of the $H$ entries plus an error term $\cal R_{l+1}$. Again, this error term is negligible if we take $l$ sufficiently large. We now need to estimate $Y_1$ in \eqref{HL2} and the following terms: 
\begin{align} \label{HL3}
Y_{n,m}  & := - \frac{1}{DN}\sum_{\al,\beta\in \cI}\frac{1}{ n! m!}\cal C_{\beta\al}^{(n,m+1)}\E \partial_{\al\beta}^m \partial_{\beta\al}^n ( M_1\Lambda  G_2 \Lambda G_1)_{\beta\al}  ,\quad 2\le m+n \le l.
\end{align}
Taking the derivatives, we can write \eqref{HL3} into a linear combination of $\OO(1)$ many terms of the following forms for some $s,t\in \N$ with $s+t=n+m$ and $(x(i),y(i)), (\wt x(j),\wt y(j))\in\{(\al,\beta),(\beta,\al)\}$, $i \in \qq{s}$, $j \in\qq{t}$:
\begin{align}\label{eq:p+q>2}
   &\frac{1}{DN}\sum_{\al,\beta\in \cI}\cal C_{\beta\al}^{(n,m+1)}  \E  (M_1\Lambda G_2)_{\beta x(1)} (G_2)_{y(1)x(2)}\cdots (G_2 \Lambda G_1)_{y(s) \wt x(1)}(G_1)_{\wt y(1) \wt x(2)} \cdots (G_1)_{\wt y(t)\al}  ,\ \ s\ge 1;  \\
   &\frac{1}{DN}\sum_{\al,\beta\in \cI}\cal C_{\beta\al}^{(n,m+1)}  \E  (M_1\Lambda G_2 \Lambda G_1)_{\beta \wt x(1)}(G_1)_{\wt y(1) \wt x(2)} \cdots (G_1)_{\wt y(t)\al}  ,\ \  s=0.  \label{eq:p+q>22}
\end{align}
By the anisotropic local law \eqref{eq:aniso_local}, we have that 
$$(M_1\Lambda G_2)_{\beta x(1)} \prec \|\mathbf e_\beta^* M_1\Lambda\|,\ \ (G_2 \Lambda G_1)_{y(s) \wt x(1)} \prec \sqrt{N}\|A\|_{\HS}, \ \ (M_1\Lambda G_2 \Lambda G_1)_{\beta \wt x(1)} \prec \|\mathbf e_\beta^* M_1\Lambda\|\cdot \sqrt{N}\|A\|_{\HS},$$ 
where $\mathbf e_\beta$ denotes the standard unit vector along the $\beta$-th direction. So, we can bound \eqref{eq:p+q>2} and \eqref{eq:p+q>22} by 
$$ |\eqref{eq:p+q>2}|+|\eqref{eq:p+q>22}| \prec \frac{N^{3/2}\|A\|_{\HS}}{N^{(n+m+1)/2+1}}\sum_{\beta} \|\mathbf e_\beta^* M_1\Lambda\| \prec N^{-(m+n-3)/2}\cdot N^{-1} \|A\|_{\HS}^2,$$
which is sufficiently small if $m+n\ge 3$. 

When $m+n=2$, the above bound has an extra $N^{1/2}$ factor, so we need to further expand the terms in $Y_{n,m}$. Take two terms in $Y_{1,1}$ as examples:
\begin{align*}
    - \frac{1}{DN}\sum_{x=1}^D\sum_{\al,\beta\in \cI_x}\cal C_{\beta\al}^{(1,2)}\E  \left[( M_1\Lambda  G_2 \Lambda G_1)_{\beta\beta}(G_1)_{\al\al}  +( M_1\Lambda  G_2)_{\beta\beta}(G_2\Lambda G_1)_{\al\al} \right](G_1)_{\beta\al} .
\end{align*}
In the above expression, we can expand either $G_2$ or $G_1$ in the factor $G_2\Lambda G_1$ with \eqref{eq:G-M}. If we replace $G_1$ with $M_1$, by the anisotropic local law \eqref{eq:aniso_local}, the resulting expression can be bounded by 
\begin{align*}
    \frac{1}{N^{5/2}}\sum_{x=1}^D\sum_{\al,\beta\in \cI_x} \left[\|\mathbf e_\beta^* M_1\Lambda \| \|\Lambda M_1\mathbf e_\beta\| + \|\mathbf e_\beta^* M_1\Lambda \|\|\Lambda M_1\mathbf e_\al\| \right] \lesssim N^{-3/2}\|A\|_{\HS}^2.
\end{align*}
(Note that if we have replaced $G_2$ with $M_2$ instead, the resulting expression can be bounded by $N^{-1}\|A\|_{\HS}^2$, which is still good enough.) Now, for the term with $G_1$ replaced by $-M_1(m_1+H)G_1$, applying cumulant expansions we get a sum of small enough errors, Gaussian integration by parts terms, and terms involving third or higher-order cumulants: for some $2\le m+n \le l$,
\begin{align*} 
Y_{1,1;n,m} & =  \frac{1}{DN}\sum_{\al,\beta\in \cI}\cal C_{\beta\al}^{(1,2)}\sum_{i,j\in \cI}\cal C_{ji}^{(n,m+1)} \E \partial_{ij}^m \partial_{ji}^n  \left[( M_1\Lambda  G_2 \Lambda M_1)_{\beta i} (G_1)_{j\beta}(G_1)_{\al\al} (G_1)_{\beta\al}  \right]\\
&+  \frac{1}{DN}\sum_{\al,\beta\in \cI}\cal C_{\beta\al}^{(1,2)}\sum_{i,j\in \cI}\cal C_{ji}^{(n,m+1)} \E \partial_{ij}^m \partial_{ji}^n  \left[( M_1\Lambda  G_2)_{\beta\beta}(G_2\Lambda M_1)_{\al i}(G_1)_{j\al} (G_1)_{\beta\al} \right].
\end{align*}
With the anisotropic local law \eqref{eq:aniso_local}, we can bound this term by 
$$ Y_{1,1;n,m} \prec \frac{1}{N^{3+\frac{n+m}{2}}}\sum_{\al,\beta,i,j} \|\mathbf e_\beta^*  M_1\Lambda\| \left( \|\Lambda M_1\mathbf e_i\|+\|\Lambda M_1\mathbf e_j\|\right) \lesssim N^{-\frac{n+m}{2}}\|A\|_{\HS}^2 \le N^{-1}\|A\|_{\HS}^2.$$
Hence, we only need to further expand the Gaussian integration by parts terms. 

We can estimate all the terms $Y_{n,m}$ with $n+m=2$ using a similar argument as above. In this way, we obtain that  
$$\E \langle (G_1-M_1) \Lambda  G_2   \Lambda\rangle = Y_1 + \sum_{m+n=2}Y_{n,m; 1} + \OO_\prec\big(N^{-1}\|A\|_{\HS}^2\big),$$
where $Y_1$ is given in \eqref{HL2} and $Y_{n,m; 1}$ denotes the Gaussian integration by parts terms obtained from the expansions of $Y_{n,m}$. To bound the RHS, we still need to further expand $Y_1$ and $Y_{n,m; 1}$. 
Since the proof closely resembles that of \Cref{claim:Gauss}, we will provide only a concise overview of the proof by outlining the main expansion strategy, while omitting some of the finer details.

In the expansions of $Y_{n,m;1}$, if an expression involves third or higher-order cumulants of the $H$ entries, then it can be bounded by $\OO_\prec (N^{-1}\|A\|_{\HS}^2)$. Hence, we only need to consider expressions obtained from repeated applications of \eqref{eq:G-M} and Gaussian integration by parts. In this process, we either get a sufficiently small error of size $\OO_\prec(N^{-C_0})$ or expressions consisting of light weights, $G$-loops, at most one $(G,\Lambda)$-loop $\langle \cal G^{(k)}\Lambda^{(n)}\rangle$ for some $k\ge 0$ and $n\ge 1$, and terms of the form 
\be \nonumber
(\Lambda^{(n_1)} \cal G^{(k_1)}\Lambda^{(n_1')})_{\beta \fa} (\Lambda^{(n_2')}\cal G^{(k_2)}\Lambda^{(n_2)} \cal G^{(k_2')}\Lambda^{(n_2'')})_{\fb\fc},\quad \text{or}\quad (\Lambda^{(n_1)} \cal G^{(k_1)}\Lambda^{(n_2)} \cal G^{(k_2)}\Lambda^{(n_2')})_{\beta \fa},
\ee
for some $n_1,n_2\ge 1$, $k_1,k_2,n_1',n_2',n_2''\ge 0$, and $\fa,\fb,\fc\in \{\al,\beta\}$. One can check that if (i) there is a $(G,\Lambda)$-loop, (ii) at least one of $n_1',n_2',n_2''$ is non-zero, (iii) at least one of $n_1,n_2$ is at least 2, or (iv) $k_2=0$ or $k_2'=0$, then the expressions is bounded by $\OO_\prec (N^{-1}\|A\|_{\HS}^2)$. Hence, we only need to handle expressions containing the following factors:
\be\label{eq:newterms} 
(\Lambda \cal G^{(k_1)})_{\beta \fa} (\cal G^{(k_2)}\Lambda\cal G^{(k_2')})_{\fb\fc},\quad \text{or}\quad (\Lambda \cal G^{(k_1)}\Lambda \cal G^{(k_2)})_{\beta \fa}, \quad k_1\ge 0,\ k_2,k_2'\ge 1.
\ee
Then, similar to the expansion strategy (Steps 1--3) in the proof of \Cref{claim:Gauss}, we choose the $G$ in $\cal G^{(k_2)}$ right next to $\Lambda$, use \eqref{eq:G-M} to expand it, and then apply Gaussian integration by parts. (We will also need to use \eqref{eq:wtM1} and add $I=\sum_a E_a$ between matrix products in this process.) For each new expression, at least one of the following cases holds: (1) it is a small error $\OO_\prec (N^{-C_0})$; (2) it can be bounded by $\OO_\prec (N^{-1}\|A\|_{\HS}^2)$; (3) its size is smaller than the original expression by $\OO((N\eta)^{-1})$; (4) the value of $k_2+k_2'$ is decreased at least by 1 (we let $k_2'=0$ for $(\Lambda \cal G^{(k_1)}\Lambda \cal G^{(k_2)})_{\beta \fa}$ as a convention). We can show that after at most $(\lceil C_0/c\rceil+10)^2$ many iterations of this expansion strategy, all the resulting expressions is bounded by $\OO_\prec (N^{-1}\|A\|_{\HS}^2+N^{-C_0})$.

Now, we look at the expansions of $Y_{1}$. If an expression involves (1) a fourth or higher-order cumulant or (2) at least two third-order cumulants of the $H$ entries, then it can be bounded by $\OO_\prec (N^{-1}\|A\|_{\HS}^2)$. If we only apply Gaussian integration by parts in the whole expansion process, then the relevant expressions have been handled in \Cref{sec:Gauss_exp}. Finally, we need to address the scenario in which we employ the third-order cumulant expansion only once, while utilizing Gaussian integration for the remaining expansions.
Before the third-order cumulant expansion, we follow the expansion strategy in the proof \Cref{claim:Gauss}, while after that, we follow the strategy below \eqref{eq:newterms}. Here we omit the relevant details. This concludes the proof \Cref{lem:nonmix_key2}, which further concludes \Cref{NonMixEV}.

\appendix

\section{Some deterministic estimates}\label{appd_determ}

In this section, we provide some deterministic estimates about $M$ and $\wh M$ that have been used in the main proofs. 

\begin{lemma}[Basic properties of $M$ and $\wh M$]\label{lem_bas_M}
Let $A$ be an arbitrary deterministic matrix with $\|A\|=\oo(1)$. For any constant $\tau>0$, the following estimates hold uniformly for all $z=E+\ii\eta $ with $|z|\le \tau^{-1}$ and $\eta > 0$. 

\begin{itemize}

\item We have that 
\be\label{eq:msc}
|m(z) - m_{sc}(z)| \lesssim \|A\|^{1/2}  ,\quad \| M(z)-m(z)I\| \lesssim \|A\|^{1/2}.
\ee
Furthermore, when $|E|\le 2-\kappa$ for a constant $\kappa>0$, we have the better bounds
\be\label{eq:msc0.5}
|m(z) - m_{sc}(z)| \lesssim \|A\|  ,\quad \| M(z)-m(z)I\| \lesssim \|A\|.
\ee
As a consequence of \eqref{eq:msc0.5}, we have 
\be\label{eq:msc2}
| (m(z)+z)^{-1}| \lesssim 1,\quad \text{and}\quad  | (1-m(z)^2)^{-1}| \lesssim 1 \ \ \text{if} \ \ |E|\le 2-\kappa. 
\ee

\item $\wh M$ is translation invariant, i.e., 
$\wh M_{ab}=\wh M_{a'b'}$ whenever $a-b = a'-b' \mod D$.

\item For $z_1 = \bar z_2\in \{z,\bar z\}$ and $a\in \qqD$, we have that 
\be\label{sumwtM}
   \sum_{b=1}^D  \wh M(z_1,z_2)_{ab}=\frac{\im m(z) }{\im m(z)+\eta}=:d_1, 
\ee
where $d_1$ denotes the Perron–Frobenius eigenvalue of $\wh M(z_1,z_2)$ with $(1,\ldots, 1)^\top$ being the corresponding eigenvector. The other eigenvalues of $\wh M(z_1,z_2)$ satisfy
\be\label{eq:otherM}
d_k = d_1 - a_k - \ii b_k, \quad k=2,3,\ldots,D,
\ee
where $a_k,b_k\in \R$ satisfy that  
\be\label{eq:akbk}
a_k\ge 0,\quad a_k \sim N^{-1}\|A\|_{\HS}^2,\quad  |b_k|=\oo( N^{-1}\|A\|_{\HS}^2). 
\ee

\item  For $z_1 = \bar z_2\in \{z,\bar z\}$, we have that
\begin{align}
\big\|[1-\wh M(z_1,z_2)]^{-1}\big\|& = \frac{\im m(z)+\eta} {\eta }\lesssim \eta^{-1}.\label{1-M}
\end{align}

\item When $|E|\le 2-\kappa$ for a constant $\kappa>0$, there exists a constant $C_\kappa>0$ such that
\be\label{1-M-2}
\big\|[1-\wh M(z,z)]^{-1}\big\| \le C_\kappa.
\ee
Moreover, for $z_1, z_2\in \{z,\bar z\}$, we have that
\be\label{flat_M}
\max_{a,b,a',b'\in \qqD}\left|\left[(1-\wh M_{(1,2)})^{-1}\wh M_{(1,2)} \right]_{ab}-\left[(1-\wh M_{(1,2)})^{-1}\wh M_{(1,2)} \right]_{a'b'}\right|\lesssim \frac{N }{\|A\|_{\HS}^2}.
\ee
\end{itemize}
\end{lemma}

\begin{proof}
The first estimates in \eqref{eq:msc} and \eqref{eq:msc0.5} follow by comparing the self-consistent equation \eqref{self_m} with the equation $m_{sc}(z)=-[z+m_{sc}(z)]^{-1}$ and by utilizing the stability of the equation. The second estimates in \eqref{eq:msc} and \eqref{eq:msc0.5} then follow from the Taylor expansion 
\be\label{eq:MTaylor}
M=-\sum_{k=0}^{\infty}(m+z)^{-k-1}\Lambda^{k}.
\ee
The estimate \eqref{eq:msc2} holds due to the simple fact that $m_{sc}(z)$ satisfies the corresponding bounds. 

The translation invariance of \smash{$\wh M$} follows from the block translation invariance of $M$. As a consequence, its eigenvectors are given by $\bu_k$, $ k \in \qqD$, with 
\be\label{eq:cylicv} u_k(a)=D^{-1/2}\exp(\ii \cdot 2 \pi (k-1)(a-1)/D). 
\ee
The corresponding eigenvalues are 
\be\label{eq_lambdak}
d_k =\sum_{a=1}^D\wh M_{1a}(z_1,z_2)e^{i 2\pi (k-1) (a-1)/D}.
\ee
In particular, when $z_1 = \bar z_2\in \{z,\bar z\}$, $d_1$ is the Perron–Frobenius eigenvalue of $\wh M(z_1,z_2)$ and the identity \eqref{sumwtM} follows by applying Ward's identity \eqref{eq_Ward} below to $M(z)$ and utilizing the identity 
$$
\sum_{b=1}^D\wh M(z_1,z_2)_{ab}= \frac{1}{D}\sum_{a,b=1}^D\wh M(z_1,z_2)_{ab} =\frac{1}{DN}\sum_{i,j\in \cal I}|M_{ij}(z)|^2.
$$
With the Taylor expansion \eqref{eq:MTaylor}, we can derive that when $z_1=\bar z_2\in \{z,\bar z\}$, 
\be\label{wtMlead}
\begin{aligned}
\wh M(z_1,z_2)_{11}& =\frac{1}{|m+z|^{2}}
+\frac{2(1+{\bf 1}_{D>2})}{|m+z|^{2}}\re \left[ (m+z)^{-2}\right]\cdot N^{-1}\|A\|_{\HS}^2 +\oo(N^{-1}\|A\|^2_{\HS}),
\\
\wh M(z_1,z_2)_{12}&=|m+z|^{-4}N^{-1}\|A\|_{\HS}^2+\oo(N^{-1}\|A\|^2_{\HS}),
\\
\wh M(z_1,z_2)_{1D}& =|m+z|^{-4}N^{-1}\|A\|_{\HS}^2+\oo(N^{-1}\|A\|^2_{\HS}),
\\
\wh M(z_1,z_2)_{1k}&  = \oo(N^{-1}\|A\|^2_{\HS}),\quad 3\le k\le D-1.
\end{aligned}
\ee
Plugging these estimates into \eqref{eq_lambdak}, we immediately obtain \eqref{eq:otherM} and \eqref{eq:akbk}. Finally, using the eigendecomposition of $\wh M(z_1,z_2)$, we can derive \eqref{1-M} from with \eqref{eq:otherM}. 

When $|E|\le 2-\kappa$, we have 
\be\label{est_1-m2}
|1-m_{sc}(z)^2|\ge c_\kappa
\ee 
for a constant $c_\kappa>0$. Combining this estimate with \eqref{eq:msc}, we can obtain \eqref{1-M-2}. 
Now, with the eigendecomposition of $\wh M(z_1,z_2)$, we get   
$$
\left[(1-\wh M_{(1,2)})^{-1}\wh M_{(1,2)} \right]_{ab}
=\frac1D\sum_{k=2}^D
\frac{d_k}{1-d_k} e^{\ii\cdot 2\pi (k-1) (a-b) /D}+\frac{1}{D}\frac{d_1}{1-d_1},
$$
which gives that 
\be\label{eq_bddM(1-M)}
\left|\left[(1-\wh M_{(1,2)})^{-1}\wh M_{(1,2)} \right]_{ab}
-\frac{1}{D}\frac{d_1}{1-d_1}\right|\lesssim \max_{k=2}^{D}
\left|(1-d_k)^{-1}\right|.
\ee
If $z_1=z_2$, then combining \eqref{est_1-m2} with \eqref{eq:msc}, we obtain that $|1-d_k|\gtrsim 1$, which gives \eqref{flat_M}.  
If $z_1=\bar z_2$, applying \eqref{eq:otherM} and \eqref{eq:akbk}, we obtain that $ |1-d_k|\gtrsim N^{-1}\|A\|_{\HS}^2$ for $ k \in \llbracket 2, D\rrbracket$. Together with \eqref{eq_bddM(1-M)}, it concludes \eqref{flat_M}. 
\end{proof}

\section{Green's function comparison}\label{sec:comparison}

In this section, we complete the proof of \Cref{main_lemma_com} using a standard Green's function comparison argument. We adopt the continuous comparison method introduced in \cite{knowles2017anisotropic}, which is based on the following interpolation.

\begin{definition}[Interpolating matrices]
Introduce the notations $H^0:=\wt H $ and $H^1:=H$. Let $\rho_{ij}^0$ and $\rho_{ij}^1$ be the laws of \smash{$\wt H_{ij} $} and $H_{ij}$, respectively. For $\theta\in [0,1]$, we define the interpolated laws
\smash{$ \rho_{ij}^\theta := (1-\theta)\rho_{ij}^0+\theta\rho_{ij}^1.$}
(Note that $\rho_{ij}^0=\rho_{ij}^1=\delta_0$ if $i$ and $j$ are not in the same $\cI_a$.) Let \smash{$\{H^\theta: \theta\in (0,1) \}$} be a collection of random matrices such that the followings hold. For any fixed $\theta\in (0,1)$, $(H^0,H^\theta, H^1)$ is a triple of independent $DN\times DN$ random matrices, and the matrix $H^\theta=(H_{ij}^\theta)$ has law
\begin{equation}\label{law_interpol}
\prod_{i\le j\in \mathcal I } \rho_{ij}^\theta(\dd H_{ij}^\theta).
\end{equation}
(Note that we do not require $H^{\theta_1}$ to be independent of $H^{\theta_2}$ for $\theta_1\ne \theta_2 \in (0,1)$.) 
For $\lambda \in \mathbb C$, $i,j\in \mathcal I$, we define the matrix $H_{(ij)}^{\theta,\lambda}$ as
\be\label{Ximulambda} 
\left(H_{(ij)}^{\theta,\lambda}\right)_{kl}:=\begin{cases}H_{ij}^{\theta}, &\text{ if }\ \{k,l\}\ne \{i,j\},\\ \lambda, &\text{ if }\ (k,l)= (i,j),\\
\bar \lambda, &\text{ if }\ (k,l) = (j,i).
\end{cases} 
\ee
Correspondingly, we define the resolvents 
\[G^{\theta}(z):=G\left(z,H^{\theta},\Lambda\right),\ \ \ G^{\theta, \lambda}_{(ij)}(z):= G \big(z,H_{(ij)}^{\theta,\lambda},\Lambda\big).\]
\end{definition}

With (\ref{law_interpol}), taking the derivative with respect to $\theta$, we get the following interpolation formula: for any differentiable function $F:\mathbb R^{DN \times DN} \rightarrow \mathbb C$, 
\begin{equation}\label{basic_interp}
\begin{split}
\frac{\dd}{\dd\theta}\mathbb E F(H^\theta)&=\sum_{i\le j\in\mathcal I}\left[\mathbb E F\left(H^{\theta,H_{ij}^1}_{(ij)}\right)-\mathbb E F\left(H^{\theta,H_{ij}^0}_{(ij)}\right)\right] ,
\end{split}
\end{equation}
 provided all the expectations exist.
Then, \Cref{main_lemma_com} follows from the next estimate on the RHS of (\ref{basic_interp}) with $F(H^\theta)=\langle G_1^\theta E_{a}G_2^\theta E_{b}\rangle$. 

\begin{lemma}\label{lemm_comp_4}
Under the assumptions of \Cref{main_lemma_com}, we have
 \begin{equation}\label{compxxx}
  \sum_{i\le j\in\mathcal I} \left[\mathbb EF\left(H^{\theta,H_{ij}^1}_{(ij)}\right)-\mathbb EF\left(H^{\theta,H_{ij}^0}_{(ij)}\right)\right] \prec N^{-1-\delta}\eta^{-2}
 \end{equation}
 for all $\theta\in[0,1]$, where $F(H^\theta):=\langle G_1^\theta E_{a}G_2^\theta E_{b}\rangle$. 
\end{lemma}

\begin{proof}
The proof of (\ref{compxxx}) uses the moment matching conditions \eqref{comH0} and \eqref{comH} and the following resolvent expansion: for any $\lambda,\lambda'\in \mathbb C$ and $K\in \mathbb N$,
\begin{equation}\label{eq_comp_expansion2}
\begin{split}
G_{(ij)}^{\theta,\lambda'} = G_{(ij)}^{\theta,\lambda}&+\sum_{k=1}^{K}  G_{(ij)}^{\theta,\lambda}\left\{ \left[\re(\lambda-\lambda') \Delta_{ij} + \ii\im( \lambda- \lambda') \wt\Delta_{ij}\right]G_{(ij)}^{\theta,\lambda}\right\}^k  \\
&+ G_{(ij)}^{\theta,\lambda'}\left\{\left[\re(\lambda-\lambda') \Delta_{ij} + \ii\im( \lambda- \lambda') \wt\Delta_{ij}\right]G_{(ij)}^{\theta,\lambda}\right\}^{K+1},
\end{split}
\end{equation}
where the matrices $\Delta_{ij}$ and $\wt\Delta_{ji}$ are $DN\times DN$ matrices defined as 
$\left(\Delta_{ij} \right)_{kl}:= \delta_{ki}\delta_{lj}+ \delta_{kj}\delta_{li}$ and $(\wt \Delta_{ij})_{kl}:= \delta_{ki}\delta_{lj}- \delta_{kj}\delta_{li}.$ 
Since $H^\theta$ and $G^{\theta,H_{ij}^a}$, $a\in \{0,1\}$, also satisfy the setting of \Cref{lem_loc}, we have that for each $\theta\in [0,1]$, $G^\theta$ and $G^{\theta,H_{ij}^a}$ satisfy the local laws \eqref{eq:aniso_local} and \eqref{entprodG}. 
Furthermore, let $\xi$ be an arbitrary random variable satisfying $|\xi|\prec N^{-1/2}$. With the expansion \eqref{eq_comp_expansion2}, we can readily show that the local laws \eqref{eq:aniso_local} and \eqref{entprodG} hold also for the resolvent $G_{(ij)}^{\theta,\xi}$. We will use these estimates tacitly in the following proof.

By changing from $H_{ij}^0 := X_{ij}^0 + \ii Y_{ij}^0$ to $H_{ij}^1 := X_{ij}^1 + \ii Y_{ij}^1$, we can break down the replacement into two distinct steps. We first modify the real part $X_{ij}^0$ to $X_{ij}^1$, followed by replacing the imaginary part $Y_{ij}^0$ with $Y_{ij}^1$.
Hence, without loss of generality, we assume for the rest of the proof that $H_{ij}^0$ and $H_{ij}^1$ have the same imaginary part, i.e., $Y_{ij}^0=Y_{ij}^1$, while the case of $X_{ij}^0=X_{ij}^1$ can be handled in the same way. Moreover, with a slight abuse of notation, we denote 
 $$H^{\theta,X_{ij}^\gamma}_{(ij)}:= H^{\theta,H_{ij}^\gamma}_{(ij)},\quad G_{(ij)}^{\theta, X_{ij}^\gamma}:= G_{(ij)}^{\theta, H_{ij}^\gamma},\quad \gamma\in \{0,1\}.$$  
Using \eqref{eq_comp_expansion2} with $\lambda=0$ and $K=7$ and applying the local law \eqref{eq:aniso_local}, we get that 
\begin{align}\label{eq:G01exp}
    \left[(G_{s})_{(ij)}^{\theta, X_{ij}^\gamma}\right]_{xy}=\left[(G_{s})_{(ij)}^{\theta,0}\right]_{xy} + \sum_{k=1}^7 \cal X^{(ij)}_{xy}(s,\gamma,k) + \OO_\prec (N^{-4}) ,    \quad  x,y\in \cI,\ s\in\{1,2\},\ \gamma\in \{0,1\},
\end{align}
where the random matrices $\cal X^{(ij)}(s,\gamma,k)$ are defined by 
\be\label{eq:Xk} 
 \cal X^{(ij)}(s,\gamma,k):=(-X_{ij}^\gamma)^k(G_{s})_{(ij)}^{\theta,0}\left[\Delta_{ij}(G_{s})_{(ij)}^{\theta,0}\right]^k,\quad \text{with}\quad \max_{x,y\in \cI}|\cal X_{xy}^{(ij)}(s,\gamma,k)| \prec N^{-k/2}.
\ee
Note that $(G_{s})_{(ij)}^{\theta,0}$ is independent of $X_{ij}^\gamma$, $\gamma\in\{0,1\}$. Then, with \eqref{eq:G01exp} and \eqref{eq:Xk}, using the moment matching condition \eqref{comH0}, we obtain that
\begin{align}
    & \sum_{i\le j\in \cI} \left[\mathbb \E F\left(H^{\theta,X_{ij}^1}_{(ij)}\right)-\mathbb \E F\left(H^{\theta,X_{ij}^0}_{(ij)}\right)\right]\nonumber \\  
     &= \sum_{i\le j\in \cI} \sum_{1\le k,l\le 7,k+l\ge 4}\E  \left[\langle \cal X^{(ij)}(1,1,k)E_a \cal X^{(ij)}(2,1,l)E_b\rangle  -  \langle \cal X^{(ij)}(1,0,k) E_a \cal X^{(ij)}(2,0,l)E_b\rangle  \right]\nonumber\\
     & +\OO_\prec(N^{-3/2}\eta^{-1/2}),\label{eq:Xk2}
\end{align}
where we have applied the anisotropic local law for $(G_{s})_{(ij)}^{\theta,0}$ to control the terms containing the error $\OO_\prec (N^{-4}) $ in \eqref{eq:G01exp}. 
When $k+l\ge 4$, using the moment matching condition \eqref{comH}, we get that  
\begin{align}
	&\sum_{i\le j \in \cI}\E \left[\langle \cal X^{(ij)}(1,1,k)E_a \cal X^{(ij)}(2,1,l)E_b\rangle  -  \langle \cal X^{(ij)}(1,0,k)E_a \cal X^{(ij)}(2,0,l)E_b\rangle\right] \nonumber \\
	&\prec \left(N^{-2-\delta}\wedge N^{-(k+l)/2}\right)\sum_{i\le j \in \cI} \E \left|\avgB{(G_{1})_{(ij)}^{\theta,0}\left[\Delta_{ij}(G_{1})_{(ij)}^{\theta,0}\right]^k E_a \left[(G_{2})_{(ij)}^{\theta,0}\Delta_{ij}\right]^l (G_{2})_{(ij)}^{\theta,0} E_b}\right|\nonumber\\
	&\prec  N^{-3-\delta}\sum_{i\le j \in \cI} \E \left|\tr\left[ (G_{2})_{(ij)}^{\theta,0} E_b(G_{1})_{(ij)}^{\theta,0}\left[\Delta_{ij}(G_{1})_{(ij)}^{\theta,0}\right]^k E_a \left[(G_{2})_{(ij)}^{\theta,0}\Delta_{ij}\right]^l\right]\right| \nonumber\\
	&\prec  N^{-3-\delta}\cdot \left(N^2\eta^{-2}\right)=N^{-1-\delta}\eta^{-2},
	\label{eq:bddsumij}
\end{align}
where in the third step we have applied the estimates \eqref{eq:aniso_local} and \eqref{entprodG} to  $(G_{s})_{(ij)}^{\theta,0}$. Plugging \eqref{eq:bddsumij} into \eqref{eq:Xk2} concludes the proof. 
\end{proof}

\section{Proof of local laws}\label{sec:Gprop}

\subsection{Proof of \Cref{lem_loc}}\label{sec:pf_locallaw}

In this subsection, we briefly outline the proof of \Cref{lem_loc}, which is divided into two parts: a probabilistic part that focuses on deriving the self-consistent equation, and a deterministic part concerning the stability of this equation. 

The probabilistic part of the proof has been addressed by the arguments presented in \cite{He2018}. Following the notation therein, we define the functions $\Pi\equiv \Pi(\cdot, z): \C^{DN\times DN}\to \C^{DN\times DN}$ and $R\equiv R(\cdot ,z): \C\to \C$ as follows:
$$
\Pi(X):=I+z X+\mathcal{S}(X) X-\Lambda  X ,\quad R_{\xi}\equiv R(\xi):=\left(\Lambda-\xi-z\right)^{-1}. 
$$
Then, we can write $M(z)$ and $G(z)$ as  
$$
M=R_{m}, \quad G=R_{\cal S(G)}-R_{\cal S(G)} \Pi(G).
$$
Subtracting the above two equations yields
$$
G-M=R_{\cal S(G)}-R_m-R_{\cal S(G)} \Pi(G) .
$$
Multiplying both sides with $E_k$, $k\in \qqD$, and taking the averaged trace, we obtain that 
\be\label{eq:gks} g_k - m - D\avg{(R_{\cal S(G)}-R_m)E_k}=-D\avg{R_{\cal S(G)} \Pi(G)E_k},\quad \text{with}\quad g_k:=D\avg{GE_{k}}.
\ee
Using the arguments in \cite{He2018}, one can show that $\Pi(G)$ is indeed a small error, and $\langle \Pi(G)B\rangle$ satisfies a better bound for any deterministic matrix $B$ with $\|B\|\le 1$. 
Furthermore, $\cal S(G)$ can be written as $\cal S(G)=\sum_{k=1}^D g_k E_{k}$. Hence, \eqref{eq:gks} leads to a self-consistent vector equation for $\mathbf g=(g_1,\ldots, g_D) \in \C^D$:
\be\label{eq:gks2} g_k - m - D\avg{(R_{\bg \cdot \mathbf E}-R_m)E_k}=\cal E_k, \quad k \in \qqD,
\ee
where $\mathbf E:=(E_1,\ldots, E_D)$, $\bg \cdot \mathbf E:=\sum_{k=1}^D g_k E_{k}$, and $\cal E_k$ denotes an error term. When $\cal E_k=0$, this equation admits a solution $\mathbf m:=(m(z),\ldots, m(z))\in \C^D$. By combining the stability of the self-consistent equation \eqref{eq:gks2} as presented in \Cref{lem_stab} below with the error estimates established for $\Pi(G)$ and $\langle \Pi(G)B \rangle$, we can deduce the local laws \eqref{eq:aniso_local} and \eqref{eq:aver_local} employing the arguments in Section 4 of \cite{He2018}. Let $\kappa_E:= |E-b_N|\wedge |E+a_N|$ (recall that $-a_N$ and $b_N$ represent the spectral edges). In the proof of the averaged local law \eqref{eq:aver_local}, we can also derive a stronger estimate outside the spectrum: 
$$ \langle G(z)\rangle - m(z) \prec \frac{1}{N(\kappa_E+\eta)}+\frac{1}{(N\eta)^2\sqrt{\kappa_E+\eta}}, $$
uniformly in $z=E+\ii \eta$ with $|z|\le \tau^{-1}$, $\eta\ge N^{-1+\tau}$, and $N\eta \sqrt{\kappa_E+\eta}\ge N^\tau$. By combining this estimate with the local law  \eqref{eq:aver_local} when $B=I$, we can derive \eqref{eq:rigidity} using the argument in \cite{erdHos2012rigidity}.

It remains to establish the stability of the self-consistent equation \eqref{eq:gks2}. Our framework closely aligns with the general settings outlined in \cite{EKS_Forum,AEK_PTRF}, except that our model does not satisfy the flatness assumption in them (see Assumption A1 of \cite{AEK_PTRF} or Assumption E of \cite{EKS_Forum}). This flatness assumption is required by \cite{EKS_Forum,AEK_PTRF} to guarantee the stability of the self-consistent equations. Nevertheless, the absence of this assumption can be compensated for by the simple form of our self-consistent equation, which is a perturbation of the quadratic self-consistent equation for semicircle law, due to our assumption that $\|A\|=\oo(1)$. Consider the self-consistent equation $\mathbf f(z,\mathbf x)=0$, where the vector-valued function $\mathbf f:\C_+\times \mathbb C^D\to \C^D$ is defined as 
$$f_k(z,\bx):= x_k - m(z) - D\avg{(R_{\bx \cdot \mathbf E}(z)-R_{m(z)}(z))E_k} , \quad k \in \qqD.$$
Define the domain $\mathbf D(\tau):=\{z=E+\ii\eta :|z|\le \tau^{-1}, \eta \ge N^{-1+\tau}\}$. For each $z\in \mathbf D(\tau)$, introduce the set 
\begin{align*}
L(z):=\{z\}\cup \{z'\in \mathbf D(\tau): \text{Re}\, z' = \text{Re}\, z, \text{Im}\, z'\in [\text{Im}\, z, \infty)\cap (N^{-10}\mathbb N)\} .
\end{align*}
In other words, $L(z)$ is a 1-dimensional lattice with spacing $N^{-10}$ along with the point $z$. The stability estimate we aim to establish for the equation $\mathbf f(z,\mathbf x)=0$ on $\mathbf D(\tau)$ is as follows:

\begin{lemma}\label{lem_stab} 
Let $c_0>0$ be a sufficiently small constant. The self-consistent equation $\mathbf f(z,\bx)=0$ is stable on $\mathbf D(\tau)$ in the following sense. Suppose the $z$-dependent function $\delta$ satisfies $N^{-2} \le \delta(z) \le (\log N)^{-1}$ for $z\in \mathbf D(\tau)$ and is Lipschitz continuous with a Lipschitz constant $\le N^3$. Additionally, assume that for each fixed $E$, the function $\eta \mapsto \delta(E+\ii\eta)$ is non-increasing for $\eta>0$. Suppose every component  $u_k$ of a vector function $\mathbf u: \mathbf D(\tau)\to \mathbb C$ is the Stieltjes transform of a probability measure on $\R$. For any $z\in \mathbf D(\tau)$, if for every $z'\in L(z)$ the bound 
\begin{equation}\label{Stability0}
\left\|\mathbf f(z', \mathbf u(z'))\right\|_\infty \le \delta(z')
\end{equation}
holds, then we have that
\begin{equation}
\left\|\mathbf u(z)-\mathbf m(z)\right\|_\infty\le \frac{C\delta(z)}{\im m(z)+\sqrt{\delta(z)}},\label{eq:est_stab}
\end{equation}
where $C>0$ is a constant independent of both $z$ and $N$.
\end{lemma}

\begin{proof}
A delicate stability analysis of a general class of quadratic self-consistent equations has been carried out carefully in \cite{isotropic,Quadratic_vector,AEK_PTRF}. 
Following the ideas from these works, we will now outline the proof of the stability estimate \eqref{eq:est_stab} without giving all the details.

Given an arbitrary $z=E+\ii \eta\in\mathbf D(\tau)$, suppose we have proved that \eqref{eq:est_stab} holds for all $z\in L(z')$ with $\im z'>\im z$. Then, using the Lipschitz continuity of the functions $\delta(z)$, $\bu(z)$, and $\mathbf m(z)$, we obtain that 
\begin{equation}
\left\|\mathbf u(z)-\mathbf m(z)\right\|_\infty\le \frac{2C\delta(z)}{\im m(z)+\sqrt{\delta(z)}}.\label{eq:est_stab2.0}
\end{equation}
By assumption, $\mathbf u(z)$ satisfies the system of equations $ f_k(z,\mathbf u(z))=\e_k(z)$, $k\in \qq{D}$, for some quantities $\e_k(z)$ with $\max_k|\e_k|\le \delta(z)$. Under condition \eqref{eq:est_stab2.0}, we can perform a Taylor expansion of $R_{\bu \cdot \mathbf E}$ to express these equations as  
$$
v_k -D\langle  M  (\mathbf v\cdot \mathbf E) M  E_{k}\rangle-D\langle  M  (\mathbf v\cdot \mathbf E) M (\mathbf v\cdot \mathbf E) M E_{k} \rangle  = \e_k + \OO(\|\mathbf v\|_\infty^3),\quad k\in \qq{D},
$$
where we denote $\bv\equiv \bu-\mathbf m$. Recalling $\wh M$ defined in \eqref{def_ML}, we can rewrite the above equation as 
\be\label{eq:est_stab2}
v_k  - \sum_{l=1}^D \wh M_{kl}(z,z) v_l  - m(z)^3 v_k^2 =\e_k + \OO(\|\mathbf v\|_\infty^3+\|A\|^{1/2}\|\mathbf v\|_\infty^2).
\ee
In the derivation of \eqref{eq:est_stab2}, we also  utilized the fact that
\begin{align*}
D\langle  M  (\mathbf v\cdot \mathbf E) M (\mathbf v\cdot \mathbf E) M E_{k} \rangle
&= \sum_{l,s}D\langle  M E_l M E_s M E_k \rangle v_l v_s = m(z)^3 v_k^2 +\OO(\|A\|^{1/2}\|\mathbf v \|_\infty^2),
\end{align*}
where the second step stems from \eqref{eq:msc}. Let ${k}\in \qq{D}$ such that $|v_{k}|=\|\bv\|_\infty$. Then, we can write \eqref{eq:est_stab2} as
\be\label{eq:est_stab3} a(z) v_k  - m(z)^3 v_k^2 =\e_k + \OO(\|\mathbf v\|_\infty^3+\|A\|^{1/2}\|\mathbf v\|_\infty^2),\ee
where $a(z):=1 - \sum_{l=1}^D \wh M_{kl}(z,z) v_l/v_k$ satisfies that 
\be\label{eq:estaaa} |a(z)|\gtrsim \sqrt{\kappa_E+\eta}\sim  \im m(z)\ee
by using \Cref{lem_bas_M} and a similar expansion as in \eqref{wtMlead} for $\wh M_{kl}(z_1,z_2)$ with $z_1=z_2=z$. By using \eqref{eq:estaaa} and \eqref{eq:est_stab2.0}, along with an analysis of equation \eqref{eq:est_stab3} similar to the reasoning in the proof of \cite[Lemma 4.5]{isotropic}, we can conclude the estimate \eqref{eq:est_stab}. Detailed specifics are omitted here.

It remains to establish the stability estimate \eqref{eq:est_stab} for a specific $z_0\in L(z)$. Combining this initial stability estimate with the continuity argument outlined above, we can conclude the proof of \Cref{lem_stab} by induction.  
For this purpose, we choose $z_0=E+ \ii C_0$ for a sufficiently large constant $C_0>0$.  Since both $m(z)$ and the components of $\mathbf u$ are Stieltjes transforms, we have the a priori bounds 
\be\label{eq:trivialG}
|m(z_0)|_\infty \le C_0^{-1},\quad \|\mathbf u(z_0)\|_\infty \le C_0^{-1}. 
\ee
As long as $C_0$ is chosen sufficiently large, $R_{\bu \cdot \mathbf E}$ behaves well, allowing us to perform a Taylor expansion. This expansion leads us back to equation \eqref{eq:est_stab2} with $|a(z_0)|\ge 1/2$. Analyzing this equation reveals that 
$$ \|\mathbf u(z_0)-\mathbf m(z_0)\|_\infty \le C_1 \delta(z_0)\quad \text{or}\quad \|\mathbf u(z_0)-\mathbf m(z_0)\|_\infty \ge C_1^{-1}$$
for a constant $C_1>0$ independent of $C_0$. However, the latter scenario contradicts \eqref{eq:trivialG} if we choose $C_0>2C_1$. Therefore,  we must have $\|\mathbf u(z_0)-\mathbf m(z_0)\|_\infty \le C \delta(z)$, establishing the initial stability estimate at $z_0$ and concluding the proof of \Cref{lem_stab}. 
\end{proof}

\subsection{Proof of \Cref{appen1}}\label{sec:pf_appen}

The proof of \Cref{appen1} uses an induction argument based on the Cauchy-Schwarz inequality and the following classical Ward's identity, which follows from a simple algebraic calculation. 
\begin{lemma}[Ward's identity]\label{lem-Ward}
Let $\cal A$ be a Hermitian matrix. Define its resolvent as $R(z):=(\cal A-z)^{-1}$ for any $z= E+ \ii \eta\in \C_+$. Then, we have 
    \be\label{eq_Ward0}
    \begin{split}
\sum_x \overline {R_{xy'}}  R_{xy} = \frac{R_{y'y}-\overline{R_{yy'}}}{2\ii \eta},\quad
\sum_x \overline {R_{y'x}}  R_{yx} = \frac{R_{yy'}-\overline{R_{y'y}}}{2\ii \eta}.
\end{split}
\ee
As a special case, if $y=y'$, we have
\be\label{eq_Ward}
\sum_x |R_{xy}( z)|^2 =\sum_x |R_{yx}( z)|^2 = \frac{\im R_{yy}(z) }{ \eta}.
\ee
\end{lemma}

The case of $p=1$ follows directly from the anisotropic local law \eqref{eq:aniso_local}. Suppose now that $p\geq 2$ and the estimate \eqref{entprodG} holds for products of $k$ many $G_iB_i$'s when $1\le k \leq p - 1$. 

First, if $p = 2q$ is even, then we have
\begin{align*}|(G_1B_1\cdots G_pB_p)_{\mathbf{uv}}| &= |(\mathbf{u}^*G_1B_1\cdots G_q)B_q(G_{q+1}B_{q+1}\cdots G_pB_p\mathbf{v})|
\\&\leq \Big(\sum_j |(G_1B_1\cdots G_q)_{\mathbf{u}j}|^2\Big)^{1/2} \Big(\sum_j |(G_{q+1}B_{q+1}\cdots G_pB_p)_{j\mathbf{v}}|^2\Big)^{1/2}.\end{align*}
Using Ward's identity, we get 
\begin{align}
\sum_j |(G_1B_1\cdots G_q)_{\mathbf{v}j}|^2 
&= \eta^{-1}[G_1B_1\cdots G_{q-1}B_{q-1} (\im G_q)B_{q-1}^* G^*_{q-1}\cdots B_1^*G_1^*]_{\mathbf{vv}}\label{eq:bdd2q}\\
&= \frac{(G_1B_1\cdots  G_q \cdots B_1^*G_1^*)_{\mathbf{vv}}-(G_1B_1\cdots G_q^* \cdots B_1^*G_1^*)_{\mathbf{vv}}}{2\ii \eta}\prec \frac{1}{\eta^{2q-1}},\nonumber
\end{align}
where we used the induction hypothesis in the last step with the observation that the two terms in the numerator have $(2q - 1)$ many $G$'s. The term $\sum_j |(G_{q+1}B_{q+1}\cdots G_pB_p)_{j\mathbf{v}}|^2$ can be proved in the same way. This gives \eqref{entprodG} for the even $p$ case.

On the other hand, if $p = 2q + 1$ is odd, we find that
\begin{align}|(G_1B_1\cdots G_pB_p)_{\mathbf{uv}}| &\leq \Big(\sum_j |(G_1B_1\cdots G_{q+1})_{\mathbf{u}j}|^2\Big)^{1/2}\Big(\sum_j |(G_{q+2}B_{q+2}\cdots G_pB_p)_{j\mathbf{v}}|^2\Big)^{1/2}\label{eq:bdd2q2}\\
&\prec \eta^{-(2q-1)/2}\Big(\sum_j |(G_1B_1\cdots G_{q+1})_{\mathbf{u}j}|^2\Big)^{1/2},\nonumber
\end{align} 
where we have bounded $\sum_j |(G_{q+2}B_{q+2}\cdots G_pB_p)_{j\mathbf{v}}|^2$ as in \eqref{eq:bdd2q}. Now, we use the induction hypothesis and the trivial bound $\|G_i\|\le \eta^{-1}$ to get that 
\begin{align*}
\sum_j |(G_1B_1\cdots G_{q+1})_{\mathbf{u}j}|^2 &= \eta^{-1}(G_1B_1\cdots G_q B_q (\im G_{q+1})B_q^* G_q^*\cdots B_1^*G_1^*)_{\mathbf{uu}}\\
&\leq \eta^{-2}(G_1B_1\cdots G_q G_q^*\cdots B_1^*G_1^*)_{\mathbf{uu}} \prec \eta^{-(2q+1)}.
\end{align*}
Plugging it into \eqref{eq:bdd2q2}, we get \eqref{entprodG} for the odd $p$ case.

Finally,  \Cref{appen1} follows by induction in $p$.



\end{document}